\documentclass[11pt]{article}
\usepackage[english]{babel}
\usepackage[T1]{fontenc}
\usepackage[utf8]{inputenc}
\usepackage{amsbsy}
\usepackage{amsmath}
\usepackage{amssymb}
\usepackage{amsfonts}
\usepackage{amsthm}
\usepackage{bm}
\usepackage{graphicx}
\usepackage{hyperref}
\usepackage[active]{srcltx}
\usepackage{upgreek}
\usepackage{xcolor}

\newcommand{\C}{\mathbb{C}}
\newcommand{\N}{\mathbb{N}}

\newcommand{\R}{\mathbb{R}}
\renewcommand{\S}{\mathbb{S}}

\newcommand{\Z}{\mathbb{Z}}

\newcommand{\boC}{\mathcal{C}}

\newcommand{\boE}{\mathcal{E}}

\newcommand{\boH}{\mathcal{H}}
\newcommand{\boI}{\mathcal{I}}

\newcommand{\boK}{\mathcal{K}}

\newcommand{\boN}{\mathcal{N}}
\newcommand{\boO}{\mathcal{O}}

\newcommand{\boV}{\mathcal{V}}

\newcommand{\gE}{\mathfrak{E}}
\newcommand{\gI}{\mathfrak{I}}
\newcommand{\gS}{\mathfrak{S}}

\newcommand{\gj}{\mathfrak{j}}

\renewcommand{\div}{\mathop{{\rm div}}\nolimits}
\DeclareMathOperator{\diag}{{\rm diag}}

\DeclareMathOperator{\supp}{{\rm supp}}

\renewcommand{\div}{\operatorname{div}}

\newtheorem{cas}{Case}

\newtheorem{cor}{Corollary}
\newtheorem{lem}{Lemma}
\newtheorem{prop}{Proposition}
\newtheorem{step}{Step}
\newtheorem{thm}{Theorem}
\theoremstyle{definition}
\newtheorem*{merci}{Acknowledgments}
\newtheorem{rem}{Remark}
\theoremstyle{remark}

\begin{document}

\title{The Sine-Gordon regime of the Landau-Lifshitz equation with a strong easy-plane anisotropy}
\author{
\renewcommand{\thefootnote}{\arabic{footnote}}
Andr\'e de Laire\footnotemark[1]~ and Philippe Gravejat\footnotemark[2]}
\footnotetext[1]{Universit\'e de Lille, CNRS, UMR 8524 - Laboratoire Paul Painlev\'e, F-59000 Lille, France. E-mail: {\tt andre.de\-laire@univ-lille.fr}}
\footnotetext[2]{Universit\'e de Cergy-Pontoise, Laboratoire de Math\'ematiques (UMR 8088), F-95302 Cergy-Pontoise Cedex, France. E-mail: {\tt philippe.gravejat@u-cergy.fr}}
\maketitle

\begin{abstract}
It is well-known that the dynamics of biaxial ferromagnets with a strong easy-plane anisotropy is essentially governed by the Sine-Gordon equation. In this paper, we provide a rigorous justification to this observation. More precisely, we show the convergence of the solutions to the Landau-Lifshitz equation for biaxial ferromagnets towards the solutions to the Sine-Gordon equation in the regime of a strong easy-plane anisotropy. Moreover, we establish the sharpness of our convergence result.

This result holds for solutions to the Landau-Lifshitz equation in high order Sobolev spaces. We first provide an alternative proof for local well-posedness in this setting by introducing high order energy quantities with better symmetrization properties. We then derive the convergence from the consistency of the Landau-Lifshitz equation with the Sine-Gordon equation by using well-tailored energy estimates. As a by-product, we also obtain a further derivation of the free wave regime of the Landau-Lifshitz equation.

\medskip


\medskip

\end{abstract}

%%%%%%%%%%%%%%%%%%%%%%
%%%%%%%%%%%%%%%%%%%%%%
%%%%%%%%%%%%%%%%%%%%%%
\section{Introduction}
%%%%%%%%%%%%%%%%%%%%%%
%%%%%%%%%%%%%%%%%%%%%%
%%%%%%%%%%%%%%%%%%%%%%

The Landau-Lifshitz equation
\begin{equation}
\tag{LL}
\label{LL}
\partial_t m + m \times \big( \Delta m - J(m) \big) = 0,
\end{equation}
was introduced by Landau and Lifshitz~\cite{LandLif1} as a model for the magnetization $m : \R^N \times \R \to \S^2$ in a ferromagnetic material. The matrix $J := \diag(J_1, J_2, J_3)$ gives account of the anisotropy of the material (see e.g.~\cite{KosIvKo1}). The equation describes the Hamiltonian dynamics corresponding to the Landau-Lifshitz energy
$$E_{\rm LL}(m) := \frac{1}{2} \int_{\R^N} \big( |\nabla m|^2 + \lambda_1 m_1^2 + \lambda_3 m_3^2 \big).$$
The two values of the characteristic numbers $\lambda_1 := J_2 - J_1$ and $\lambda_3 := J_2 - J_3$ are non-zero for biaxial ferromagnets, while $\lambda_1$ is chosen to be equal to $0$ in the case of uniaxial ferromagnets. When $\lambda_3 < 0$, uniaxial ferromagnets own an easy-axis anisotropy along the vector $e_3 = (0, 0, 1)$, whereas the anisotropy is easy-plane along the plane $x_3 = 0$ when $\lambda_3 > 0$. The material is isotropic when $\lambda_1 = \lambda_3 = 0$, and the Landau-Lifshitz equation reduces to the well-known Schr\"odinger map equation (see e.g.~\cite{DingWan1, SulSuBa1, ChaShUh1, BeIoKeT1} and the references therein).

In the sequel, we are interested in the dynamics of biaxial ferromagnets in a regime of strong easy-plane anisotropy. The characteristic numbers $\lambda_1$ and $\lambda_3$ satisfy the inequalities $0 < \lambda_1 \ll 1 \ll \lambda_3$. More precisely, we assume that
\begin{equation}
\label{eq:lambda-sigma}
\lambda_1 := \sigma \varepsilon, \quad {\rm and} \quad \lambda_3 := \frac{1}{\varepsilon}.
\end{equation}
As usual, the parameter $\varepsilon$ is a small positive number, whereas $\sigma$ is a fixed positive constant. In this regime, the Landau-Lifshitz equation recasts as
$$\partial_t m + m \times \Big( \Delta m - \varepsilon \sigma m_1 e_1 - \frac{m_3 e_3}{\varepsilon} \Big) = 0,$$
with $e_1 := (1, 0, 0)$. In~\cite{Sklyani1}, Sklyanin observed that the solutions of this equation are governed by the Sine-Gordon equation in the limit $\varepsilon \to 0$ (see also~\cite{FaddTak0}). In the physical literature, this approximation is widely used for understanding the properties of the experimentally measurable quantities in ferromagnets (see e.g.~\cite{MikeSte1}). In order to clarify this approximation, it is useful to introduce the hydrodynamical formulation of the Landau-Lifshitz equation.

Assume that the map $\check{m} := m_1 + i m_2$ corresponding to a solution $m$ to~\eqref{LL} does not vanish. In this case, it can be written as
$$\check{m} = (1 - m_3^2)^\frac{1}{2} \big( \sin(\phi) + i\cos(\phi) \big).$$
The introduction of the phase function $\phi$ is reminiscent from the use of the Madelung transform~\cite{Madelun1} in the context of nonlinear Schr\"odinger equations (see e.g.~\cite{deLaGra1} for more details). This transform leads to a hydrodynamical version of the Landau-Lifshitz equation in terms of the variables $u := m_3$ and $\phi$, which is given by the system
\begin{equation}
\tag{{\rm HLL}}
\label{HLL}
\begin{cases} \partial_t u = \div \big( (1 - u^2) \nabla \phi \big) - \frac{\lambda_1}{2} (1 - u^2) \sin(2 \phi),\\[5pt]
\partial_t \phi = - \div \Big( \frac{\nabla u}{1 - u^2} \Big) + u \frac{|\nabla u|^2}{(1 - u^2)^2} - u |\nabla \phi|^2 + u \Big( \lambda_3 - \lambda_1 \sin^2(\phi) \Big). \end{cases}
\end{equation}

Under the scaling in~\eqref{eq:lambda-sigma}, this hydrodynamical system is related to the Sine-Gordon equation in the long-wave regime corresponding to the rescaled variables $(U_\varepsilon, \Phi_\varepsilon)$ given by the identities
$$u(x, t) = \varepsilon U_\varepsilon \big( \sqrt{\varepsilon} x, t \big), \quad {\rm and} \quad \phi(x, t) = \Phi_\varepsilon(\sqrt{\varepsilon} x, t).$$
The pair $(U_\varepsilon, \Phi_\varepsilon)$ indeed satisfies
\begin{equation}
\tag{${\rm HLL}_\varepsilon$}
\label{HLLeps}
\begin{cases} \partial_t U_\varepsilon = \div \big( (1 - \varepsilon^2 U_\varepsilon^2) \nabla \Phi_\varepsilon \big) - \frac{\sigma}{2} (1 - \varepsilon^2 U_\varepsilon^2) \sin(2 \Phi_\varepsilon),\\[5pt]
\partial_t \Phi_\varepsilon = U_\varepsilon \big( 1 - \varepsilon^2 \sigma \sin^2(\Phi_\varepsilon) \big) - \varepsilon^2 \div \Big( \frac{\nabla U_\varepsilon}{1 - \varepsilon^2 U_\varepsilon^2} \Big) +\varepsilon^4 U_\varepsilon \frac{|\nabla U_\varepsilon|^2}{(1 - \varepsilon^2 U_\varepsilon^2)^2} - \varepsilon^2 U_\varepsilon |\nabla \Phi_\varepsilon|^2. \end{cases}
\end{equation}
As $\varepsilon \to 0$, the limit system is formally given by 
\begin{equation}
\label{sys:SG}
\tag{SGS}
\begin{cases} \partial_t U = \Delta \Phi - \frac{\sigma}{2} \sin(2 \Phi),\\[5pt]
\partial_t \Phi = U. \end{cases}
\end{equation}
Therefore, the limit function $\Phi$ is a solution to the Sine-Gordon equation
\begin{equation}
\tag{SG}
\label{SG}
\partial_{tt} \Phi - \Delta \Phi + \frac{\sigma}{2} \sin(2 \Phi) = 0.
\end{equation}

Our main goal in the sequel is to provide a rigorous justification for this Sine-Gordon regime of the Landau-Lifshitz equation.

%%%%%%%%%%%%%%%%%%%%%%%%%
%%%%%%%%%%%%%%%%%%%%%%%%%
\subsection{Main results}
\label{sub:main}
%%%%%%%%%%%%%%%%%%%%%%%%%
%%%%%%%%%%%%%%%%%%%%%%%%%

In order to analyze rigorously this regime, we introduce a functional setting in which we can legitimate the use of the hydrodynamical framework. This condition is at least checked when the inequality $|m_3| < 1$ holds on $\R^N$. In terms of the hydrodynamical pair $(u, \varphi)$, this writes as
\begin{equation}
\label{eq:cond-u}
|u| < 1 \quad {\rm on} \ \R^N.
\end{equation}
Under this condition, it is natural to work in the Hamiltonian framework in which the solutions $m$ have finite Landau-Lifshitz energy. In the hydrodynamical formulation, the Landau-Lifshitz energy is given by
\begin{equation}
\label{eq:E-HLL}
E_{\rm LL}(u, \varphi) := \frac{1}{2} \int_{\R^N} \Big( \frac{|\nabla u|^2}{1 - u^2} + (1 - u^2) |\nabla \varphi|^2 + \lambda_1 (1 - u^2) \sin^2(\varphi) + \lambda_3 u^2 \Big).
\end{equation}
As a consequence of this formula, it is natural to work with the non-vanishing set
$$\boN\boV(\R^N) := \big\{ (u, \varphi) \in H^1(\R^N) \times H_{\sin}^1(\R^N) : |u| < 1 \ {\rm on} \ \R^N \big\}.$$
In this definition, we have set
$$H_{\sin}^1(\R^N) := \big\{ v \in L_{\rm loc}^1(\R^N) : \nabla v \in L^2(\R^N) \ {\rm and} \ \sin(v) \in L^2(\R^N) \big\}.$$
The set $H_{\sin}^1(\R^N)$ is an additive group. It is naturally endowed with the pseudometric distance
$$d_{\sin}^1(v_1, v_2) := \Big( \big\| \sin(v_1 - v_2) \big\|_{L^2}^2 + \big\| \nabla v_1 - \nabla v_2 \big\|_{L^2}^2 \Big)^\frac{1}{2},$$
which vanishes if and only if $v_1 - v_2 \in \pi \Z$. This quantity is not a distance on the group $H_{\sin}^1(\R^N)$, but it is on the quotient group $H_{\sin}^1(\R^N)/\pi \Z$. In the sequel, we identify the set $H_{\sin}^1(\R^N)$ with this quotient group when necessary, in particular when a metric structure is required. This identification is not a difficulty as far as we deal with the hydrodynamical form of the Landau-Lifshitz equation and with the Sine-Gordon equation. Both the equations are indeed left invariant by adding a constant number in $\pi \Z$ to the phase functions $\phi$, respectively, $\Phi$. This property is one of the motivations for introducing the pseudometric distance $d_{\sin}^1$. We refer to Appendix~\ref{sec:Hsink} for more details concerning this distance, as well as the set $H_{\sin}^1(\R^N)$.

Our derivation of the Sine-Gordon equation also requires to control the non-vanishing condition in~\eqref{eq:cond-u} along the flow of the Landau-Lifshitz equation. In dimension one, it follows from the Sobolev embedding theorem that the function $u$ is uniformly controlled in the non-vanishing set $\boN\boV(\R)$. This property does not remain in higher dimensions. In the sequel, we by-pass this difficulty by restricting our analysis to solutions $(u, \varphi)$ with additional regularity. There might be other ways to handle this problem. Requiring additional smoothness is also useful for our rigorous derivation of the Sine-Gordon regime.

Given an integer $k \geq 1$, we set 
\begin{equation}
\label{def:NVk}
\boN\boV^k(\R^N) := \big\{ (u, \varphi) \in H^k(\R^N) \times H_{\sin}^k(\R^N) : |u| < 1 \ {\rm on} \ \R^N \big\}.
\end{equation}
Here, the additive group $H_{\sin}^k(\R^N)$ is defined as
$$H_{\sin}^k(\R^N) := \big\{ v \in L_{\rm loc}^1(\R^N) : \nabla v \in H^{k - 1}(\R^N) \ {\rm and} \ \sin(v) \in L^2(\R^N) \big\}.$$
As before, we identify this group, when necessary, with the quotient group $H_{\sin}^k(\R^N)/\pi \Z$, and then we endow it with the distance
\begin{equation}
\label{def:dsink}
d_{\sin}^k(v_1, v_2) := \Big( \big\| \sin(v_1 - v_2) \big\|_{L^2}^2 + \big\| \nabla v_1 - \nabla v_2 \big\|_{H^{k - 1}}^2 \Big)^\frac{1}{2}.
\end{equation}
With this notation at hand, the vanishing set $\boN\boV(\R^N)$ identifies with $\boN\boV^1(\R^N)$.

We are now in position to state our main result.

\begin{thm}
\label{thm:conv-SG}
Let $N \geq 1$ and $k \in \N$, with $k > N/2 + 1$, and $0 < \varepsilon < 1$. Consider an initial condition $(U_\varepsilon^0, \Phi_\varepsilon^0) \in \boN\boV^{k + 2}(\R^N)$ and set
$$\boK_\varepsilon^0 := \big\| U_\varepsilon^0 \big\|_{H^k} + \varepsilon \big\| \nabla U_\varepsilon^0 \big\|_{H^k} + \big\| \nabla \Phi_\varepsilon^0 \big\|_{H^k} + \big\| \sin(\Phi_\varepsilon^0) \big\|_{H^k}.$$
Consider similarly an initial condition $(U^0, \Phi^0) \in L^2(\R^N) \times H_{\sin}^1(\R^N)$, and denote by $(U, \Phi) \in \boC^0(\R, L^2(\R^N) \times H_{\sin}^1(\R^N))$ the unique corresponding solution to~\eqref{sys:SG}. Then, there exists a positive number $C_*$, depending only on $\sigma$, $k$ and $N$, such that, if the initial data satisfy the condition
\begin{equation}
\label{cond:key}
C_* \, \varepsilon \, \boK_\varepsilon^0 \leq 1,
\end{equation}
we have the following statements.

$(i)$ There exists a positive number
\begin{equation}
\label{cond:Teps}
T_\varepsilon \geq \frac{1}{C_* (\boK_\varepsilon^0)^2},
\end{equation}
such that there exists a unique solution $(U_\varepsilon, \Phi_\varepsilon) \in \boC^0([0, T_\varepsilon], \boN\boV^{k + 1}(\R^N))$ to~\eqref{HLLeps} with initial datum $(U_\varepsilon^0, \Phi_\varepsilon^0)$.

$(ii)$ If $\Phi_\varepsilon^0 - \Phi^0 \in L^2(\R^N)$, then we have
\begin{equation}
\label{est:0}
\begin{split}
\big\| \Phi_\varepsilon(\cdot, t) - \Phi(\cdot, t) \big\|_{L^2} \leq C_* \, \Big( \big\| \Phi_\varepsilon^0 - \Phi^0 \big\|_{L^2} + \big\| U_\varepsilon^0 - U^0 \big\|_{L^2} + \varepsilon^2 \, \boK_\varepsilon^0 \, \big( 1 + \boK_\varepsilon^0 \big)^3 \Big) \, e^{C_* t},
\end{split}
\end{equation}
for any $0 \leq t \leq T_\varepsilon$.

$(iii)$ If $N \geq 2$, or $N = 1$ and $k > N/2 + 2$, then we have
\begin{equation}
\label{est:1}
\begin{split}
& \big\| U_\varepsilon(\cdot, t) - U(\cdot, t) \big\|_{L^2} + \big\| \nabla \Phi_\varepsilon(\cdot, t) - \nabla \Phi(\cdot, t) \big\|_{L^2} + \big\| \sin(\Phi_\varepsilon(\cdot, t) - \Phi(\cdot, t)) \big\|_{L^2}\\
\leq C_* \, \Big( & \big\| U_\varepsilon^0 - U^0 \big\|_{L^2} + \big\| \nabla \Phi_\varepsilon^0 - \nabla \Phi^0 \big\|_{L^2} + \big\| \sin(\Phi_\varepsilon^0 - \Phi^0) \big\|_{L^2} + \varepsilon^2 \, \boK_\varepsilon^0 \, \big( 1 + \boK_\varepsilon^0 \big)^3 \Big) \, e^{C_* t},
\end{split}
\end{equation}
for any $0 \leq t \leq T_\varepsilon$.

$(iv)$ Take $(U^0, \Phi^0) \in H^k(\R^N) \times H_{\sin}^{k + 1}(\R^N)$ and set
$$\kappa_\varepsilon^0 := \boK_\varepsilon^0 + \big\| U^0 \big\|_{H^k} + \big\| \nabla \Phi^0 \big\|_{H^k} + \big\| \sin(\Phi^0) \big\|_{H^k}.$$
There exists a positive number $A_*$, depending only on $\sigma$, $k$ and $N$, such that the solution $(U, \Phi)$ lies in $\boC^0([0, T_\varepsilon^*], H^k(\R^N) \times H_{\sin}^{k + 1}(\R^N))$ for a positive number
\begin{equation}
\label{cond:Teps*}
T_\varepsilon \geq T_\varepsilon^* \geq \frac{1}{A_* (\kappa_\varepsilon^0)^2}.
\end{equation}
Moreover, when $k > N/2 + 3$, we have
\begin{equation}
\label{est:k}
\begin{split}
& \big\| U_\varepsilon(\cdot, t) - U(\cdot, t) \big\|_{H^{k - 3}} + \big\| \nabla \Phi_\varepsilon(\cdot, t) - \nabla \Phi(\cdot, t) \big\|_{H^{k - 3}} + \big\| \sin(\Phi_\varepsilon(\cdot, t) - \Phi(\cdot, t)) \big\|_{H^{k - 3}}\\
\leq A_* \, & e^{A_* (1 + \kappa_\varepsilon^0)^2 t} \times\\
\times \Big( & \big\| U_\varepsilon^0 - U^0 \big\|_{H^{k - 3}} + \big\| \nabla \Phi_\varepsilon^0 - \nabla \Phi^0 \big\|_{H^{k - 3}} + \big\| \sin(\Phi_\varepsilon^0 - \Phi^0) \big\|_{H^{k - 3}} + \varepsilon^2 \kappa_\varepsilon^0 \big( 1 + \kappa_\varepsilon^0 \big)^3 \Big),
\end{split}
\end{equation}
for any $0 \leq t \leq T_\varepsilon^*$.
\end{thm}

In arbitrary dimension, Theorem~\ref{thm:conv-SG} provides a quantified convergence of the Landau-Lifshitz equation towards the Sine-Gordon equation in the regime of strong easy-plane anisotropy. Three types of convergence are proved depending on the dimension, and the levels of regularity of the solutions. This trichotomy is related to the analysis of the Cauchy problems for the Landau-Lifshitz and Sine-Gordon equations.

In its natural Hamiltonian framework, the Sine-Gordon equation is globally well-posed. Its Hamiltonian is the Sine-Gordon energy
\begin{equation}
\label{def:E-SG}
E_{\rm SG}(\phi) := \frac{1}{2} \int_{\R^N} \big( (\partial_t \phi)^2 + |\nabla \phi|^2 + \sigma \sin(\phi)^2 \big).
\end{equation}
Given an initial condition $(\Phi^0, \Phi^1) \in H_{\sin}^1(\R^N) \times L^2(\R^N)$, there exists a unique corresponding solution $\Phi \in \boC^0(\R, H_{\sin}^1(\R^N))$ to~\eqref{SG}, with $\partial_t \Phi \in \boC^0(\R, L^2(\R^N))$. Moreover, the Sine-Gordon equation is locally well-posed in the spaces $H_{\sin}^k(\R^N) \times H^{k - 1}(\R^N)$, when $k > N/2 + 1$. In other words, the solution $\Phi$ remains in $\boC^0([0, T], H_{\sin}^k(\R^N))$, with $\partial_t \Phi \in \boC^0([0, T], H^{k - 1}(\R^N))$, at least locally in time, when $(\Phi^0, \Phi^1) \in H_{\sin}^k(\R^N) \times H^{k - 1}(\R^N)$. We refer to Subsection~\ref{sub:Cauchy-SG} below for more details regarding these two results and their proofs.

In contrast, the Cauchy problem for the Landau-Lifshitz equation at its Hamiltonian level is far from being completely understood. Global weak and strong solutions are known to exist (see e.g.~\cite{GuoDing0, BeIoKeT1} and the references therein), but blow-up can occur (see~\cite{MerRaRo1}).

On the other hand, the Landau-Lifshitz equation is locally well-posed at the same level of high regularity as the Sine-Gordon equation. In the hydrodynamical context, this reads as the existence of a maximal time $T_{\max}$ and a unique solution $(U, \Phi) \in \boC^0([0, T_{\max}), \boN\boV^{k - 1}(\R^N))$ to~\eqref{HLL} corresponding to an initial condition $(U^0, \Phi^0) \in \boN\boV^k(\R^N)$, when $k > N/2 + 1$ (see Corollary~\ref{cor:HLL-Cauchy} in Subsection~\ref{sub:Cauchy-LL}). Note the loss of one derivative here. This loss explains why we take initial conditions $(U_\varepsilon^0, \Phi_\varepsilon^0)$ in $\boN\boV^{k + 2}(\R^N)$, though the quantity $\boK_\varepsilon^0$ is already well-defined when $(U_\varepsilon^0, \Phi_\varepsilon^0) \in \boN\boV^{k + 1}(\R^N)$.

In view of this local well-posedness result, we restrict our analysis of the Sine-Gordon regime to the solutions $(U_\varepsilon, \Phi_\varepsilon)$ to the rescaled system~\eqref{HLLeps} with sufficient regularity. A further difficulty then lies in the fact that their maximal times of existence possibly depend on the small parameter $\varepsilon$.

Statement $(i)$ in Theorem~\ref{thm:conv-SG} provides an explicit control on these maximal times. In view of~\eqref{cond:Teps}, these maximal times are bounded from below by a positive number depending only on the choice of the initial data $(U_\varepsilon^0, \Phi_\varepsilon^0)$. Note that, in case a family of initial data $(U_\varepsilon^0, \Phi_\varepsilon^0)$ converges towards a pair $(U^0, \Phi^0)$ in $H^k(\R^N) \times H_{\sin}^k(\R^N)$ as $\varepsilon \to 0$, it is possible to find a positive number $T$ such that all the corresponding solutions $(U_\varepsilon, \Phi_\varepsilon)$ are well-defined on $[0, T]$. This property is necessary in order to make possible a consistent analysis of the limit $\varepsilon \to 0$.

Statement $(i)$ only holds when the initial data $(U_\varepsilon^0, \Phi_\varepsilon^0)$ satisfy the condition in~\eqref{cond:key}. However, this condition is not a restriction in the limit $\varepsilon \to 0$. It is satisfied by any fixed pair $(U^0, \Phi^0) \in \boN\boV^{k + 1}(\R^N)$ provided that $\varepsilon$ is small enough, so that it is also satisfied by a family of initial data $(U_\varepsilon^0, \Phi_\varepsilon^0)$, which converges towards a pair $(U^0, \Phi^0)$ in $H^k(\R^N) \times H_{\sin}^k(\R^N)$ as $\varepsilon \to 0$.

Statements $(ii)$ and $(iii)$ in Theorem~\ref{thm:conv-SG} provide two estimates~\eqref{est:0} and~\eqref{est:1} between the previous solutions $(U_\varepsilon, \Phi_\varepsilon)$ to~\eqref{HLLeps}, and an arbitrary global solution $(U, \Phi)$ to~\eqref{sys:SG} at the Hamiltonian level. The first one yields an $L^2$-control on the difference $\Phi_\varepsilon - \Phi$, the second one, an energetic control on the difference $(U_\varepsilon, \Phi_\varepsilon) - (U, \Phi)$. Due to the fact that the difference $\Phi_\varepsilon - \Phi$ is not necessarily in $L^2(\R^N)$, statement $(ii)$ is restricted to initial conditions such that this property is satisfied.

Finally, statement $(iv)$ bounds the difference between the solutions $(U_\varepsilon, \Phi_\varepsilon)$ and $(U, \Phi)$ at the same initial Sobolev level. In this case, we also have to control the maximal time of regularity of the solutions $(U, \Phi)$. This follows from the control from below in~\eqref{cond:Teps*}, which is of the same order as the one in~\eqref{cond:Teps}.

We then obtain the Sobolev estimate in~\eqref{est:k} of the difference $(U_\varepsilon, \Phi_\varepsilon) - (U, \Phi)$ with a loss of three derivatives. Here, the choice of the Sobolev exponents $k > N/2 + 3$ is tailored so as to gain a uniform control on the functions $U_\varepsilon - U$, $\nabla \Phi_\varepsilon - \nabla \Phi$ and $\sin(\Phi_\varepsilon - \Phi)$ by the Sobolev embedding theorem.

A loss of derivatives is natural in the context of long-wave regimes (see e.g.~\cite{BeGrSaS2, BeGrSaS3} and the references therein). It is related to the terms with first and second-order derivatives in the right-hand side of~\eqref{HLLeps}. This loss is the reason why the energetic estimate in statement $(iii)$ requires an extra derivative in dimension one, that is the condition $k > N/2 + 2$. Using the Sobolev bounds~\eqref{borne-W-bis} in Corollary~\ref{cor:T} below, we can (partly) recover this loss by a standard interpolation argument, and deduce an estimate in $H^\ell(\R^N) \times H_{\sin}^{\ell + 1}(\R^N)$ for any number $\ell < k$. In this case, the error terms are no more of order $\varepsilon^2$ as in the right-hand sides of~\eqref{est:0},~\eqref{est:1} and~\eqref{est:k}. Our presentation of the convergence results in Theorem~\ref{thm:conv-SG} is motivated by the fact that a control of order $\varepsilon^2$ is sharp.

As a matter of fact, the system~\eqref{sys:SG} owns explicit travelling-wave solutions. Up to a suitable scaling for which $\sigma = 1$, and up to the geometric invariance by translation, they are given by the kink and anti-kink functions
\begin{equation}
\label{soliton-SG}
u_c^\pm(x, t) = \pm \frac{c}{\sqrt{1 - c^2} \cosh \big( \frac{x - c t}{\sqrt{1 - c^2}} \big)}, \quad {\rm and} \quad \phi_c^\pm(x, t) = 2 \arctan \Big( e^{\mp \frac{x - c t}{\sqrt{1 - c^2}}} \Big),
\end{equation}
for any speed $c \in (- 1, 1)$. The hydrodynamical Landau-Lifshitz system~\eqref{HLLeps} similarly owns explicit travelling-wave solutions $(U_{c, \varepsilon}, \Phi_{c, \varepsilon})$ with speed $c$, for which their exists a positive number $A$, depending only on $c$, such that
$$\| U_{c, \varepsilon} - u_c^+ \|_{L^2} + \| \nabla \Phi_{c, \varepsilon} - \nabla \phi_c^+ \|_{L^2} + \| \sin(\Phi_{c, \varepsilon} - \phi_c^+) \big\|_{L^2} \underset{\varepsilon \to 0}{\sim} A \varepsilon^2.$$
Hence, the estimate by $\varepsilon^2$ in~\eqref{est:0},~\eqref{est:1} and~\eqref{est:k} is indeed optimal. We refer to Appendix~\ref{sec:solitons} for more details about this topic, and more generally, about the travelling-wave solutions to the Landau-Lifshitz equation.

As a by-product of our analysis, we can also analyze the wave regime for the Landau-Lifshitz equation. This regime is obtained when the parameter $\sigma$ is allowed to vary so as to converge to $0$. At least formally, a solution $(U_{\varepsilon, \sigma}, \Phi_{\varepsilon, \sigma})$ to~\eqref{HLLeps} indeed satisfies the free wave system
\begin{equation}
\tag{FW}
\label{FW}
\begin{cases}
\partial_t U = \Delta \Phi,\\
\partial_t \Phi = U,
\end{cases}
\end{equation}
when $\varepsilon \to 0$ and $\sigma \to 0$. In particular, the function $\Phi$ is solution to the free wave equation
$$\partial_{tt} \Phi - \Delta \Phi = 0.$$
The following result provides a rigorous justification for this asymptotic approximation.

\begin{thm}
\label{thm:conv-wave}
Let $N \geq 1$ and $k \in \N$, with $k > N/2 + 1$, and $0 < \varepsilon, \sigma < 1$. Consider an initial condition $(U_{\varepsilon, \sigma}^0, \Phi_{\varepsilon, \sigma}^0) \in \boN\boV^{k + 2}(\R^N)$
and set
$$\boK_{\varepsilon, \sigma}^0 := \big\| U_{\varepsilon, \sigma}^0 \big\|_{H^k} + \varepsilon \big\| \nabla U_{\varepsilon, \sigma}^0 \big\|_{H^k} + \big\| \nabla \Phi_{\varepsilon, \sigma}^0 \big\|_{H^k} + \sigma^\frac{1}{2} \big\| \sin(\Phi_{\varepsilon, \sigma}^0) \big\|_{L^2}.$$
Let $m \in \N$, with $0 \leq m \leq k - 2$. Consider similarly an initial condition $(U^0, \Phi^0) \in H^m(\R^N) \times H^{m - 1}(\R^N)$, and denote by $(U, \Phi) \in \boC^0(\R, H^{m - 1}(\R^N) \times H^m(\R^N))$ the unique corresponding solution to~\eqref{FW}. Then, there exists a positive number $C_*$, depending only on $k$ and $N$, such that, if the initial datum satisfies the condition
\begin{equation}
\label{cond:small}
C_* \, \varepsilon \, \boK_{\varepsilon, \sigma}^0 \leq 1,
\end{equation}
the following statements hold true.

$(i)$ There exists a positive number
\begin{equation}
\label{cond:Teps-sigma}
T_{\varepsilon, \sigma} \geq \frac{1}{C_* \max \{ \varepsilon, \sigma \} (1 + \boK_{\varepsilon, \sigma}^0)^{\max \{ 2, \frac{k}{2} \} }},
\end{equation}
such that there exists a unique solution $(U_{\varepsilon, \sigma}, \Phi_{\varepsilon, \sigma}) \in \boC^0([0, T_{\varepsilon, \sigma}], \boN\boV^{k + 1}(\R^N))$ to~\eqref{HLLeps} with initial datum $(U_{\varepsilon, \sigma}^0, \Phi_{\varepsilon, \sigma}^0)$.

$(ii)$ If $\Phi_{\varepsilon, \sigma}^0 - \Phi^0 \in H^m(\R^N)$, then we have the estimate
\begin{equation}
\label{est:FW1}
\begin{split}
\big\| U_{\varepsilon, \sigma}(\cdot, t) & - U(\cdot, t) \big\|_{H^{m - 1}} + \big\| \Phi_{\varepsilon, \sigma}(\cdot, t) - \Phi(\cdot, t) \big\|_{H^m} \leq C_* \big( 1 + t^2 \big) \, \Big( \big\| U_{\varepsilon, \sigma}^0 - U^0 \big\|_{H^{m - 1}}\\
& + \big\| \Phi_{\varepsilon, \sigma}^0 - \Phi^0 \big\|_{H^m} + \max \big\{ \varepsilon^2,\sigma^\frac{1}{2} \big\} \, \boK_{\varepsilon, \sigma}^0 \, \big( 1 + \boK_{\varepsilon, \sigma}^0 \big)^{\max \{ 2, m \}} \Big),
\end{split}
\end{equation}
for any $0 \leq t \leq T_{\varepsilon, \sigma}$. In addition, we also have
\begin{equation}
\label{est:FW2}
\begin{split}
\big\| U_{\varepsilon, \sigma}(\cdot, t) & - U(\cdot, t) \big\|_{\dot{H}^{\ell - 1}} + \big\| \Phi_{\varepsilon, \sigma}(\cdot, t) - \Phi(\cdot, t) \big\|_{\dot{H}^\ell} \leq C_* \big( 1 + t \big) \, \Big( \big\| U_{\varepsilon, \sigma}^0 - U^0 \big\|_{\dot{H}^{\ell - 1}}\\
& + \big\| \Phi_{\varepsilon, \sigma}^0 - \Phi^0 \big\|_{\dot{H}^\ell} + \max \big\{ \varepsilon^2 , \sigma \big\} \, \boK_{\varepsilon, \sigma}^0 \big( 1 + \boK_{\varepsilon, \sigma}^0 \big)^{\max \{ 2, \ell \}} \Big),
\end{split}
\end{equation}
for any $1 \leq \ell \leq m$ and any $0 \leq t \leq T_{\varepsilon, \sigma}$.
\end{thm}

The wave regime of the Landau-Lifshitz equation was first derived rigorously by Shatah and Zeng~\cite{ShatZen1}, as a special case of the wave regimes for the Schr\"odinger map equations with values into arbitrary K\"ahler manifolds. The derivation in~\cite{ShatZen1} relies on energy estimates, which are similar in spirit to the ones we establish in the sequel, and a compactness argument. Getting rid of this compactness argument provides the quantified version of the convergence in Theorem~\ref{thm:conv-wave}. This improvement is based on the arguments developed by B\'ethuel, Danchin and Smets~\cite{BetDaSm1} in order to quantify the convergence of the Gross-Pitaevskii equation towards the free wave equation in a similar long-wave regime. Similar arguments were also applied in~\cite{Chiron9} in order to derive rigorously the (modified) Korteweg-de Vries and (modified) Kadomtsev-Petviashvili regimes of the Landau-Lifshitz equation (see also~\cite{GermRou1}).

In the remaining part of this introduction, we detail the main ingredients in the proof of Theorem~\ref{thm:conv-SG}. We first clarify the analysis of the Cauchy problems for the Sine-Gordon and Landau-Lifshitz equations.

%%%%%%%%%%%%%%%%%%%%%%%%%%%%%%%%%%%%%%%%%%%%%%%%%%%%%%%%%%%%
%%%%%%%%%%%%%%%%%%%%%%%%%%%%%%%%%%%%%%%%%%%%%%%%%%%%%%%%%%%%
\subsection{The Cauchy problem for the Sine-Gordon equation}
\label{sub:Cauchy-SG}
%%%%%%%%%%%%%%%%%%%%%%%%%%%%%%%%%%%%%%%%%%%%%%%%%%%%%%%%%%%%
%%%%%%%%%%%%%%%%%%%%%%%%%%%%%%%%%%%%%%%%%%%%%%%%%%%%%%%%%%%%

The Sine-Gordon equation is a semilinear wave equation with a Lipschitz nonlinearity. The well-posedness analysis of the corresponding Cauchy problem is classical (see e.g~\cite[Chapter 6]{ShatStr0} and~\cite[Chapter 12]{Evans0}). With the proof of Theorem~\ref{thm:conv-SG} in mind, we now provide some precisions about this analysis in the context of the product sets $H_{\sin}^k(\R^N) \times H^{k - 1}(\R^N)$.

In the Hamiltonian framework, it is natural to solve the equation for initial conditions $\phi(\cdot, 0) = \phi^0 \in H_{\sin}^1(\R^N)$ and $\partial_t \phi(\cdot, 0) = \phi^1 \in L^2(\R^N)$, which guarantees the finiteness of the Sine-Gordon energy in~\eqref{def:E-SG}. Note that we do not assume that the function $\phi^0$ lies in $L^2(\R^N)$. This is motivated by formula~\eqref{soliton-SG} for the one-dimensional solitons $\phi_c^\pm$, which lie in $H_{\sin}^1(\R)$, but not in $L^2(\R)$. In this Hamiltonian setting, the Cauchy problem for~\eqref{SG} is globally well-posed.

\begin{thm}
\label{thm:SG-Cauchy}
Let $\sigma \in \R^*$. Given two functions $(\phi^0, \phi^1) \in H_{\sin}^1(\R^N) \times L^2(\R^N)$, there exists a unique solution $\phi \in \boC^0(\R, \phi^0 + H^1(\R^N))$, with $\partial_t \phi \in \boC^0(\R, L^2(\R^N))$, to the Sine-Gordon equation with initial conditions $(\phi^0, \phi^1)$. Moreover, this solution satisfies the following statements.

$(i)$ For any positive number $T$, there exists a positive number $A$, depending only on $\sigma$ and $T$, such that the flow map $(\phi^0, \phi^1) \mapsto (\phi, \partial_t \phi)$ satisfies
$$d_{\sin}^1 \big( \phi(\cdot, t), \tilde{\phi}(\cdot, t) \big) + \big\| \partial_t \phi(\cdot, t) - \partial_t \tilde{\phi}(\cdot, t) \big\|_{L^2} \leq A \Big( d_{\sin}^1 \big( \phi^0, \tilde{\phi}^0 \big) + \big\| \phi^1 - \tilde{\phi}^1 \big\|_{L^2} \Big),$$
for any $t \in [- T, T]$. Here, the function $\tilde{\phi}$ is the unique solution to the Sine-Gordon equation with initial conditions $(\tilde{\phi}^0, \tilde{\phi}^1)$.

$(ii)$ When $\phi^0 \in H_{\rm sin}^2(\R^N)$ and $\phi^1 \in H^1(\R^N)$, the solution $\phi$ belongs to the space $\boC^0(\R, \phi^0 + H^2(\R^N))$, with $\partial_t \phi \in \boC^0(\R, H^1(\R^N))$ and $\partial_{tt} \phi \in \boC^0(\R, L^2(\R^N))$.

$(iii)$ The Sine-Gordon energy $E_{\rm SG}$ is conserved along the flow.
\end{thm}

The proof of Theorem~\ref{thm:SG-Cauchy} relies on a classical fixed-point argument. The only difficulty consists in working in the unusual functional setting provided by the set $H_{\sin}^1(\R^N)$. This difficulty is by-passed by applying the strategy developed by Buckingham and Miller in~\cite[Appendix B]{BuckMil1} (see also~\cite{Gallo1} for similar arguments in the context of the Gross-Pitaevskii equation). In dimension $N = 1$, they fix a function $f \in \boC^\infty(\R)$, with (possibly different) limits $\ell^{\pm} \pi$ at $\pm \infty$, and with a derivative $f'$ in the Schwartz class. Given a real number $p \geq 1$, they consider an initial datum $(\phi^0 = f + \varphi^0, \phi^1)$, with $(\varphi^0, \phi^1) \in L^p(\R)^2$, and they apply a fixed-point argument in order to construct the unique corresponding solution $\phi = f + \varphi$ to the Sine-Gordon equation, with $\varphi \in L^\infty([0, T], L^p(\R))$ for some positive number $T$. This solution is global when $\phi^0$ lies in $W^{1, p}(\R)$. In view 
of Lemmas~\ref{lem:decompose} and~\ref{lem:carac} below, this result includes all the functions $\phi^0$ in the space $H_{\sin}^1(\R)$ for $p = 2$. 

Our proof of Theorem~\ref{thm:SG-Cauchy} extends this strategy to arbitrary dimensions. We fix a smooth function $f \in H_{\sin}^\infty(\R^N) := \cap_{k \geq 1} H_{\sin}^k(\R^N)$, and we apply a fixed-point argument in order to solve the Cauchy problem for initial conditions $\phi^0 \in f + H^1(\R^N)$ and $\phi^1 \in L^2(\R^N)$. We finally check the local Lipschitz continuity in $H_{\sin}^1(\R^N) \times L^2(\R^N)$ of the corresponding flow.

With the proof of Theorem~\ref{thm:conv-SG} in mind, we also extend this analysis to the initial conditions $\phi^0 \in H_{\sin}^k(\R^N)$ and $\phi^1 \in H^{k - 1}(\R^N)$, with $k \in \N^*$. When the integer $k$ is large enough, we obtain the following local well-posedness result.

\begin{thm}
\label{thm:SG-Cauchy-smooth}
Let $\sigma \in \R^*$ and $k \in \N$, with $k > N/2 + 1$. Given two functions $(\phi^0, \phi^1) \in H_{\sin}^k(\R^N) \times H^{k - 1}(\R^N)$, there exist a positive number $T_{\max}^k$, and a unique solution $\phi \in \boC^0([0, T_{\max}^k), \linebreak[0] \phi^0 + H^k(\R^N))$, with $\partial_t \phi \in \boC^0([0, T_{\max}^k), H^{k - 1}(\R^N))$, to the Sine-Gordon equation with initial conditions $(\phi^0, \phi^1)$. Moreover, this solution satisfies the following statements.

$(i)$ The maximal time of existence $T_{\max}^k$ is characterized by the condition
$$\lim_{t \to T_{\max}^k} d_{\sin}^k \big( \phi(\cdot, t), 0 \big) = \infty \quad {\rm if} \ T_{\max}^k < \infty.$$

$(ii)$ Let $0 \leq T < T_{\max}^k$. There exist two positive numbers $R$ and $A$, depending only on $T$, $d_{\sin}^k(\phi^0, 0)$ and $\| \phi^1 \|_{H^{k - 1}}$, such that the flow map $(\phi^0, \phi^1) \mapsto (\phi, \partial_t \phi)$ is well-defined from the ball
$$B \big( (\phi^0, \phi^1), R) = \big\{ (\tilde{\phi}^0, \tilde{\phi}^1) \in H_{\sin}^k(\R^N) \times H^{k - 1}(\R^N) : d_{\sin}^k(\phi^0, \tilde{\phi}^0) + \| \phi^1 - \tilde{\phi}^1 \|_{H^{k - 1}} < R \big\},$$
to $\boC^0([0, T], H_{\sin}^k(\R^N)) \times H^{k - 1}(\R^N))$, and satisfies
$$d_{\sin}^k \big( \phi(\cdot, t), \tilde{\phi}(\cdot, t) \big) + \big\| \partial_t \phi(\cdot, t) - \partial_t \tilde{\phi}(\cdot, t) \big\|_{H^{k - 1}} \leq A \Big( d_{\sin}^k \big( \phi^0, \tilde{\phi}^0 \big) + \big\| \phi^1 - \tilde{\phi}^1 \big\|_{H^{k - 1}} \Big),$$
for any $t \in [0, T]$. Here, the function $\tilde{\phi}$ is the unique solution to the Sine-Gordon equation with initial conditions $(\tilde{\phi}^0, \tilde{\phi}^1)$.

$(iii)$ When $\phi^0 \in H_{\rm sin}^{k + 1}(\R^N)$ and $\phi^1 \in H^k(\R^N)$, the function $\phi$ is in $\boC^0([0, T_{\max}^k), \phi^0 + H^{k + 1}(\R^N))$, with $\partial_t \phi \in \boC^0([0, T_{\max}^k), H^k(\R^N))$ and $\partial_{tt} \phi \in \boC^0([0, T_{\max}^k), H^{k - 1}(\R^N))$. In particular, the maximal time of existence $T_{\max}^{k + 1}$ satisfies
$$T_{\max}^{k + 1} = T_{\max}^k.$$

$(iv)$ When $1 \leq N \leq 3$, the solution $\phi$ is global in time. Moreover, when $N \in \{ 2, 3 \}$, the flow remains continuous for $k = 2$.
\end{thm}

Theorem~\ref{thm:SG-Cauchy-smooth} follows from a fixed-point argument similar to the one of Theorem~\ref{thm:SG-Cauchy}. The control on the nonlinear terms is derived from a uniform bound on the gradient of the solutions. This is the origin of the condition $k > N/2 + 1$ for which the Sobolev embedding theorem guarantees a uniform control on the gradient. This condition is natural in the context of the spaces $H_{\sin}^k(\R^N)$. Indeed, a function $f \in H_{\sin}^k(\R^N)$ is not controlled uniformly (see Remark~\ref{rem:carac1}, and the discussion in Appendix~\ref{sec:Hsink}). At least in principle, the classical condition $k > N/2$ is not sufficient to handle the nonlinear terms of the Sine-Gordon equation.

The maximal time of existence $T_{\max}^k$ in statement $(i)$ can be estimated by performing standard energy estimates. When $1 \leq N \leq 3$, this leads to the global well-posedness of the Sine-Gordon equation in the space $H_{\sin}^k(\R^N) \times H^{k - 1}(\R^N)$ for $k > N/2 + 1$. Actually, it is possible to extend this global well-posedness result to dimensions $4 \leq N \leq 9$. This extension relies on the Strichartz estimates for the free wave equation (see e.g.~\cite{KeelTao1}), and the use of fractional Sobolev spaces. A blow-up in finite time is possible when $N \geq 10$. For the sake of simplicity, and since this is not our main goal, we do not address this question any further. We refer to~\cite{Tao2} for a detailed discussion on this topic, and for the construction of blowing-up solutions to related semilinear wave systems.

When $1 \leq N \leq 3$, the fixed-point arguments in the proofs of Theorems~\ref{thm:SG-Cauchy} and~\ref{thm:SG-Cauchy-smooth} provide the continuity of the flow with values in $\boC^0([0, T], H_{\sin}^k(\R^N) \times H^{k - 1}(\R^N))$ for any positive number $T$, except if $k = 2$ and $2 \leq N \leq 3$. We fill this gap by performing standard energy estimates. We conclude that the Sine-Gordon equation is globally well-posed in the spaces $H_{\sin}^k(\R^N) \times H^k(\R^N)$ for any $1 \leq N \leq 3$ and any $k \geq 1$.

Note finally that the previous well-posedness analysis of the Sine-Gordon equation translates immediately into the Sine-Gordon system~\eqref{sys:SG} by setting $u = \partial_t \phi$.

%%%%%%%%%%%%%%%%%%%%%%%%%%%%%%%%%%%%%%%%%%%%%%%%%%%%%%%%%%%%%%%%
%%%%%%%%%%%%%%%%%%%%%%%%%%%%%%%%%%%%%%%%%%%%%%%%%%%%%%%%%%%%%%%%
\subsection{The Cauchy problem for the Landau-Lifshitz equation}
\label{sub:Cauchy-LL}
%%%%%%%%%%%%%%%%%%%%%%%%%%%%%%%%%%%%%%%%%%%%%%%%%%%%%%%%%%%%%%%%
%%%%%%%%%%%%%%%%%%%%%%%%%%%%%%%%%%%%%%%%%%%%%%%%%%%%%%%%%%%%%%%%

The Landau-Lifshitz equation is an anisotropic perturbation of the Schr\"odinger map equation. Solving the Cauchy problem for this further equation is known to be intrinsically involved due to the geometric nature of the equation (see e.g.~\cite{BeIoKeT1}). The situation is similar for the Landau-Lifshitz equation, but one has also to handle the anisotropy of the equation.

The natural functional setting for solving the Landau-Lifshitz equation is given by the energy set
$$\boE(\R^N) := \big\{ v \in L_{\rm loc}^1(\R^N, \S^2) : \nabla v \in L^2(\R^N) \ {\rm and} \ (v_1, v_3) \in L^2(\R^N)^2 \big\}.$$
We endow this set with the metric structure provided by the norm
$$\| v \|_{Z^1} := \big( \| v_1 \|_{H^1}^2 + \| v_2 \|_{L^\infty}^2 + \| \nabla v_2 \|_{L^2}^2 + \| v_3 \|_{H^1}^2\big)^\frac{1}{2},$$
of the vector space
$$Z^1(\R^N) := \big\{ v \in L_{\rm loc}^1(\R^N, \R^3) : \nabla v \in L^2(\R^N), v_2 \in L^\infty(\R^N) \ {\rm and} \ (v_1, v_3) \in L^2(\R^N)^2 \big\}.$$
With this structure at hand, the Landau-Lifshitz energy is well-defined and continuous. The uniform control on the second component $v_2$ in the $Z^1$-norm is not necessary to guarantee these properties, but it is in order to ensure the map $\| \cdot \|_{Z^1}$ to be a norm. This uniform control is not the only possible choice. Our choice is motivated by the boundedness of the functions $v$ in $\boE(\R^N)$.

To our knowledge, establishing the global well-posedness of the Landau-Lifshitz equation for general initial data in the energy set $\boE(\R^N)$ remains an open question. We do not address this question any further in the sequel. However, it is possible to construct global weak solutions by adapting the construction by Sulem, Sulem and Bardos~\cite{SulSuBa1} in the case of the Schr\"odinger map equation. Due to the requirement of additional smoothness in order to handle the Sine-Gordon regime, we instead focus on the local well-posedness of smooth solutions. 

Given an integer $k \geq 1$, we introduce the set 
$$\boE^k(\R^N) := \big\{ v \in \boE(\R^N) : \nabla v \in H^{k - 1}(\R^N) \big\},$$
which we endow with the metric structure provided by the norm
$$\| v \|_{Z^k} := \Big( \| v_1 \|_{H^k}^2 + \| v_2 \|_{L^\infty}^2 + \| \nabla v_2 \|_{H^{k - 1}}^2 + \| v_3 \|_{H^k}^2 \big)^\frac{1}{2},$$
of the vector space
\begin{equation}
\label{def:Zk}
Z^k(\R^N) := \big\{ v \in L_{\rm loc}^1(\R^N, \R^3) : (v_1, v_3) \in L^2(\R^N)^2, v_2 \in L^\infty(\R^N) \ {\rm and} \ \nabla v \in H^{k - 1}(\R^N) \big\}.
\end{equation}
Observe that the energy space $\boE(\R^N)$ identifies with $\boE^1(\R^N)$.

When $k$ is large enough, we show the local well-posedness of the Landau-Lifshitz equation in the set $\boE^k(\R^N)$.

\begin{thm}
\label{thm:LL-Cauchy}
Let $\lambda_1$ and $\lambda_3$ be non-negative numbers, and $k \in \N$, with $k > N/2 + 1$. Given any function $m^0 \in \boE^k(\R^N)$, there exists a positive number $T_{\max}$ and a unique solution $m : \R^N \times[0, T_{\max}) \to \S^2$ to the Landau-Lifshitz equation with initial datum $m^0$, which satisfies the following statements.

$(i)$ The solution $m$ is in the space $L^\infty([0, T], \boE^k(\R^N))$, while its time derivative $\partial_t m$ is in $L^\infty([0, T], H^{k - 2}(\R^N))$, for any number $0 < T < T_{\max}$.

$(ii)$ If the maximal time of existence $T_{\max}$ is finite, then
\begin{equation}
\label{eq:cond-Tmax-LL}
\int_0^{T_{\max}} \| \nabla m(\cdot, t) \|_{L^\infty}^2 \, dt = \infty.
\end{equation}

$(iii)$ The flow map $m^0 \mapsto m$ is well-defined and locally Lipschitz continuous from $\boE^k(\R^N)$ to $\boC^0([0, T], \boE^{k - 1}(\R^N))$ for any number $0 < T < T_{\max}$.

$(iv)$ When $m^0 \in \boE^\ell(\R^N)$, with $\ell > k$, the solution $m$ lies in $L^\infty([0, T], \boE^\ell(\R^N))$, with $\partial_t m \in L^\infty([0, T], H^{\ell - 2}(\R^N))$ for any number $0 < T < T_{\max}$.

$(v)$ The Landau-Lifshitz energy is conserved along the flow.
\end{thm}

Theorem~\ref{thm:LL-Cauchy} provides the existence and uniqueness of a local continuous flow corresponding to smooth solutions of the Landau-Lifshitz equation. This kind of statement is standard in the context of hyperbolic systems (see e.g.~\cite[Theorem 1.2]{Taylor03}). The critical regularity for the equation is given by the condition $k = N/2$, so that local well-posedness is expected when $k > N/2 + 1$. As in the proof of Theorem~\ref{thm:SG-Cauchy-smooth}, this assumption is used to control uniformly the gradient of the solutions by the Sobolev embedding theorem.

In the isotropic case of the Schr\"odinger map equation, local well-posedness at the same level of regularity was established in~\cite{ChaShUh1} when $N = 1$ by using the Hasimoto transform, and in~\cite{McGahag1} in arbitrary dimensions by using parallel transport (see also~\cite{ZhouGuo1, SulSuBa1, DingWan1} for the construction of smooth solutions). Our proof of Theorem~\ref{thm:LL-Cauchy} is based on a more direct strategy. We introduce new quantities, which improve classical energy estimates. This approach applies to the isotropic situation of the Schr\"odinger map equation, as well as to the anisotropic setting of the Landau-Lifshitz equation.

In contrast with the proof of Theorem~\ref{thm:SG-Cauchy-smooth}, we do not rely on a fixed-point argument, but on a compactness argument. Due to this difference, the solutions $m$ corresponding to initial data $m^0 \in \boE^k(\R^N)$ are not necessarily continuous with values in $\boE^k(\R^N)$, but they remain bounded with values in this set. Continuity is recovered with a loss of one derivative, that is in $\boE^{k - 1}(\R^N)$, and the flow map is then locally Lipschitz continuous. By standard interpolation, the solutions are actually continuous with values in the fractional sets
$$\boE^s(\R^N) := \big\{ v \in \boE(\R^N) : \nabla v \in H^{s - 1}(\R^N) \big\},$$
as soon as $1 \leq s < k$.

More precisely, the construction of the solution $m$ in Theorem~\ref{thm:LL-Cauchy} is based on the strategy developed by Sulem, Sulem and Bardos~\cite{SulSuBa1} in the context of the Schr\"odinger map equation. The first step is to compute a priori energy estimates. Given a fixed positive number $T$ and a smooth solution $m : \R^N \times [0, T] \to \S^2$ to the Landau-Lifshitz equation, we define an energy of order $k \geq 2$ as
\begin{equation}
\label{def:E-LL-k}
\begin{split}
E_{\rm LL}^k(t) := \frac{1}{2} \sum_{|\alpha| = k - 2} \int_{\R^N} \Big( |\partial_t \partial_x^\alpha m|^2 + |\Delta \partial_x^\alpha m|^2 & + (\lambda_1 + \lambda_3) \big( |\nabla \partial_x^\alpha m_1|^2 + |\nabla \partial_x^\alpha m_3|^2 \big)\\
& + \lambda_1 \lambda_3 \big( |\partial_x^\alpha m_1|^2 + |\partial_x^\alpha m_3|^2 \big) \Big)(x, t) \, dx,
\end{split}
\end{equation}
for any $t \in [0, T]$. Here as in the sequel, we set $\partial_x^\alpha := \partial_{x_1}^{\alpha_1} \ldots \partial_{x_N}^{\alpha_N}$ for any $\alpha \in \N^N$. We can differentiate this quantity so as to obtain the following energy estimates.

\begin{prop}
\label{prop:LL-energy-estimate}
Let $\lambda_1$ and $\lambda_3$ be fixed non-negative numbers, and $k \in \N$, with $k > 1 + N/2$. Assume that $m$ is a solution to~\eqref{LL}, which lies in $\boC^0([0, T], \boE^{k + 2}(\R^N))$, with $\partial_t m \in \boC^0([0, T], H^k(\R^N))$.

$(i)$ The Landau-Lifshitz energy is well-defined and conserved along flow, that is
$$E_{\rm LL}^1(t) := E_{\rm LL} \big( m(\cdot, t) \big) = E_{\rm LL}^1(0),$$
for any $t \in [0, T]$.

$(ii)$ Given any integer $2 \leq \ell \leq k$, the energies $E_{\rm LL}^\ell$ are of class $\boC^1$ on $[0, T]$, and there exists a positive number $C_k$, depending only on $k$, such that their derivatives satisfy
\begin{equation}
\label{eq:energy-estimate-LL}
\big[ E_{\rm LL}^\ell \big]'(t) \leq C_k \big( 1 + \| m_1(\cdot, t) \|_{L^\infty}^2 + \| m_3(\cdot, t) \|_{L^\infty}^2 + \| \nabla m(\cdot, t) \|_{L^\infty}^2 \big) \, \Sigma_{\rm LL}^\ell(t),
\end{equation}
for any $t \in [0, T]$. Here, we have set $\Sigma_{\rm LL}^\ell := \sum_{j = 1}^\ell E_{\rm LL}^j$.
\end{prop}

We next discretize the equation by using a finite-difference scheme. The a priori bounds remain available in this discretized setting. We then apply standard weak compactness and local strong compactness results in order to construct local weak solutions, which satisfy statement $(i)$ in Theorem~\ref{thm:LL-Cauchy}. Applying the Gronwall lemma to the inequalities in~\eqref{eq:energy-estimate-LL} prevents a possible blow-up when the condition in~\eqref{eq:cond-Tmax-LL} is not satisfied.

Finally, we establish uniqueness, as well as continuity with respect to the initial datum, by computing energy estimates for the difference of two solutions. More precisely, we show

\begin{prop}
\label{prop:LL-diff-control}
Let $\lambda_1$ and $\lambda_3$ be non-negative numbers, and $k \in \N$, with $k > N/2 + 1$. Consider two solutions $m$ and $\tilde{m}$ to~\eqref{LL}, which lie in $\boC^0([0, T], \boE^{k + 1}(\R^N))$, with $(\partial_t m, \partial_t \tilde{m}) \in \boC^0([0, T], H^{k - 1}(\R^N))^2$, and set $u := \tilde{m} - m$ and $v := (\tilde{m} + m)/2$.

$(i)$ The function
\begin{equation}
\label{def:gE0}
\gE_{\rm LL}^0(t) := \frac{1}{2} \int_{\R^N} \big| u(x, t) - u_2^0(x) \, e_2 \big|^2 \, dx,
\end{equation}
is of class $\boC^1$ on $[0, T]$, and there exists a positive number $C$ such that
\begin{equation}
\label{eq:control-diff-0-LL}
\begin{split}
\big[ & \gE_{\rm LL}^0 \big]'(t) \leq C \big( 1 + \| \nabla \tilde{m}(\cdot, t) \|_{L^2} + \| \nabla m(\cdot, t) \|_{L^2} + \| \tilde{m}_1(\cdot, t) \|_{L^2} + \| m_1(\cdot, t) \|_{L^2}\\
& + \| \tilde{m}_3(\cdot, t) \|_{L^2} + \| m_3(\cdot, t) \|_{L^2} \big) \, \big( \| u(\cdot, t) - u_2^0 \, e_2 \|_{L^2}^2 + \| u(\cdot, t) \|_{L^\infty}^2 + \| \nabla u(\cdot, t) \|_{L^2}^2 + \| \nabla u_2^0 \|_{L^2}^2 \big),
\end{split}
\end{equation}
for any $t \in [0, T]$.

$(ii)$ The function
$$\gE_{\rm LL}^1(t) := \frac{1}{2} \int_{\R^N} \big( |\nabla u|^2 + |u \times \nabla v + v \times \nabla u|^2 \big)(x, t) \, dx,$$
is of class $\boC^1$ on $[0, T]$, and there exists a positive number $C$ such that
\begin{equation}
\label{eq:control-diff-1-LL}
\begin{split}
\big[ \gE_{\rm LL}^1 \big]'(t) \leq C & \big( 1 + \| \nabla m(\cdot, t) \|_{L^\infty}^2 + \| \nabla \tilde{m}(\cdot, t) \|_{L^\infty}^2 \big) \, \big( \| u (\cdot, t) \|_{L^\infty}^2 + \| \nabla u (\cdot, t) \|_{L^2}^2 \big) \times\\
& \times \big( 1 + \| \nabla m(\cdot, t) \|_{L^\infty} + \| \nabla \tilde{m}(\cdot, t) \|_{L^\infty} + \| \nabla m(\cdot, t) \|_{H^1} + \| \nabla \tilde{m}(\cdot, t) \|_{H^1} \big).
\end{split}
\end{equation}

$(iii)$ Let $2 \leq \ell \leq k - 1$. The function
\begin{align*}
\gE_{\rm LL}^\ell(t) := \frac{1}{2} \sum_{|\alpha| = \ell - 2} \int_{\R^N} \Big( |\partial_t \partial_x^\alpha u|^2 + |\Delta \partial_x^\alpha u|^2 + (\lambda_1 & + \lambda_3) \big( |\nabla \partial_x^\alpha u_1|^2 + |\nabla \partial_x^\alpha u_3|^2 \big)\\
& + \lambda_1 \lambda_3 \big( |\partial_x^\alpha u_1|^2 + |\partial_x^\alpha u_3|^2 \big) \Big)(x, t) \, dx,
\end{align*}
is of class $\boC^1$ on $[0, T]$, and there exists a positive number $C_k$, depending only on $k$, such that
\begin{equation}
\label{eq:control-diff-l-LL}
\begin{split}
\big[ & \gE_{\rm LL}^\ell \big]'(t) \leq C_k \Big( 1 + \| \nabla m(\cdot, t) \|_{H^\ell}^2 + \| \nabla \tilde{m}(\cdot, t) \|_{H^\ell}^2 + \| \nabla m(\cdot, t) \|_{L^\infty}^2 + \| \nabla \tilde{m}(\cdot, t) \|_{L^\infty}^2\\
& + \delta_{\ell = 2} \big( \| \tilde{m}_1(\cdot, t) \|_{L^2} + \| m_1(\cdot, t) \|_{L^2} + \| \tilde{m}_3(\cdot, t) \|_{L^2} + \| m_3(\cdot, t) \|_{L^2} \big) \Big) \, \big( \gS_{\rm LL}^\ell(t) + \| u (\cdot, t) \|_{L^\infty}^2 \big).
\end{split}
\end{equation}
Here, we have set $\gS_{\rm LL}^\ell := \sum_{j = 0}^\ell \gE_{\rm LL}^j$.
\end{prop}

When $\ell \geq 2$, the quantities $\gE_{\rm LL}^\ell$ in Proposition~\ref{prop:LL-diff-control} are anisotropic versions of the ones used in~\cite{SulSuBa1} for similar purposes. Their explicit form is related to the linear part of the second-order equation in~\eqref{eq:second-LL}. The quantity $\gE_{\rm LL}^0$ is tailored to close off the estimates.

The introduction of the quantity $\gE_{\rm LL}^1$ is of a different nature. The functions $\nabla u$ and $u \times \nabla v + v \times \nabla u$ in its definition appear as the good variables for performing hyperbolic estimates at an $H^1$-level. They provide a better symmetrization corresponding to a further cancellation of the higher order terms. Without any use of the Hasimoto transform, or of parallel transport, this makes possible a direct proof of local well-posedness at an $H^k$-level, with $k > N/2 + 1$ instead of $k > N/2 + 2$. We refer to the proof of Proposition~\ref{prop:LL-diff-control} in Subsection~\ref{sub:LL-diff-control} for more details.

With the proof of Theorem~\ref{thm:conv-SG} in mind, we now translate the analysis of the Cauchy problem for the Landau-Lifshitz equation into the hydrodynamical framework. We obtain the following local well-posedness result, which makes possible the analysis of the Sine-Gordon regime in Theorem~\ref{thm:conv-SG}.

\begin{cor}
\label{cor:HLL-Cauchy}
Let $\lambda_1$ and $\lambda_3$ be non-negative numbers, and $k \in \N$, with $k > N/2 + 1$. Given any pair $(u^0, \phi^0) \in \boN\boV^k(\R^N)$, there exists a positive number $T_{\max}$ and a unique solution $(u, \phi) : \R^N \times[0, T_{\max}) \to (- 1, 1) \times \R$ to~\eqref{HLL} with initial datum $(u^0, \phi^0)$, which satisfies the following statements.

$(i)$ The solution $(u, \phi)$ is in the space $L^\infty([0, T], \boN\boV^k(\R^N))$, while its time derivative $(\partial_t u, \partial_t \phi)$ is in $L^\infty([0, T], H^{k - 2}(\R^N)^2)$, for any number $0 < T < T_{\max}$.

$(ii)$ If the maximal time of existence $T_{\max}$ is finite, then
$$\int_0^{T_{\max}} \Big( \Big\| \frac{\nabla u(\cdot, t)}{(1 - u(\cdot, t)^2)^\frac{1}{2}} \Big\|_{L^\infty}^2 + \Big\| (1 - u(\cdot, t)^2)^\frac{1}{2} \nabla \phi(\cdot, t) \Big\|_{L^\infty}^2 \Big) \, dt = \infty, \quad {\rm or} \quad \lim_{t \to T_{\max}} \| u(\cdot, t) \|_{L^\infty} = 1.$$

$(iii)$ The flow map $(u^0, \phi^0) \mapsto (u, \phi)$ is well-defined, and locally Lipschitz continuous from $\boN\boV^k(\R^N)$ to $\boC^0([0, T], \boN\boV^{k - 1}(\R^N))$ for any number $0 < T < T_{\max}$.

$(iv)$ When $(u^0, \phi^0) \in \boN\boV^\ell(\R^N)$, with $\ell > k$, the solution $(u, \phi)$ lies in $L^\infty([0, T], \boN\boV^\ell(\R^N))$, with $(\partial_t u, \partial_t \phi) \in L^\infty([0, T], H^{\ell - 2}(\R^N)^2)$ for any number $0 < T < T_{\max}$.

$(v)$ The Landau-Lifshitz energy in~\eqref{eq:E-HLL} is conserved along the flow.
\end{cor}

\begin{rem}
\label{rem:bounded-Hsink}
Here as in the sequel, the set $L^\infty([0, T], H_{\sin}^k(\R^N))$ is defined as
$$L^\infty \big( [0, T], H_{\sin}^k(\R^N) \big) := \Big\{ v \in L_{\rm loc}^1(\R^N \times [0, T], \R) : \sup_{0 \leq t \leq T} \| \sin(v(\cdot, t)) \|_{L^2} + \| \nabla v(\cdot, t) \|_{H^{k - 1}} < \infty \Big\},$$
for any integer $k \geq 1$ and any positive number $T$. This definition is consistent with the fact that a family $(v(\cdot, t))_{0 \leq t \leq T}$ of functions in $H_{\sin}^k(\R^N)$ (identified with the quotient group $H_{\sin}^k(\R^N)/\pi \Z$) is then bounded with respect to the distance $d_{\sin}^k$. In particular, the set $L^\infty([0, T], \boN\boV^k(\R^N))$ is given by
\begin{align*}
L^\infty \big( [0, T], \boN\boV^k(\R^N) \big) := \Big\{ & (u, \phi) \in L_{\rm loc}^1(\R^N \times [0, T], \R^2) : |u| < 1 \ {\rm on} \ \R^N \times [0, T]\\
& {\rm and} \ \sup_{0 \leq t \leq T} \| u(\cdot, t) \|_{H^{k - 1}} + \| \sin(\phi(\cdot, t)) \|_{L^2} + \| \nabla \phi(\cdot, t) \|_{H^{k - 1}} < \infty \Big\}.
\end{align*}
\end{rem}

The proof of Corollary~\ref{cor:HLL-Cauchy} is complicated by the metric structure corresponding to the set $H_{\sin}^k(\R^N)$. Establishing the continuity of the flow map with respect to the pseudometric distance $d_{\sin}^k$ is not so immediate. We by-pass this difficulty by using some simple trigonometric identities. We refer to Subsection~\ref{sub:Cauchy-smooth-HLL} below for more details.

Another difficulty lies in controlling the non-vanishing condition in~\eqref{eq:cond-u}. Due to the Sobolev embedding theorem, this can be done at an $H^k$-level, with $k > N/2 + 1$. However, this does not prevent a possible break-up of this condition in finite time. Statement $(ii)$ exactly expresses this simple fact. In the hydrodynamical setting, blow-up can originate from either blow-up in the original setting, or the break-up of the non-vanishing condition.

%%%%%%%%%%%%%%%%%%%%%%%%%%%%%%%%%%%%%%%%%%%%%%%%%%%%%%%%%%%%%
%%%%%%%%%%%%%%%%%%%%%%%%%%%%%%%%%%%%%%%%%%%%%%%%%%%%%%%%%%%%%
\subsection{Sketch of the proof of Theorem~\ref{thm:conv-SG}}
\label{sub:deriv-SG}
%%%%%%%%%%%%%%%%%%%%%%%%%%%%%%%%%%%%%%%%%%%%%%%%%%%%%%%%%%%%%
%%%%%%%%%%%%%%%%%%%%%%%%%%%%%%%%%%%%%%%%%%%%%%%%%%%%%%%%%%%%%

When $(U_\varepsilon^0, \Phi_\varepsilon^0)$ lies in $\boN\boV^{k + 2}(\R^N)$, we deduce from Corollary~\ref{cor:HLL-Cauchy} above the existence of a positive number $T_{\max}$, and a unique solution $(U_\varepsilon, \Phi_\varepsilon) \in \boC^0([0, T_{\max}), \boN\boV^{k + 1}(\R^N))$ to~\eqref{HLLeps} with initial datum $(U_\varepsilon^0, \Phi_\varepsilon^0)$. The maximal time of existence $T_{\max}$ a priori depends on the scaling parameter $\varepsilon$. The number $T_{\max}$ might become smaller and smaller in the limit $\varepsilon \to 0$, so that analyzing this limit would have no sense.

As a consequence, our first task in the proof of Theorem~\ref{thm:conv-SG} is to provide a control on $T_{\max}$. In view of the conditions in statement $(ii)$ of Corollary~\ref{cor:HLL-Cauchy}, this control can be derived from uniform bounds on the functions $U_\varepsilon$, $\nabla U_\varepsilon$ and $\nabla \Phi_\varepsilon$. Taking into account the Sobolev embedding theorem and the fact that $k > N/2 + 1$, we are left with the computations of energy estimates for the functions $U_\varepsilon$ and $\Phi_\varepsilon$, at least in the spaces $H^k(\R^N)$, respectively $H_{\sin}^k(\R^N)$. 

In this direction, we recall that the Landau-Lifshitz energy corresponding to the scaled hydrodynamical system~\eqref{HLLeps} writes as
$$E_\varepsilon(U_\varepsilon, \Phi_\varepsilon) = \frac{1}{2} \int_{\R^N} \Big( \varepsilon^2 \frac{|\nabla U_\varepsilon|^2}{1 - \varepsilon^2 U_\varepsilon^2} + U_\varepsilon^2 + (1 - \varepsilon^2 U_\varepsilon^2) |\nabla \Phi_\varepsilon|^2 + \sigma (1 - \varepsilon^2 U_\varepsilon^2) \sin^2(\Phi_\varepsilon) \Big).$$
Hence, it is natural to define an energy of order $k \in \N^*$ according to the formula
\begin{equation}
\label{E-k}
\begin{split}
E_\varepsilon^k(U_\varepsilon, \Phi_\varepsilon) := \frac{1}{2} \sum_{|\alpha| = k - 1} \int_{\R^N} \Big( & \varepsilon^2 \frac{|\nabla \partial_x^\alpha U_\varepsilon|^2}{1 - \varepsilon^2 U_\varepsilon^2} + |\partial_x^\alpha U_\varepsilon|^2 + (1 - \varepsilon^2 U_\varepsilon^2) |\nabla \partial_x^\alpha \Phi_\varepsilon|^2\\
& + \sigma (1 - \varepsilon^2 U_\varepsilon^2) |\partial_x^\alpha \sin(\Phi_\varepsilon)|^2 \Big).
\end{split}
\end{equation}
The factors $1 - \varepsilon^2 U_\varepsilon^2$ in this expression, as well as the non-quadratic term corresponding to the function $\sin(\Phi_\varepsilon)$, are of substantial importance. As for the energy $E_\varepsilon(U_\varepsilon, \Phi_\varepsilon)$, they provide a better symmetrization of the energy estimates corresponding to the quantities $E_\varepsilon^k(U_\varepsilon, \Phi_\varepsilon)$ by inducing cancellations in the higher order terms. More precisely, we have

\begin{prop}
\label{prop:estimate}
Let $\varepsilon$ be a fixed positive number, and $k \in \N$, with $k > N/2 + 1$. Consider a solution $(U_\varepsilon,\Phi_\varepsilon)$ to~\eqref{HLLeps}, with $(U_\varepsilon, \Phi_\varepsilon) \in \boC^0([0, T], \boN\boV^{k + 3}(\R^N))$ for a fixed positive number $T$, and assume that 
\begin{equation}
\label{borne-W}
\inf_{\R^N \times [0, T]} 1 - \varepsilon^2 U_\varepsilon^2 \geq \frac{1}{2}. 
\end{equation}
There exists a positive number $C$, depending only on $k$ and $N$, such that
\begin{equation}
\label{der-E-j}
\begin{split}
\big[ E_\varepsilon^\ell \big]'(t) \leq C \, \max \big\{ 1, \sigma^\frac{3}{2} \big\} \, \big( 1 & + \varepsilon^4 \big) \, \Big( \| \sin(\Phi_\varepsilon(\cdot, t)) \|_{L^\infty}^2 + \| U_\varepsilon(\cdot, t) \|_{L^\infty}^2 + \| \nabla \Phi_\varepsilon(\cdot, t) \|_{L^\infty}^2\\
& + \| \nabla U_\varepsilon(\cdot, t) \|_{L^\infty}^2 + \| d^2 \Phi_\varepsilon(\cdot, t) \|_{L^\infty}^2 + \varepsilon^2 \| d^2 U_\varepsilon(\cdot, t) \|_{L^\infty}^2\\
& + \varepsilon \, \| \nabla \Phi_\varepsilon(\cdot, t) \|_{L^\infty} \, \big( \| \nabla \Phi_\varepsilon(\cdot, t) \|_{L^\infty}^2 + \| \nabla U_\varepsilon(\cdot, t) \|_{L^\infty}^2 \big) \Big) \, \Sigma_\varepsilon^{k + 1}(t),
\end{split}
\end{equation}
for any $t \in [0, T]$ and any $2 \leq \ell \leq k + 1$. Here, we have set $\Sigma_\varepsilon^{k + 1} := \sum_{j = 1}^{k + 1} E_\varepsilon^j$.
\end{prop}

As a first consequence of Proposition~\ref{prop:estimate}, the maximal time $T_{\max}$ is at least of order $1/(\| U_\varepsilon^0 \|_{H^k} + \varepsilon \| \nabla U_\varepsilon^0 \|_{H^k} + \| \nabla \Phi_\varepsilon^0 \|_{H^k} + \| \sin(\Phi_\varepsilon^0) \|_{H^k})^2$, when the initial conditions $(U_\varepsilon^0, \Phi_\varepsilon^0)$ satisfy the inequality in~\eqref{cond:key}. In particular, the dependence of $T_{\max}$ on the small parameter $\varepsilon$ only results from the possible dependence of the pair $(U_\varepsilon^0, \Phi_\varepsilon^0)$ on $\varepsilon$. Choosing suitably these initial conditions, we can assume without loss of generality that $T_{\max}$ is uniformly bounded from below when $\varepsilon$ tends to $0$, so that analyzing this limit makes sense. More precisely, we deduce from Proposition~\ref{prop:estimate} the following results.

\begin{cor}
\label{cor:T}
Let $\varepsilon$ be a fixed positive number, and $k \in \N$, with $k > N/2 + 1$. There exists a positive number $C_*$, depending only on $\sigma$, $k$ and $N$, such that if an initial datum $(U_\varepsilon^0, \Phi_\varepsilon^0) \in \boN\boV^{k + 2}(\R^N)$ satisfies 
\begin{equation}
\label{CI-petit}
C_* \varepsilon \Big( \big\| U_\varepsilon^0 \big\|_{H^k} + \varepsilon \big\| \nabla U_\varepsilon^0 \big\|_{H^k} + \big\| \nabla \Phi_\varepsilon^0 \big\|_{H^k} + \big\| \sin(\Phi_\varepsilon^0) \big\|_{H^k} \Big) \leq 1, 
\end{equation}
then there exists a positive time
$$T_\varepsilon \geq \frac{1}{C_* \big( \| U_\varepsilon^0 \|_{H^k} + \varepsilon \| \nabla U_\varepsilon^0 \|_{H^k} + \| \nabla \Phi_\varepsilon^0 \|_{H^k} + \| \sin(\Phi_\varepsilon^0) \|_{H^k} \big)^2},$$
such that the unique solution $(U_\varepsilon, \Phi_\varepsilon)$ to~\eqref{HLLeps} with initial condition $(U_\varepsilon^0, \Phi_\varepsilon^0)$ satisfies the uniform bound
$$\varepsilon \big\| U_\varepsilon(\cdot, t) \big\|_{L^\infty} \leq \frac{1}{\sqrt{2}},$$
as well as the energy estimate
\begin{equation}
\label{borne-W-bis}
\begin{split}
\big\| U_\varepsilon(\cdot, t) \big\|_{H^k} + \varepsilon \big\| \nabla U_\varepsilon(\cdot, t) \big\|_{H^k} + & \big\| \nabla \Phi_\varepsilon(\cdot, t) \big\|_{H^k} + \big\| \sin(\Phi_\varepsilon(\cdot, t)) \big\|_{H^k}\\
\leq C_* \Big( & \big\| U_\varepsilon^0 \big\|_{H^k} + \varepsilon \big\| \nabla U_\varepsilon^0 \big\|_{H^k} + \big\| \nabla \Phi_\varepsilon^0 \big\|_{H^k} + \big\| \sin(\Phi_\varepsilon^0) \big\|_{H^k} \Big),
\end{split}
\end{equation}
for any $0 \leq t \leq T_\varepsilon$.
\end{cor}

\begin{rem}
In the one-dimensional case, the conservation of the energy provides a much direct control on the quantity $\varepsilon \| U_\varepsilon \|_{L^\infty}$. This claim follows from the inequality
$$\varepsilon^2 \| U_\varepsilon \|_{L^\infty}^2 \leq 2 \varepsilon^2 \int_{\R} |U_\varepsilon'(x)| \, |U_\varepsilon(x)| \, dx \leq \varepsilon \int_{\R} \big( \varepsilon^2 U_\varepsilon'(x)^2 + U_\varepsilon(x)^2 \big) \, dx.$$
When $\varepsilon \| U_\varepsilon^0 \|_{L^\infty} < 1$, and the quantity $\varepsilon E_\varepsilon(0)$ is small enough, combining this inequality with the conservation of the energy $E_\varepsilon$ and performing a continuity argument give a uniform control on the function $\varepsilon U_\varepsilon$ for any possible time.
\end{rem}

As a further consequence of Proposition~\ref{prop:estimate}, Corollary~\ref{cor:T} also provides the Sobolev control in~\eqref{borne-W-bis} on the solution $(U_\varepsilon, \Phi_\varepsilon)$, which is uniform with respect to $\varepsilon$. This control is crucial in the proof of Theorem~\ref{thm:conv-SG}. As a matter of fact, the key ingredient in this proof is the consistency of~\eqref{HLLeps} with the Sine-Gordon system in the limit $\varepsilon \to 0$. Indeed, we can rewrite~\eqref{HLLeps} as
\begin{equation}
\label{HLLeps-ter}
\begin{cases} \partial_t U_\varepsilon = \Delta \Phi_\varepsilon - \frac{\sigma}{2} \sin(2 \Phi_\varepsilon) + \varepsilon^2 R_\varepsilon^U,\\[5pt]
\partial_t \Phi_\varepsilon = U_\varepsilon + \varepsilon^2 R_\varepsilon^\Phi, \end{cases}
\end{equation}
where we have set
\begin{equation}
\label{def:RepsU}
R_\varepsilon^U := - \div \big( U_\varepsilon^2 \, \nabla \Phi_\varepsilon \big) + \sigma U_\varepsilon^2 \, \sin(\Phi_\varepsilon) \, \cos(\Phi_\varepsilon),
\end{equation}
and
\begin{equation}
\label{def:RepsPhi}
R_\varepsilon^\Phi := - \sigma U_\varepsilon \, \sin^2(\Phi_\varepsilon) - \div \Big( \frac{\nabla U_\varepsilon}{1 - \varepsilon^2 U_\varepsilon^2} \Big) + \varepsilon^2 U_\varepsilon \, \frac{|\nabla U_\varepsilon|^2}{(1 - \varepsilon^2 U_\varepsilon^2)^2} - U_\varepsilon \, |\nabla \Phi_\varepsilon|^2.
\end{equation}
In view of the Sobolev control in~\eqref{borne-W-bis}, the remainder terms $R_\varepsilon^U$ and $R_\varepsilon^\Phi$ are bounded uniformly with respect to $\varepsilon$ in Sobolev spaces, with a loss of three derivatives. Due to this observation, the differences $u_\varepsilon := U_\varepsilon - U$ and $\varphi_\varepsilon := \Phi_\varepsilon - \Phi$ between a solution $(U_\varepsilon, \Phi_\varepsilon)$ to~\eqref{HLLeps} and a solution $(U, \Phi)$ to~\eqref{sys:SG} are expected to be of order $\varepsilon^2$, if the corresponding initial conditions are close enough.

The proof of this claim would be immediate if the system~\eqref{HLLeps-ter} would not contain the nonlinear term $\sin(2 \Phi_\varepsilon)$. Due to this extra term, we have to apply a Gronwall argument in order to control the differences $u_\varepsilon$ and $\varphi_\varepsilon$. Rolling out this argument requires an additional Sobolev control on the solution $(U, \Phi)$ to~\eqref{sys:SG}.

In this direction, we use the consistency of the systems~\eqref{HLLeps-ter} and~\eqref{sys:SG} so as to mimic the proof of Corollary~\ref{cor:T} for a solution $(U, \Phi)$ to~\eqref{sys:SG}. Indeed, when $\varepsilon = 0$, the quantities $E_\varepsilon^k$ in~\eqref{E-k} reduce to
$$E_{\rm SG}^k(U, \Phi) := \frac{1}{2} \sum_{|\alpha| = k - 1} \int_{\R^N} \Big( |\partial_x^\alpha U|^2 + |\partial_x^\alpha \nabla \Phi|^2 + \sigma |\partial_x^\alpha \sin(\Phi)|^2 \Big).$$
When $(U, \Phi)$ is a smooth enough solution to~\eqref{sys:SG}, we can perform energy estimates on these quantities in order to obtain

\begin{lem}
\label{lem:est-SG}
Let $k \in \N$, with $k > N/2 + 1$. There exists a positive number $A_*$, depending only on $\sigma$, $k$ and $N$, such that, given any initial datum $(U^0, \Phi^0) \in H^{k - 1}(\R^N) \times H_{\sin}^k(\R^N)$, there exists a positive time
$$T_* \geq \frac{1}{A_* \big( \| U^0 \|_{H^{k - 1}} + \| \nabla \Phi^0 \|_{H^{k - 1}} + \| \sin(\Phi^0) \|_{H^{k - 1}} \big)^2},$$
such that the unique solution $(U, \Phi)$ to~\eqref{sys:SG} with initial condition $(U^0, \Phi^0)$ satisfies the energy estimate
\begin{align*}
\| U(\cdot, t) \|_{H^{k - 1}} + \| \nabla \Phi(\cdot, t) \|_{H^{k - 1}} + & \| \sin(\Phi(\cdot, t)) \|_{H^{k - 1}}\\
\leq & A_* \Big( \| U^0 \|_{H^{k - 1}} + \| \nabla \Phi^0 \|_{H^{k - 1}} + \| \sin(\Phi^0) \|_{H^{k - 1}} \Big),
\end{align*}
for any $0 \leq t \leq T_*$.
\end{lem}

In view of~\eqref{sys:SG} and~\eqref{HLLeps-ter}, the differences $v_\varepsilon = U_\varepsilon - U$ and $\varphi_\varepsilon = \Phi_\varepsilon - \Phi$ satisfy
\begin{equation}
\label{eq:sys-diff}
\begin{cases} \partial_t v_\varepsilon = \Delta \varphi_\varepsilon - \sigma \sin(\varphi_\varepsilon) \, \cos(\Phi_\varepsilon + \Phi) + \varepsilon^2 R_\varepsilon^U,\\
\partial_t \varphi_\varepsilon = v_\varepsilon + \varepsilon^2 R_\varepsilon^\Phi. \end{cases}
\end{equation}
With Corollary~\ref{cor:T} and Lemma~\ref{lem:est-SG} at hand, we can control these differences by performing similar energy estimates on the functionals
\begin{equation}
\label{def:gE-SG-k}
\gE_{\rm SG}^k := \frac{1}{2} \sum_{|\alpha| = k - 1} \int_{\R^N} \big( |\partial_x^\alpha v_\varepsilon|^2 + |\partial_x^\alpha \nabla \varphi_\varepsilon|^2 + \sigma |\partial_x^\alpha \sin(\varphi_\varepsilon)|^2 \big).
\end{equation}
This is enough to obtain

\begin{prop}
\label{prop:error}
Let $k \in \N$, with $k > N/2 + 1$. Given an initial condition $(U_\varepsilon^0, \Phi_\varepsilon^0) \in \boN\boV^{k + 2}(\R^N)$, assume that the unique corresponding solution $(U_\varepsilon, \Phi_\varepsilon)$ to~\eqref{HLLeps} is well-defined on a time interval $[0, T]$ for a positive number $T$, and that it satisfies the uniform bound
\begin{equation}
\label{eq:unif-bound}
\varepsilon \big\| U_\varepsilon(\cdot, t) \big\|_{L^\infty} \leq \frac{1}{\sqrt{2}},
\end{equation}
for any $t \in [0, T]$. Consider similarly an initial condition $(U^0, \Phi^0) \in L^2(\R^N) \times H_{\sin}^1(\R^N)$, and denote by $(U, \Phi) \in \boC^0(\R, L^2(\R^N) \times H_{\sin}^1(\R^N))$ the unique corresponding solution to~\eqref{sys:SG}. Set $u_\varepsilon := U_\varepsilon - U$, $\varphi_\varepsilon := \Phi_\varepsilon - \Phi$, and
$$\boK_\varepsilon(T) := \max_{t \in [0, T]} \Big( \| U_\varepsilon(\cdot, t) \|_{H^k} + \varepsilon \| \nabla U_\varepsilon(\cdot, t) \|_{H^k} + \| \nabla \Phi_\varepsilon(\cdot, t) \|_{H^k} + \| \sin(\Phi_\varepsilon(\cdot, t)) \|_{H^k} \Big).$$

$(i)$ Assume that $\Phi_\varepsilon^0 - \Phi^0 \in L^2(\R^N)$. Then, there exists a positive number $C_1$, depending only on $\sigma$ and $N$, such that
\begin{equation}
\label{error1}
\| \varphi_\varepsilon(\cdot, t) \|_{L^2} \leq C_1 \Big( \| \varphi_\varepsilon^0 \|_{L^2} + \| v_\varepsilon^0 \|_{L^2} + \varepsilon^2 \, \boK_\varepsilon(T) \, \big( 1 + \varepsilon^2 \boK_\varepsilon(T)^2 + \boK_\varepsilon(T)^3 \big) \Big) \, e^{C_1 t},
\end{equation}
for any $t \in [0, T]$.

$(ii)$ Assume that $N \geq 2$, or that $N = 1$ and $k > N/2 + 2$. Then, there exists a positive number $C_2$, depending only on $\sigma$ and $N$, such that
\begin{equation}
\label{error2}
\begin{split}
\| u_\varepsilon(\cdot, t) \|_{L^2} + \| \nabla \varphi_\varepsilon(\cdot, t) \|_{L^2} + \| \sin(\varphi_\varepsilon(\cdot, t)) & \|_{L^2} \leq C_2 \Big( \| u_\varepsilon^0 \|_{L^2} + \| \nabla \varphi_\varepsilon^0 \|_{L^2} + \| \sin(\varphi_\varepsilon^0) \|_{L^2}\\
& + \varepsilon^2 \, \boK_\varepsilon(T) \, \big( 1 + \varepsilon^2 \boK_\varepsilon(T)^2 + \boK_\varepsilon(T)^3 \big) \Big) \, e^{C_2 t},
\end{split}
\end{equation}
for any $t \in [0, T]$.

$(iii)$ Assume that $k > N/2 + 3$ and that the pair $(U, \Phi)$ belongs to $\boC^0([0, T], H^k(\R^N) \times H_{\sin}^{k + 1}(\R^N))$. Set
$$\kappa_\varepsilon(T) := \boK_\varepsilon(T) + \max_{t \in [0, T]} \Big( \| U(\cdot, t) \|_{H^k} + \| \nabla \Phi(\cdot, t) \|_{H^k} + \| \sin(\Phi(\cdot, t)) \|_{H^k} \Big).$$
Then, there exists a positive number $C_k$, depending only on $\sigma$, $k$ and $N$, such that
\begin{equation}
\label{errork}
\begin{split}
\| u_\varepsilon(\cdot, t) \|_{H^{k - 3}} + \| & \nabla \varphi_\varepsilon(\cdot, t) \|_{H^{k - 3}} + \| \sin(\varphi_\varepsilon(\cdot, t)) \|_{H^{k - 3}}\\
\leq C_k \Big( & \| u_\varepsilon^0 \|_{H^{k - 3}} + \| \nabla \varphi_\varepsilon^0 \|_{H^{k - 3}} + \| \sin(\varphi_\varepsilon^0) \|_{H^{k - 3}}\\
& + \varepsilon^2 \, \kappa_\varepsilon(T) \, \big( 1 + \varepsilon^2 \kappa_\varepsilon(T)^2 + (1 + \varepsilon^2) \kappa_\varepsilon(T)^3 \big) \Big) \, e^{C_k (1 + \kappa_\varepsilon(T)^2) t},
\end{split}
\end{equation}
for any $t \in [0, T]$.
\end{prop}

We are now in position to conclude the proof of Theorem~\ref{thm:conv-SG}.

\begin{proof}[Proof of Theorem~\ref{thm:conv-SG}]
In view of Corollaries~\ref{cor:HLL-Cauchy} and~\ref{cor:T}, there exists a positive number $C_*$, depending only on $\sigma$, $k$ and $N$, for which, given any initial condition $(U_\varepsilon^0, \Phi_\varepsilon^0) \in \boN\boV^{k + 2}(\R^N)$ such that~\eqref{cond:key} holds, there exists a number $T_\varepsilon$ satisfying~\eqref{cond:Teps} such that the unique solution $(U_\varepsilon, \Phi_\varepsilon)$ to~\eqref{HLLeps} with initial conditions $(U_\varepsilon^0, \Phi_\varepsilon^0)$ lies in $\boC^0([0, T_\varepsilon], \boN\boV^{k + 1}(\R^N))$. Moreover, the quantity $\boK_\varepsilon(T_\varepsilon)$ in Proposition~\ref{prop:error} is bounded by
$$\boK_\varepsilon(T_\varepsilon) \leq C_* \boK_\varepsilon(0).$$
Enlarging if necessary the value of $C_*$, we then deduce statements $(ii)$ and $(iii)$ in Theorem~\ref{thm:conv-SG} from statements $(i)$ and $(ii)$ in Proposition~\ref{prop:error}. 

Similarly, given a pair $(U^0, \Phi^0) \in H^k(\R^N) \times H_{\sin}^{k + 1}(\R^N))$, we derive from Theorem~\ref{thm:SG-Cauchy-smooth} and Lemma~\ref{lem:est-SG} the existence of a number $T_\varepsilon^*$ such that~\eqref{cond:Teps*} holds, and the unique solution $(U, \Phi)$ to~\eqref{sys:SG} with initial conditions $(U^0, \Phi^0)$ is in $\boC^0([0, T_\varepsilon^*], H^k(\R^N) \times H_{\sin}^{k + 1}(\R^N))$. Statement $(iv)$ in Theorem~\ref{thm:conv-SG} then follows from statement $(iii)$ in Proposition~\ref{prop:error}. This completes the proof of Theorem~\ref{thm:conv-SG}.
\end{proof}

%%%%%%%%%%%%%%%%%%%%%%%%%%%%%%%%%
%%%%%%%%%%%%%%%%%%%%%%%%%%%%%%%%%
\subsection{Outline of the paper}
\label{sub:plan}
%%%%%%%%%%%%%%%%%%%%%%%%%%%%%%%%%
%%%%%%%%%%%%%%%%%%%%%%%%%%%%%%%%%

The paper is organized as follows. In the next section, we gather the proofs of Theorems~\ref{thm:SG-Cauchy} and~\ref{thm:SG-Cauchy-smooth} concerning the Cauchy problem for the Sine-Gordon equation. These results are well-known by the experts, but we did not find their proofs in the literature. For the sake of completeness, we provide them below. 

Section~\ref{sec:LL-Cauchy} is devoted to the analysis of the local well-posedness of the Landau-Lifshitz equation in the original and hydrodynamical frameworks.

In Section~\ref{sec:SG-derivation}, we collect the various elements concerning the derivation of the Sine-Gordon regime in Theorem~\ref{thm:conv-SG} by addressing the proofs of Proposition~\ref{prop:estimate}, Corollary~\ref{cor:T}, Lemma~\ref{lem:est-SG} and Proposition~\ref{prop:error}.

We similarly clarify the derivation of the wave equation in Section~\ref{sec:wave}.

In Appendix~\ref{sec:Hsink}, we describe the main properties of the sets $H_{\sin}^k(\R^N)$, while Appendix~\ref{sec:tame-estimates} gathers the tame estimates that we use in our computations. Finally, Appendix~\ref{sec:solitons} contains more material about the solitons for the one-dimensional Landau-Lifshitz equations and their correspondence with the solitons for the Sine-Gordon equation.

%%%%%%%%%%%%%%%%%%%%%%%%%%%%%%%%%%%%%%%%%%%%%%%%%%%%%%%%%%%%%%%%%%%%%%%%%%%%%%%%%%%%%%%%%%%%%%%%%%%%%%%%%%%%%%
%%%%%%%%%%%%%%%%%%%%%%%%%%%%%%%%%%%%%%%%%%%%%%%%%%%%%%%%%%%%%%%%%%%%%%%%%%%%%%%%%%%%%%%%%%%%%%%%%%%%%%%%%%%%%%
%%%%%%%%%%%%%%%%%%%%%%%%%%%%%%%%%%%%%%%%%%%%%%%%%%%%%%%%%%%%%%%%%%%%%%%%%%%%%%%%%%%%%%%%%%%%%%%%%%%%%%%%%%%%%%
\numberwithin{cor}{section}
\numberwithin{lem}{section}
\numberwithin{equation}{section}
\numberwithin{prop}{section}
\numberwithin{rem}{section}
\numberwithin{thm}{section}

\section{The Cauchy problem for the Sine-Gordon equation in the product sets \texorpdfstring{$H_{\sin}^k(\R^N) \times H^{k - 1}(\R^N)$}{}}
\label{sec:SG-Cauchy}
%%%%%%%%%%%%%%%%%%%%%%%%%%%%%%%%%%%%%%%%%%%%%%%%%%%%%%%%%%%%%%%%%%%%%%%%%%%%%%%%%%%%%%%%%%%%%%%%%%%%%%%%%%%%%%
%%%%%%%%%%%%%%%%%%%%%%%%%%%%%%%%%%%%%%%%%%%%%%%%%%%%%%%%%%%%%%%%%%%%%%%%%%%%%%%%%%%%%%%%%%%%%%%%%%%%%%%%%%%%%%
%%%%%%%%%%%%%%%%%%%%%%%%%%%%%%%%%%%%%%%%%%%%%%%%%%%%%%%%%%%%%%%%%%%%%%%%%%%%%%%%%%%%%%%%%%%%%%%%%%%%%%%%%%%%%%

In this section, we assume, up to a scaling in space, that $\sigma = \pm 1$, and we refer to Appendix~\ref{sec:Hsink} for the various notations concerning the spaces $H_{\sin}^k(\R^N)$.

%%%%%%%%%%%%%%%%%%%%%%%%%%%%%%%%%%%%%%%%%%%%%%%%%
%%%%%%%%%%%%%%%%%%%%%%%%%%%%%%%%%%%%%%%%%%%%%%%%%
\subsection{Proof of Theorem~\ref{thm:SG-Cauchy}}
\label{sub:SG-Hsin1}
%%%%%%%%%%%%%%%%%%%%%%%%%%%%%%%%%%%%%%%%%%%%%%%%%
%%%%%%%%%%%%%%%%%%%%%%%%%%%%%%%%%%%%%%%%%%%%%%%%%

The proof follows from the following proposition.

\begin{prop}
\label{prop:SG-Cauchy}
Let $f \in H_{\sin}^\infty(\R^N)$. Given two functions $(\varphi^0, \phi^1) \in H^1(\R^N) \times L^2(\R^N)$,
there exists a unique function $\varphi \in \boC^0(\R, H^1(\R^N)) \cap \boC^1(\R, L^2(\R^N))$ such that $\phi = f + \varphi$ satisfies the Sine-Gordon equation with initial conditions $(\phi^0 = f + \varphi^0, \phi^1)$. Moreover, this solution satisfies the following statements.

$(i)$ For any positive number $T$, the flow map $(\varphi^0, \phi^1) \mapsto (\varphi, \partial_t \phi)$ is continuous from $H^1(\R^N) \times L^2(\R^N)$ to $\boC^0([-T, T], H^1(\R^N)) \times \boC^0([-T, T], L^2(\R^N))$.

$(ii)$ When $\varphi^0 \in H^2(\R^N)$ and $\phi^1 \in H^1(\R^N)$, the function $\varphi$ belongs to the space $\boC^0(\R, H^2(\R^N)) \cap \boC^1(\R, H^1(\R^N)) \cap \boC^2(\R, L^2(\R))$.

$(iii)$ The Sine-Gordon energy $E_{\rm SG}$ is conserved along the flow.
\end{prop}

\begin{proof}
We decompose a solution $\phi$ to the Sine-Gordon equation as $\phi = f + \varphi$. The function $\varphi$ then solves the nonlinear wave equation
\begin{equation}
\label{eq:SG-varphi}
\partial_{t t} \varphi - \Delta \varphi = \Delta f - \frac{\sigma}{2} \sin(2 f + 2 \varphi),
\end{equation}
with a Lipschitz nonlinearity $H(\varphi) := - \sigma \sin(2 f + 2 \varphi)/2$. Therefore, we can apply the contraction mapping theorem in order to construct a unique local solution. Its global nature, the continuity of the corresponding flow and the conservation of the energy then follow from standard arguments in the context of the nonlinear wave equations. For the sake of completeness, we provide the following details.

The Duhamel formula for the nonlinear wave equation~\eqref{eq:SG-varphi} with initial conditions $(\varphi(\cdot, 0) = \varphi^0, \partial_t \varphi(\cdot, 0) = \varphi^1)$ writes as
\begin{equation}
\label{eq:Duhamel-SG}
\begin{split}
\varphi(\cdot, t) = A(\varphi)(\cdot, t) := & \cos(t D) \varphi^0 + \frac{\sin(t D)}{D} \varphi^1 - \int_0^t \sin((t - \tau) D) \, D f \, d\tau\\ & + \int_0^t \frac{\sin((t - \tau) D)}{D} \, H(\varphi)(\cdot, \tau) d\tau,
\end{split}
\end{equation}
where we set, here as in the sequel, $D := \sqrt{- \Delta}$. In order to solve this equation, we now apply the contraction mapping theorem to the functional $A$ in the function space $\boC^0([- T, T], H^1(\R^N))$ for a well-chosen positive number $T$.

Let $k \in \N$. Given a function $g \in H^k(\R^N)$, we know that
\begin{equation}
\label{reg:conv1}
\cos(t D) g \in \boC^0(\R, H^k(\R^N)) \cap \boC^1(\R, H^{k - 1}(\R^N)),
\end{equation}
and
\begin{equation}
\label{reg:conv2}
D^j \sin(t D) g \in \boC^0(\R, H^{k -j}(\R^N)) \cap \boC^1(\R, H^{k - j - 1}(\R^N)),
\end{equation}
for any integer $j \geq - 1$ (see e.g.~\cite{Treves0}). Since $D f$ is in $H^\infty(\R^N) := \cap_{\ell = 1}^\infty H^\ell(\R^N)$, the first three terms in the definition of the functional $A$ are well-defined and belong to $\boC^0(\R, H^1(\R^N)) \cap \boC^1(\R, L^2(\R^N))$.

Let $T > 0$. Since $f$ is in $H_{\rm sin}^1(\R^N)$, the function $H(v)$ belongs to $\boC^0([- T, T],H^1(\R^N))$, when $v$ lies in this space. In view of~\eqref{reg:conv1} and~\eqref{reg:conv2}, the last term in the definition of the function $A(v)$ is therefore in $\boC^0([- T, T], H^2(\R^N)) \cap \boC^1([- T, T], H^1(\R^N))$. Hence, the map $A$ is well-defined from $\boC^0([- T, T], H^1(\R^N))$ to $\boC^0([- T, T], H^1(\R^N)) \cap \boC^1([- T, T], L^2(\R^N))$, and its time derivative is given by
$$\partial_t A(v)(\cdot, t) = - \sin(t D) D \varphi^0 + \cos(t D) \varphi^1 + \int_0^t \cos((t - \tau) D) \big( H(v)(\cdot, \tau) + \Delta f \big) \, d\tau,$$
when $v \in \boC^0([- T, T], H^1(\R^N))$.

Given two functions $(v_1, v_2) \in \boC^0([- T, T], H^1(\R^N)^2$, we next have
$$A(v_1)(\cdot, t) - A(v_2)(\cdot, t) = \int_0^t \frac{\sin((t - \tau) D)}{D} \, \big(H(v_1)(\cdot, \tau) - H(v_2)(\cdot, \tau) \big) \, d\tau.$$
Since
$$\big\| H(v_1) - H(v_2) \big\|_{\boC^0([- T, T], L^2)} \leq \| v_1 - v_2 \|_{\boC^0([- T, T], L^2)},$$
and since the operators $\sin(\nu D)$ and $\sin(\nu D)/(\nu D)$ are bounded on $L^2(\R^N)$, uniformly with respect to $\nu \in \R$, we deduce that
$$\| A(v_1) - A(v_2) \|_{\boC^0([- T, T], H^1)} \leq T \sqrt{1 + T^2} \| v_1 - v_2 \|_{\boC^0([- T, T], L^2)}.$$
Taking $T = 1/2$, we conclude that the functional $A$ is a contraction on $\boC^0([- 1/2, 1/2], L^2(\R^N))$. The contraction mapping theorem provides the existence and uniqueness of a solution $\varphi \in \boC^0([- 1/2, 1/2], H^1(\R^N))$ to the equation $\varphi = A(\varphi)$, which is also in $\boC^1(\R, L^2(\R^N))$ due to this fixed-point equation. Since the existence time $T = 1/2$ is independent of the initial conditions, we can extend this unique solution on $\R$. Note finally that the function $\phi = f + \varphi$ solves the Sine-Gordon equation with initial datum $(\phi^0 = f + \varphi^0, \phi^1 = \varphi^1)$ (at least in the sense of distributions).

In order to prove $(i)$, we make explicit the dependence on the initial conditions $(\varphi^0, \varphi^1)$ of the functional $A$ by writing $A_{\varphi^0, \varphi^1}$. Given another pair of initial conditions $(\tilde{\varphi}^0, \tilde{\varphi}^1) \in H^1(\R^N) \times L^2(\R^N)$, we infer again from~\eqref{reg:conv1} and~\eqref{reg:conv2} that
\begin{align*}
\big\| A_{\varphi^0, \varphi^1}(v) - A_{\tilde{\varphi}^0, \tilde{\varphi}^1}(\tilde{v}) \big\|_{\boC^0([- T, T], H^1)} +
\big\| \partial_t A_{\varphi^0, \varphi^1}(v) - \partial_t A_{\tilde{\varphi}^0, \tilde{\varphi}^1}(\tilde{v}) \big\|_{\boC^0([- T, T ], L^2)}\\
\leq 2 \| \varphi^0 - \tilde{\varphi}^0 \|_{H^1} + \sqrt{1 + T^2} \| \varphi^1 - \tilde{\varphi}^1 \|_{L^2} + T (1 + \sqrt{1 + T^2}) \| v - \tilde{v} \|_{\boC^0([- T, T], L^2)},
\end{align*}
for any functions $(v, \tilde{v}) \in \boC^0([- T, T], H^1)^2$. Taking $T = 1/4$ so that $T (1 + \sqrt{1 + T^2}) < 1$, we deduce the existence of a universal constant $K$ such that the solutions $\varphi$ and $\tilde{\varphi}$ corresponding to the initial conditions $(\varphi^0, \varphi^1)$, respectively $(\tilde{\varphi}^0, \tilde{\varphi}^1)$, satisfy
$$\big\| \varphi - \tilde{\varphi} \big\|_{\boC^0([- \frac{1}{4}, \frac{1}{4}], H^1)} + \big\| \partial_t \varphi - \partial_t \tilde{\varphi} \big\|_{\boC^0([- \frac{1}{4}, \frac{1}{4}], L^2)} \leq K \big( \| \varphi^0 - \tilde{\varphi}^0 \|_{H^1} + \| \varphi^1 - \tilde{\varphi}^1 \|_{L^2} \big).$$
A covering argument is enough to establish the continuity of the flow with respect to the initial datum.

When the initial datum $(\varphi^0, \varphi^1)$ lies in $H^2(\R^N) \times H^1(\R^N)$, it follows from~\eqref{reg:conv1} and~\eqref{reg:conv2} that the first two terms in the definition of the functional $A$ are in $\boC^0(\R, H^2(\R^N)) \cap \boC^1(\R, H^1(\R^N))$. Due to the smoothness of the function $f$, this is also true for the third term, while we have already proved this property for the fourth term, when $\varphi$ is in $\boC^0(\R, H^1(\R^N))$. As a consequence, it only remains to invoke the equation $\varphi = A(\varphi)$ in order to prove that the solution $\varphi$ is actually in $\boC^0(\R, H^2(\R^N)) \cap \boC^1(\R, H^1(\R^N))$. Coming back to~\eqref{eq:SG-varphi}, we conclude that the solution $\varphi$ is also in $\boC^2(\R, L^2(\R^N))$.

In this situation, we are authorized to differentiate the Sine-Gordon energy and to integrate by parts in order to compute
$$\frac{d}{dt} \, E_{\rm SG}(\phi(\cdot, t)) = 0.$$
Therefore, the energy of a solution $\varphi \in \boC^0(\R, H^2(\R^N)) \cap \boC^1(\R, H^1(\R^N)) \cap \boC^0(\R, L^2(\R^N))$ is conserved along the flow. In the general situation where the solution is only in $\boC^0(\R, H^1(\R^N)) \cap \boC^1(\R, L^2(\R^N))$, the conservation of the energy follows from the continuity of the flow by applying a standard density argument.
\end{proof}

With Proposition~\ref{prop:SG-Cauchy} at hand, we are in position to show Theorem~\ref{thm:SG-Cauchy}.

\begin{proof}[Proof of Theorem~\ref{thm:SG-Cauchy}]
Consider initial conditions $(\phi^0, \phi^1) \in H_{\sin}^1(\R^N) \times H^1(\R^N)$. Lemma~\ref{lem:decompose} provides the existence of two functions $f \in H_{\sin}^\infty(\R^N)$ and $\varphi^0 \in H^1(\R^N)$ such that $\phi^0 = f + \varphi^0$. In this case, the space $\phi^0 + H^1(\R^N)$ is equal to $f + H^1(\R^N)$. Proposition~\ref{prop:SG-Cauchy} then provides the existence of a solution to the Sine-Gordon equation $\phi = f + \varphi \in \boC^0(\R, \phi^0 + H^1(\R^N))$, with $\partial_t \phi \in \boC^0(\R, L^2(\R^N))$, for the initial conditions $(\phi^0, \phi^1)$.

Concerning the uniqueness of this solution, we have to prove that it does not depend on the choice of the function $f$. Given an alternative decomposition $\phi^0 = \tilde{f} + \tilde{\varphi}^0$ and the corresponding solution $\tilde{\phi} = \tilde{f} + \tilde{\varphi}$, we observe that
$$\tilde{f} - f = \varphi^0 - \tilde{\varphi}^0 \in H^1(\R^N).$$
Hence, the function $\tilde{f} - f + \tilde{\varphi}$ is a solution in $\boC^0(\R, H^1(\R^N)) \cap \boC^1(\R, L^2(\R^N))$ to~\eqref{eq:SG-varphi} with initial conditions $(\varphi^0, \phi^1)$. Since $w$ is the unique solution of this equation, we deduce that $\varphi = \tilde{f} - f + \tilde{\varphi}$, which means exactly that $\phi = \tilde{\phi}.$

Statements $(ii)$ and $(iii)$ are then direct consequences of $(ii)$ and $(iii)$ in Proposition~\ref{prop:SG-Cauchy}. Concerning $(i)$, we come back to the contraction mapping argument in the proof of this proposition. We make explicit the dependence on the parameters in the definition of $A$ and $H$ by writing $A_{f, \varphi^0, \varphi^1}$ and $H_f$. We check that
$$\big\| H_f(v) - H_{\tilde{f}}(\tilde{v}) \big\|_{\boC^0([- T, T], L^2)} \leq \big\| \sin(f - \tilde{f}) \big\|_{L^2} + \| v - \tilde{v} \|_{\boC^0([- T, T], L^2)},$$
when $(f, \tilde{f}) \in H_{\sin}^1(\R^N)^2$ and $(v, \tilde{v}) \in \boC^0([- T, T], H^1(\R^N))^2$. Applying~\eqref{reg:conv1} and~\eqref{reg:conv2}, we obtain
\begin{align*}
\big\| A_{\varphi^0, \varphi^1, f}(v) - A_{\tilde{\varphi}^0, \tilde{\varphi}^1, \tilde{f}}(\tilde{v}) & \big\|_{\boC^0([- T, T], H^1)} + \big\| \partial_t A_{\varphi^0, \varphi^1, f}(v) - \partial_t A_{\tilde{\varphi}^0, \tilde{\varphi}^1, \tilde{f}}(\tilde{v}) \big\|_{\boC^0([- T, T ], L^2)}\\
\leq 2 \big\| \varphi^0 - \tilde{\varphi}^0 \big\|_{H^1} & + (1 + \sqrt{1 + T^2}) \big\| \varphi^1 - \tilde{\varphi}^1 \big\|_{L^2} + 2 T \big\| \nabla f - \nabla \tilde{f} \|_{H^1}\\
& + T (1 + \sqrt{1 + T^2}) \Big( \big\| \sin(f - \tilde{f}) \big\|_{L^2} + \big\| v - \tilde{v} \big\|_{\boC^0([- T, T], L^2)} \Big),
\end{align*}
for any $(\varphi^0, \tilde{\varphi}^0) \in H^1(\R^N)^2$ and $(\varphi^1, \tilde{\varphi}^1) \in L^2(\R^N)^2$. Taking $T = 1/4$, we are led to the existence of a universal constant $K$ such that the solutions $\varphi$ and $\tilde{\varphi}$ corresponding to the initial conditions $(\varphi^0, \varphi^1)$, respectively $(\tilde{\varphi}^0, \tilde{\varphi}^1)$, satisfy
\begin{align*}
\big\| \varphi - \tilde{\varphi} \big\|_{\boC^0([- \frac{1}{4}, \frac{1}{4}], H^1)} + & \big\| \partial_t \varphi - \partial_t \tilde{\varphi} \big\|_{\boC^0([- \frac{1}{4}, \frac{1}{4}], L^2)}\\
\leq & K \Big( \big\| \varphi^0 - \tilde{\varphi}^0 \big\|_{H^1} + \big\| \varphi^1 - \tilde{\varphi}^1 \big\|_{L^2} + \big\| f - \tilde{f} \big\|_{H_{\rm sin}^2} \Big).
\end{align*}
Invoking the estimates in Lemma~\ref{lem:decompose}, this inequality can be translated in terms of the functions $\phi$ and $\tilde{\phi}$ as 
\begin{equation}
\label{eq:diff}
\begin{split}
\big\| \phi - \tilde{\phi} + \tilde{f} - f \big\|_{\boC^0([- \frac{1}{4}, \frac{1}{4}], H^1)} + & \big\| \partial_t \phi - \partial_t \tilde{\phi} \big\|_{\boC^0([- \frac{1}{4}, \frac{1}{4}], L^2)}\\
\leq & K \Big( d_{\sin}^1 \big( \phi^0, \tilde{\phi}^0 \big) + \big\| \phi^1 - \tilde{\phi}^1 \big\|_{L^2} \Big).
\end{split}
\end{equation}
It remains to use the inequalities
$$\| \sin(\phi - \tilde{\phi}) \|_{L^2} \leq \| \sin(f - \tilde{f}) \|_{L^2} + \| \varphi - \tilde{\varphi} \|_{L^2} \leq d_{\sin}^1(\phi^0, \tilde{\phi}^0) + \| \varphi - \tilde{\varphi} \|_{L^2},$$
and 
$$\| \nabla \phi - \nabla \tilde{\phi} \|_{L^2} \leq \| \nabla f - \nabla \tilde{f} \|_{L^2} + \| \nabla \varphi - \nabla \tilde{\varphi} \|_{L^2} \leq d_{\sin}^1(\phi^0, \tilde{\phi}^0) + \| \nabla \varphi - \nabla \tilde{\varphi} \|_{L^2},$$
to obtain the estimate in $(i)$ for $T = 1/4$. The general case follows from a covering argument.
\end{proof}

%%%%%%%%%%%%%%%%%%%%%%%%%%%%%%%%%%%%%%%%%%%%%%%%%%%%%%%%
%%%%%%%%%%%%%%%%%%%%%%%%%%%%%%%%%%%%%%%%%%%%%%%%%%%%%%%%
\subsection{Proof of Theorem~\ref{thm:SG-Cauchy-smooth}}
\label{sub:SG-Hsink}
%%%%%%%%%%%%%%%%%%%%%%%%%%%%%%%%%%%%%%%%%%%%%%%%%%%%%%%%
%%%%%%%%%%%%%%%%%%%%%%%%%%%%%%%%%%%%%%%%%%%%%%%%%%%%%%%%

We split the proof into three steps.

\setcounter{step}{0}
\begin{step}
\label{SG1}
Local well-posedness in the product sets $H_{\sin}^k(\R^N) \times H^{k - 1}(\R^N)$.
\end{step}

Concerning the existence and uniqueness of a solution, we apply again the contraction mapping theorem. Consider initial conditions $(\phi^0, \phi^1) \in H_{\sin}^k(\R^N) \times H^{k - 1}(\R^N)$, write $\phi^0 = f + \varphi^0$, with $f \in H_{\sin}^\infty(\R^N)$ and $\varphi^0 \in H^{k - 1}(\R^N)$, and set $\varphi_1 = \phi^1$. Going back to the Duhamel formula in~\eqref{eq:Duhamel-SG}, we derive from~\eqref{reg:conv1} and~\eqref{reg:conv2} that the first three terms in the definition of the functional $A$ are in $\boC^0([0, T], H^k(\R^N)) \cap \boC^1([0, T], H^{k - 1}(\R^N))$. Concerning the last term, we invoke the Moser estimates in Corollary~\ref{cor:moser-trigo} and the Sobolev embedding theorem in order to check that
$$\| \nabla H(v) \|_{H^{k - 1}} \leq C \big( 1 + \| \nabla v \|_{L^\infty}^{k - 1} + \| \nabla f \|_{L^\infty}^{k - 1} \big) \big( \| \nabla v \|_{H^{k - 1}} + \| \nabla f \|_{H^{k - 1}} \big),$$
when $v \in H^k(\R^N)$. Here as in the sequel, the positive number $C$ only depends on $k$ and $N$. Due to the Sobolev embedding theorem, the functional $A$ is well-defined on $\boC^0([0, T], H^k(\R^N))$, with values in $\boC^0([0, T], H^k(\R^N)) \cap \boC^1([0, T], H^{k - 1}(\R^N))$. Moreover, we check that
\begin{equation}
\label{thiem}
\begin{split}
\| A(v) \|_{\boC^0([0, T], H^k)} \leq & \| \varphi^0 \|_{H^k} + (1 + T) \| \varphi^1 \|_{H^{k - 1}} + T \| f \|_{H_{\sin}^{k + 1}} + T^2 \big( \| \sin(f) \|_{L^2} + \| v \|_{L^2} \big)\\
& + C T \big( 1 + \| \nabla v \|_{L^\infty}^{k - 1} + \| \nabla f \|_{L^\infty}^{k - 1} \big) \big( \| \nabla v \|_{H^{k - 1}} + \| \nabla f \|_{H^{k - 1}} \big),
\end{split}
\end{equation}
when $v$ belongs to $\boC^0([0, T], H^k(\R^N))$. Note here that $A$ actually takes values in $\boC^0([0, T], \linebreak[0] H^{k + 1}(\R^N)) \cap \boC^1([0, T], H^k(\R^N))$, when $\phi^0 \in H_{\rm sin}^{k + 1}(\R^N)$ and $\phi^1 \in H^k(\R^N)$.

Next, we again deduce from Corollary~\ref{cor:moser-trigo} and the Sobolev embedding theorem that
\begin{align*}
\| \nabla H(v_1) - \nabla H(v_2) \|_{H^{k - 1}} = & \big\| \nabla \big( \sin(v_1 - v_2) \cos(2 f + v_1 + v_2) \big) \big\|_{H^{k - 1}}\\
\leq & C \big( 1 + \| \nabla v_1 \|_{L^\infty}^{k - 1} + \| \nabla v_2 \|_{L^\infty}^{k - 1} + \| \nabla f \|_{L^\infty}^{k - 1} \big) \times\\
& \times \big( 1 + \| \nabla v_1 \|_{H^{k - 1}} + \| \nabla v_2 \|_{H^{k - 1}} + \| \nabla f \|_{H^{k - 1}} \big) \big\| v_1 - v_2 \|_{H^k},
\end{align*}
when $(v_1, v_2) \in H^k(\R^N)^2$. When $v_1$ and $v_2$ belong to $\boC^0([0, T], H^k(\R^N))$, we obtain
\begin{equation}
\label{wawrinka}
\begin{split}
\| A(v_1) - A(v_2) & \|_{\boC^0([0, T], H^k)} \leq C T \Big( T \| v_1 - v_2 \|_{\boC^0([0, T], L^2)} + \big\| v_1 - v_2 \|_{\boC^0([0, T], H^k)} \times\\
& \times \big( 1 + \| \nabla v_1 \|_{\boC^0([0, T], H^k)}^k + \| \nabla v_2 \|_{\boC^0([0, T], H^k)}^k + \| \nabla f \|_{\boC^0([0, T], H^k)}^k \big) \Big).
\end{split}
\end{equation}

At this stage, we set $R_0 := \| \varphi^0 \|_{H^k} + \| \varphi^1 \|_{H^{k - 1}}$. In view of~\eqref{thiem} and~\eqref{wawrinka}, there exists a positive number $T_0$ such that the functional $A$ is a contraction on the closed ball
$$B_R := \big\{ v \in \boC^0([0, T], H^k(\R^N)) : \| v \|_{\boC^0([0, T], H^k)} \leq R \big\}.$$
The existence and uniqueness of a local solution $\phi = f + \varphi$ to the Sine-Gordon equation with initial conditions $(\phi^0, \phi^1)$ then follows from the contraction mapping theorem. The property that it belongs to $\boC^0([0, T_{\max}^k), \phi^0 + H^k(\R^N))$, with $\partial_t \phi \in \boC^0([0, T_{\max}^k), H^{k - 1}(\R^N))$, as well as the local Lipschitz continuous dependence on the initial datum in $(ii)$, can be derived as in the proof of Theorem~\ref{thm:SG-Cauchy} (invoking, when necessary, the Moser estimates and the Sobolev embedding theorem).

Concerning the characterization of the maximal time of existence $T_{\max}^k$, we infer by contradiction from the previous contraction argument that we have
$$\lim_{t \to T_{\max}^k} \| \varphi(\cdot, t) \|_{H^k} = \infty \quad {\rm if} \ T_{\max}^k < \infty.$$
The characterization in $(i)$ then follows from the formula $\phi(\cdot, t) = f + \varphi(\cdot, t)$, which guarantees that
$$\| \varphi(\cdot, t) \|_{H^k} - \| f \|_{H_{\sin}^k} \leq d_{\sin}^k(\phi(\cdot, t), 0) \leq \| \varphi(\cdot, t) \|_{H^k} + \| f \|_{H_{\sin}^k},$$
for any $t \in [0, T_{\max}^k)$.

When $\phi^0 \in H_{\rm sin}^{k + 1}(\R^N)$ and $\phi^1 \in H^k(\R^N)$, it follows from the fixed-point equation that the solution $\phi$ is in $\boC^0([0, T_{\max}^k), \phi^0 + H^{k + 1}(\R^N))$, with $\partial_t \phi \in \boC^0([0, T_{\max}^k), H^k(\R^N))$. By uniqueness of the solution, the function $\phi$ is the restriction to the interval $[0, T_{\max}^k)$ of the solution $\tilde{\phi}$ in $\boC^0([0, T_{\max}^{k + 1}), \phi^0 + H^{k + 1}(\R^N))$, with $\partial_t \tilde{\phi} \in \boC^0([0, T_{\max}^{k + 1}), H^k(\R^N))$. Hence, we have
$$T_{\max}^k \leq T_{\max}^{k + 1}.$$
The equality in this formula is then a consequence of the characterization in $(i)$.

In order to complete the proof of $(iii)$, note that the Moser estimates and the Sobolev embedding theorem also guarantee that the function $\sin(2 \phi)$ belongs to $\boC^0([0, T_{\max}^k), H^{k + 1}(\R^N))$. As a consequence of the Sine-Gordon equation, the function $\partial_{tt} \phi$ is therefore in $\boC^0([0, T_{\max}^k), H^{k - 1}(\R^N))$.

At this stage, it only remains to establish statement $(iv)$. We first address the question of global well-posedness.

\begin{step}
\label{SG2}
Global well-posedness when $1 \leq N \leq 3$.
\end{step}

In view of statement $(iii)$ in Theorem~\ref{thm:SG-Cauchy-smooth}, the maximal time of existence $T_{\max}^k$ is equal to $T_{\max}^{k_N}$, where $k_N$ denotes the smallest integer larger than $N/2 + 1$. As a consequence, we are allowed to reduce the proof to the case $k = k_N$. We then argue by contradiction assuming that $T_{\max}^{k_N}$ is finite, and we obtain a contradiction by controlling uniformly the quantity $d_{\sin}^{k_N}(\phi(\cdot, t), 0)$ on the time interval $[0, T_{\max}^{k_N})$.

With this goal in mind, we introduce the function
$$E_\ell(t) = \frac{1}{2} \int_{\R^N} \Big( (\partial_t D^{\ell - 1} \phi(x, t))^2 + (D^\ell \phi(x, t))^2 \Big) \, dx,$$
for any $1 \leq \ell \leq k_N$. This function is well-defined for any $t \in [0, T_{\max}^{k_N})$. By Theorem~\ref{thm:SG-Cauchy}, the solution $\phi$ belongs to $\boC^0(\R, \phi^0 + H^2(\R))$, with $\partial_t \phi \in \boC^0(\R, H^1(\R))$. Hence, there exists a positive number $\boE_2$ such that
\begin{equation}
\label{berdych}
\int_{\R^N} \sin \big( \phi(x, t) \big)^2 \, dx + E_1(t) + E_2(t) \leq \boE_2,
\end{equation}
for any $t \in [0, T_{\max}^{k_N})$. Since $k_N = 2$ when $N = 1$, the bound in~\eqref{berdych} is enough to complete the proof in dimension one. 

Assume now that $N = 2$ or $N = 3$ in case $k_N = 3$. When $\phi^0 \in H_{\rm sin}^4(\R^N)$ and $\phi^1 \in H^3(\R^N)$, the quantity $E_3(t)$ is differentiable on $[0, T_{\max}^{k_N})$ by statement $(iii)$ in Theorem~\ref{thm:SG-Cauchy-smooth}. Moreover, we can use the Sine-Gordon equation and integrate by parts in order to obtain
$$E_3'(t) = - \frac{\sigma}{2} \int_{\R^N} D^2 \sin(2 \phi)(x, t) \, \partial_t D^2 \phi(x, t) \, dx.$$
This provides the estimate
\begin{equation}
\label{mahut}
E_3'(t) \leq \frac{1}{2} E_3(t)^\frac{1}{2} \big\| D^2 \sin(2 \phi)(\cdot, t) \big\|_{L^2}.
\end{equation}
The chain rule then gives
$$\big\| D^2 \sin(2 \phi) \big\|_{L^2} \leq 2 \| D^2 \phi \|_{L^2} + 4 \| \nabla \phi \|_{L^4}^2.$$
Combining the Sobolev embedding theorem and the bound in~\eqref{berdych} is enough to obtain
$$\big\| D^2 \sin(2 \phi) \big\|_{L^2} \leq K \big( \boE_2 + \boE_2^2 \big),$$
where $K$ refers, here as in the sequel, to a universal constant. In view of~\eqref{mahut}, there exists a positive number $\boE_3$ such that
$$E_3(t) \leq \boE_3,$$
for any $t \in [0, T_{\max}^{k_N})$. By a standard density argument, this bounds remains true when $\phi^0 \in H_{\rm sin}^3(\R^N)$ and $\phi^1 \in H^2(\R^N)$. This concludes the proof of Step~\ref{SG2}.

\begin{step}
\label{SG3}
Continuity of the flow map on $H_{\sin}^2(\R^N) \times H^1(\R^N)$ when $2 \leq N \leq 3$.
\end{step}

More precisely, we fix two initial conditions $(\phi^0 = f + \varphi^0, \phi^1)$, with $f \in H_{\sin}^\infty(\R^N)$ and $\varphi^0 \in H^2(\R)$. We derive from Theorem~\ref{thm:SG-Cauchy} the existence and uniqueness of a corresponding solution $\phi = f + \varphi$ to the Sine-Gordon equation, with $\varphi \in \boC^0(\R, H^2(\R^N))$ and $\partial_t \varphi \in \boC^0(\R, H^1(\R^N))$. Given any positive number $T$, our goal is to establish that
\begin{equation}
\label{flushing}
\max_{t \in [- T, T]} \Big( d_{\sin}^2 \big( \phi_n(\cdot, t), \phi(\cdot, t) \big) + \big\| \partial_t \phi_n(\cdot, t) - \partial_t \phi(\cdot, t) \big\|_{H^1} \Big) \to 0,
\end{equation}
for any sequence of solutions $\phi_n$ to the Sine-Gordon equation corresponding to initial data $(\phi_n^0, \phi_n^1) \in H_{\sin}^2(\R^N) \times H^1(\R^N)$ such that
\begin{equation}
\label{meadows}
\| \phi_n^0 - \phi^0 \|_{H_{\sin}^2} + \| \phi_n^1 - \phi^1 \|_{H^1} \to 0,
\end{equation}
as $n \to \infty$.

In order to establish this statement, we take a sequence of initial conditions $(\tilde{\phi}_p^0 = f + \tilde{\varphi}_p^0, \tilde{\phi}_p^1)$, with $(\tilde{\varphi}_p^0, \tilde{\phi}_p^1) \in H^\infty(\R^N)^2$, such that
\begin{equation}
\label{pouille}
\tilde{\varphi}_p^0 \to \varphi^0 \ {\rm in} \ H^2(\R^N), \quad {\rm and} \quad \tilde{\phi}_p^1 \to \phi^1 \ {\rm in} \ H^1(\R^N),
\end{equation}
as $p \to \infty$. We denote by $\tilde{\phi}_p = f + \tilde{\varphi}_p$ the corresponding solutions to the Sine-Gordon equation. By Steps~\ref{SG1} and~\ref{SG2}, they belong to $\boC^0(\R, f + H^3(\R^N))$, with $\partial_t \tilde{\phi}_p \in \boC^0(\R, H^2(\R^N))$. As a consequence of the Sine-Gordon equation, the derivative $\partial_{tt} \tilde{\phi}_p$ is therefore in $\boC^0(\R, H^1(\R^N))$.

On the other hand, we deduce from the proof of Proposition~\ref{prop:SG-Cauchy} that there exists a positive number $A$, not depending on $m$ and $p$, such that
\begin{equation}
\label{nishikori}
\big\| \tilde{\varphi}_m - \tilde{\varphi}_p \big\|_{\boC^0([- T, T], H^1)} + \| \partial_t \tilde{\phi}_m - \partial_t \tilde{\phi}_p \|_{\boC^0([- T, T], L^2)} \leq A \big( \| \tilde{\varphi}_m^0 - \tilde{\varphi}_p^0 \|_{H^1} + \| \tilde{\phi}_m^1 - \tilde{\phi}_p^1 \|_{L^2} \big),
\end{equation}
for any integers $(m, p) \in \N^2$. Hence, it follows from~\eqref{pouille} that
\begin{equation}
\label{dimitrov}
\big\| \tilde{\varphi}_p \big\|_{\boC^0([- T, T], H^1)} + \| \partial_t \tilde{\phi}_p \|_{\boC^0([- T, T], L^2)} \leq A,
\end{equation}
where $A$ refers, here as in the sequel, to a further positive number not depending on $p$.

We next prove that $(\tilde{\varphi}_p)_{p \in \N}$ and $(\partial_t \tilde{\phi}_p)_{p \in \N}$ are Cauchy sequences in $\boC^0([- T, T], H^2)$, respectively $\boC^0([- T, T], H^1)$. We first establish their boundedness by arguing as in Step~\ref{SG2}. We introduce the quantities
$$E_2^p(t) = \frac{1}{2} \int_{\R^N} \Big( |\partial_t \nabla \tilde{\phi}_p(x, t)|^2 + |D^2 \tilde{\phi}_p(x, t)|^2 \Big) \, dx.$$
which are well-defined and of class $\boC^1$ on $\R$ in view of the differentiability properties of the functions $\tilde{\phi}_p$. As in Step~\ref{SG2}, we compute
$$[E_2^p]'(t) = - \sigma \int_{\R^N} \cos(2 \tilde{\phi}_p) \big\langle \nabla \tilde{\phi}_p, \partial_t \nabla \tilde{\phi}_p \big\rangle_{\R^N} \leq 2 \| \nabla \tilde{\phi}_p(\cdot, t) \|_{L^2} E_2^p(t)^\frac{1}{2}.$$
In view of~\eqref{dimitrov}, this inequality guarantees that the quantity $E_2^p$ is bounded on $[- T, T]$ uniformly with respect to the integer $p$. In turn, this proves that the sequences $(D^2 \tilde{\phi}_p)_{p \in \N}$ and $(\partial_t \nabla \tilde{\phi}_p)_{n \in \N}$ are bounded in $\boC^0([- T, T], L^2)$. Combining~\eqref{nishikori} with the Sobolev embedding theorem, we conclude that
\begin{equation}
\label{delpotro}
\big\| \partial_t \nabla \tilde{\phi}_p \big\|_{\boC^0([- T, T], L^2)} + \| D^2 \tilde{\phi}_p \|_{\boC^0([- T, T], L^2)} + \| \nabla \tilde{\phi}_p \|_{\boC^0([- T, T], L^6)} \leq A.
\end{equation}

Given two integers $(m, p) \in \N^2$, we next introduce the difference $z := \tilde{\phi}_m - \tilde{\phi}_p = \tilde{\varphi}_m - \tilde{\varphi}_p$, and we consider the quantity
$$\delta E_2(t) := \frac{1}{2} \int_{\R^N} \Big( |\partial_t \nabla z(x, t)|^2 + |D^2 z(x, t)|^2 \Big) \, dx,$$
which is well-defined and of class $\boC^1$ on $\R$. Since the difference $z$ is solution to the wave equation
$$\partial_{tt} z - \Delta z = - \sigma \sin(z) \cos(\tilde{\phi}_m + \tilde{\phi}_p),$$
we obtain by integrating by parts that
$$\delta E_2'(t) = - \sigma \int_{\R^N} \big\langle \nabla \big( \sin(z) \cos(\tilde{\phi}_m + \tilde{\phi}_p) \big), \partial_t \nabla z \big\rangle_{\R^N}.$$
This provides the estimate
$$\delta E_2'(t) \leq \delta E_2(t)^\frac{1}{2} \, \Big( \| \nabla z(\cdot, t) \|_{L^2} + \| \sin(z(\cdot, t)) \|_{L^3} \big( \| \nabla \tilde{\phi}_m(\cdot, t) \|_{L^6} + \| \nabla \tilde{\phi}_p(\cdot, t) \|_{L^6} \big) \Big),$$
which we bound by
$$\delta E_2'(t) \leq A \, \delta E_2(t)^\frac{1}{2} \, \big( \| \tilde{\varphi}_m^0 - \tilde{\varphi}_p^0 \|_{H^1} + \| \tilde{\phi}_m^1 - \tilde{\phi}_p^1 \|_{L^2} \big)^\frac{2}{3} \, \big( 1 + \| \tilde{\varphi}_m^0 - \tilde{\varphi}_p^0 \|_{H^1} + \| \tilde{\phi}_m^1 - \tilde{\phi}_p^1 \|_{L^2} \big)^\frac{1}{3},$$
in view of~\eqref{nishikori},~\eqref{delpotro} and the inequalities
$$\| \sin(z(\cdot, t)) \|_{L^3} \leq \| \sin(z(\cdot, t)) \|_{L^2}^\frac{2}{3} \leq \| z(\cdot, t) \|_{L^2}^\frac{2}{3}.$$
Finally, we are led to
\begin{equation}
\label{karlovic}
\delta E_2(t) \leq A \Big( \delta E_2(0) + \big( \| \tilde{\varphi}_m^0 - \tilde{\varphi}_p^0 \|_{H^1} + \| \tilde{\phi}_m^1 - \tilde{\phi}_p^1 \|_{L^2} \big)^\frac{4}{3} \, \big( 1 + \| \tilde{\varphi}_m^0 - \tilde{\varphi}_p^0 \|_{H^1} + \| \tilde{\phi}_m^1 - \tilde{\phi}_p^1 \|_{L^2} \big)^\frac{2}{3} \Big),
\end{equation}
for any $t \in [- T, T]$. Invoking~\eqref{nishikori}, this shows that $(\tilde{\varphi}_p)_{p \in \N}$ and $(\partial_t \tilde{\phi}_p)_{p \in \N}$ are Cauchy sequences in $\boC^0([- T, T], H^2)$, respectively $\boC^0([- T, T], H^1)$. Since their limits in $\boC^0([- T, T], H^1)$ and $\boC^0([- T, T], L^2)$ are equal to $\varphi$, respectively $\partial_t \phi$, by the continuity of the flow in these spaces, we conclude that
\begin{equation}
\label{sock}
\big\| \tilde{\varphi}_p - \varphi \big\|_{\boC^0([- T, T], H^2)} + \| \partial_t \tilde{\phi}_p - \partial_t \phi \|_{\boC^0([- T, T], H^1)} \to 0,
\end{equation}
as $p \to \infty$.

With this density property at hand, we are able to establish the continuity of the flow. We argue as in the proof of~\eqref{sock}. Coming back to~\eqref{meadows}, we find functions $f_n \in H_{\sin}^\infty(\R^N)$ and $\varphi_n^0 \in H^2(\R^N)$ such that $\phi_n^0 = f_n + \varphi_n^0$, and
\begin{equation}
\label{isner}
d_{\sin}^2(f_n, f) + \| \varphi_n^0 - \varphi^0 \|_{H^2} \leq K d_{\sin}^2(\phi_n^0, \phi^0),
\end{equation}
for a universal constant $K$. Here, we have set, as above, $\phi^0 = f + \varphi^0$. Given any fixed integer $n$, we introduce initial conditions $(\tilde{\phi}_{n, p}^0 = f_n + \tilde{\varphi}_{n, p}^0, \tilde{\phi}_{n, p}^1)$, with $(\tilde{\varphi}_{n, p}^0, \tilde{\phi}_{n, p}^1) \in H^\infty(\R^N)^2$, such that
\begin{equation}
\label{baghdatis}
\tilde{\varphi}_{n, p}^0 \to \varphi_n^0 \ {\rm in} \ H^2(\R^N), \quad {\rm and} \quad \tilde{\phi}_{n, p}^1 \to \phi_n^1 \ {\rm in} \ H^1(\R^N),
\end{equation}
as $p \to \infty$. We denote by $\tilde{\phi}_{n, p} = f_n + \tilde{\varphi}_{n, p}$ the corresponding solutions to the Sine-Gordon equation. They belong to $\boC^0(\R, f_n + H^3(\R^N))$, with $\partial_t \tilde{\phi}_{n, p} \in \boC^0(\R, H^2(\R^N))$ and $\partial_{tt} \tilde{\phi}_{n, p} \in \boC^0(\R, H^1(\R^N))$, and we also derive from~\eqref{sock} that
\begin{equation}
\label{kyrgios}
\big\| \tilde{\varphi}_{n, p} - \varphi_n \big\|_{\boC^0([- T, T], H^2)} + \| \partial_t \tilde{\phi}_{n, p} - \partial_t \phi_n \|_{\boC^0([- T, T], H^1)} \to 0,
\end{equation}
as $p \to \infty$.

Going back to~\eqref{eq:diff}, we next have
\begin{equation}
\label{youzhny}
\big\| \tilde{\varphi}_{n, p} - \tilde{\varphi}_p \big\|_{\boC^0([- T, T], H^1)} + \| \partial_t \tilde{\phi}_{n, p} - \partial_t \tilde{\phi}_p \|_{\boC^0([- T, T], L^2)} \leq A \big( d_{\sin}^1(\tilde{\phi}_{n, p}^0, \tilde{\phi}_p^0) + \| \tilde{\phi}_{n, p}^1 - \tilde{\phi}_p^1 \|_{L^2} \big),
\end{equation}
where the positive number $A$ depends, here as in the sequel, neither on $n$, nor on $p$. As for~\eqref{dimitrov}, this yields
$$\big\| \tilde{\varphi}_{n, p} \big\|_{\boC^0([- T, T], H^1)} + \| \partial_t \tilde{\phi}_{n, p} \|_{\boC^0([- T, T], L^2)} \leq A,$$
and we can derive as in the proof of~\eqref{delpotro} that
$$\big\| \partial_t \nabla \tilde{\phi}_{n, p} \big\|_{\boC^0([- T, T], L^2)} + \| D^2 \tilde{\phi}_{n, p} \|_{\boC^0([- T, T], L^2)} + \| \nabla \tilde{\phi}_{n, p} \|_{\boC^0([- T, T], L^6)} \leq A.$$
We then follow the lines of the proof of~\eqref{karlovic}. Setting $z_p = \tilde{\phi}_{n, p} - \tilde{\phi}_p$ and
$$\delta E_2^p(t) := \frac{1}{2} \int_{\R^N} \Big( |\partial_t \nabla z_p(x, t)|^2 + |D^2 z_p(x, t)|^2 \Big) \, dx,$$
we compute
$$\delta E_2^p(t)' \leq \delta E_2^p(t)^\frac{1}{2} \, \Big( \| \nabla z_p(\cdot, t) \|_{L^2} + \| \sin(z_p(\cdot, t)) \|_{L^2}^\frac{2}{3} \Big),$$
so that, by~\eqref{youzhny},
\begin{align*}
\delta E_2^p(t)' \leq & \delta E_2^p(t)^\frac{1}{2} \, \Big( \| \nabla f_n - \nabla f \|_{L^2} + \| \sin(f_n - f) \|_{L^2}^\frac{2}{3}\\
& + \| \nabla \tilde{\varphi}_{n, p}(\cdot, t) - \nabla \tilde{\varphi}_p(\cdot, t) \|_{L^2} + \| \tilde{\varphi}_{n, p}(\cdot, t) - \tilde{\varphi}_p(\cdot, t) \|_{L^2}^\frac{2}{3} \Big).
\end{align*}
In view of~\eqref{isner} and~\eqref{youzhny}, this provides the estimate
\begin{equation}
\label{cilic}
\begin{split}
\big\| D^2 \tilde{\varphi}_{n, p} - & D^2 \tilde{\varphi}_p \big\|_{\boC^0([- T, T], H^1)} + \| \partial_t \nabla \tilde{\phi}_{n, p} - \partial_t \nabla \tilde{\phi}_p \|_{\boC^0([- T, T], L^2)}\\
\leq & A \big( d_{\sin}^2(\phi_n^0, \phi^0) + d_{\sin}^2(\tilde{\phi}_{n, p}^0, \tilde{\phi}_p^0) + \| \tilde{\phi}_{n, p}^1 - \tilde{\phi}_p^1 \|_{H^1} \big)^\frac{2}{3} \times\\
& \times \big( 1 + d_{\sin}^2(\phi_n^0, \phi^0) + d_{\sin}^2(\tilde{\phi}_{n, p}^0, \tilde{\phi}_p^0) + \| \tilde{\phi}_{n, p}^1 - \tilde{\phi}_p^1 \|_{H^1} \big)^\frac{1}{3}.
\end{split}
\end{equation}
In view of~\eqref{pouille},~\eqref{sock},~\eqref{baghdatis} and~\eqref{kyrgios}, we can take the limit $p \to \infty$ in~\eqref{youzhny} and~\eqref{cilic} in order to write
\begin{align*}
& \big\| \varphi_n - \varphi \big\|_{\boC^0([- T, T], H^2)} + \| \partial_t \phi_n - \partial_t \phi \|_{\boC^0([- T, T], H^1)}\\
\leq A \big( & d_{\sin}^2(\phi_n^0, \phi^0) + \| \phi_n^1 - \phi^1 \|_{H^1} \big)^\frac{2}{3} \, \big( 1 + d_{\sin}^2(\phi_n^0, \phi^0) + \| \phi_n^1 - \phi^1 \|_{H^1} \big)^\frac{1}{3}.
\end{align*}
The convergence in~\eqref{flushing} finally results from~\eqref{meadows},~\eqref{isner} and the identity $\phi_n - \phi = f_n - f + \varphi_n - \varphi$. This concludes the proofs of Step~\ref{SG3} and of Theorem~\ref{thm:SG-Cauchy-smooth}. \qed

%%%%%%%%%%%%%%%%%%%%%%%%%%%%%%%%%%%%%%%%%%%%%%%%%%%%%%%%%%%%%
%%%%%%%%%%%%%%%%%%%%%%%%%%%%%%%%%%%%%%%%%%%%%%%%%%%%%%%%%%%%%
%%%%%%%%%%%%%%%%%%%%%%%%%%%%%%%%%%%%%%%%%%%%%%%%%%%%%%%%%%%%%
\section{The Cauchy problem for the Landau-Lifshitz equation}
\label{sec:LL-Cauchy}
%%%%%%%%%%%%%%%%%%%%%%%%%%%%%%%%%%%%%%%%%%%%%%%%%%%%%%%%%%%%%
%%%%%%%%%%%%%%%%%%%%%%%%%%%%%%%%%%%%%%%%%%%%%%%%%%%%%%%%%%%%%
%%%%%%%%%%%%%%%%%%%%%%%%%%%%%%%%%%%%%%%%%%%%%%%%%%%%%%%%%%%%%

In this section, the parameters $\lambda_1$ and $\lambda_3$ are fixed non-negative numbers.

%%%%%%%%%%%%%%%%%%%%%%%%%%%%%%%%%%%%%%%%%%%%%%%%%%%%%%%%%%%%%%%%%%%%
%%%%%%%%%%%%%%%%%%%%%%%%%%%%%%%%%%%%%%%%%%%%%%%%%%%%%%%%%%%%%%%%%%%%
\subsection{Density in the spaces \texorpdfstring{$\boE^k(\R^N)$}{}}
\label{sub:density-Ek}
%%%%%%%%%%%%%%%%%%%%%%%%%%%%%%%%%%%%%%%%%%%%%%%%%%%%%%%%%%%%%%%%%%%%
%%%%%%%%%%%%%%%%%%%%%%%%%%%%%%%%%%%%%%%%%%%%%%%%%%%%%%%%%%%%%%%%%%%%

The proof of Theorem~\ref{thm:LL-Cauchy} below relies on a compactness argument, which requires the density of smooth functions in the sets $\boE^k(\R^N)$. Recall that these sets are equal to $Z^k(\R^N, \S^2)$ for any integer $k \geq 1$, where the vector spaces $Z^k(\R^N)$ are defined as in~\eqref{def:Zk}. In particular, the sets $\boE^k(\R^N)$ are complete metric spaces for the distance corresponding to the $Z^k$-norm. Using this norm, we can generalize~\cite[Lemma A.1]{deLaGra1} to arbitrary dimensions in order to check the density of smooth functions.

\begin{lem}
\label{lem:dens-smooth}
Let $k \in \N$, with $k > \frac{N}{2}$. Given any function $m \in \boE^k(\R^N)$, there exists a sequence of smooth functions $m_n \in \boE(\R^N)$, with $\nabla m_n \in H^\infty(\R^N)$, such that the differences $m_n - m$ are in $H^k(\R^N)$, and satisfy
$$m_n - m \to 0 \quad {\rm in} \ H^k(\R^N),$$
as $n \to \infty$. In particular, we have
$$\| m_n - m \|_{Z^k} \to 0.$$
\end{lem}

\begin{rem}
This density result is not necessarily true when $k \leq \frac{N}{2}$ (see e.g.~\cite[Section 4]{SchoUhl1} for a discussion about this claim).
\end{rem}

\begin{proof}
The proof is reminiscent from the one of~\cite[Lemma A.1]{deLaGra1}, which relies on standard arguments introduced in~\cite{SchoUhl1}. For the sake of completeness, we recall the following details.

We consider a function $\chi \in \boC^\infty(\R^N)$, with a compactly supported Fourier transform, and such that $|\widehat{\chi}| \leq 1$, $\widehat{\chi} = 1$ on the unit ball $B(0, 1)$, and $\widehat{\chi} = 0$ outside the ball $B(0, 2)$. We set
$$\mu_n(x) = n^N \int_{\R^N} \chi(n (x - y)) \, m(y) \, dy,$$
for any $n \in \N^*$ and $x \in \R^N$. Since $\chi$ belongs to the Schwartz class, the functions $\mu_n$ are well-defined and smooth on $\R^N$, and their Fourier transforms are equal to
$$\widehat{\mu_n}(\xi) = \widehat{\chi} \Big( \frac{\xi}{n} \Big) \widehat{m}(\xi).$$
As a consequence of this identity, their gradients $\mu_n$ belong to $H^\infty(\R^N)$, and the differences $\mu_n - m$ are in $H^k(\R^N)$, with
\begin{equation}
\label{pericles}
\mu_n - m \to 0 \quad {\rm in} \ H^k(\R^N),
\end{equation}
as $n \to \infty$. However, the functions $\mu_n$ are not $\S^2$-valued, so that they do not belong to the energy space $\boE(\R^N)$.

In order to fill this gap, we deduce from~\eqref{pericles} and the Sobolev embedding theorem that
$$\big\| |\mu_n| - 1 \big\|_{L^\infty} \to 0,$$
as $n \to \infty$. Therefore, the map $m_n = \mu_n/|\mu_n|$ is well-defined for $n$ large enough, and it satisfies the conclusions of Lemma~\ref{lem:dens-smooth}.
\end{proof}

%%%%%%%%%%%%%%%%%%%%%%%%%%%%%%%%%%%%%%%%%%%%%%%%%%%%%%%%%%%%%%%
%%%%%%%%%%%%%%%%%%%%%%%%%%%%%%%%%%%%%%%%%%%%%%%%%%%%%%%%%%%%%%%
\subsection{Proof of Proposition~\ref{prop:LL-energy-estimate}}
\label{sub:LL-energy-estimate}
%%%%%%%%%%%%%%%%%%%%%%%%%%%%%%%%%%%%%%%%%%%%%%%%%%%%%%%%%%%%%%%
%%%%%%%%%%%%%%%%%%%%%%%%%%%%%%%%%%%%%%%%%%%%%%%%%%%%%%%%%%%%%%%

Let $T$ be a fixed positive number. Concerning the conservation of the Landau-Lifshitz energy, it follows from the smoothness assumptions on the solution $m$ that the function $E_{\rm LL}^1$ is of class $\boC^1$ on $[0, T]$, and that its time derivative is equal to
$$\big[ E_{\rm LL}^1 \big]'(t) = \int_{\R^N} \Big\langle \partial_t m, - \Delta m + \lambda_1 m_1 e_1 + \lambda_3 m_3 e_3 \Big\rangle_{\R^3}(x, t) \, dx.$$
In view of~\eqref{LL}, this expression identically vanishes, so that the Landau-Lifshitz energy is indeed conserved along the flow.

We now turn to the proof of~\eqref{eq:energy-estimate-LL}. Combining the assumptions in Proposition~\ref{prop:LL-energy-estimate} with the Moser estimates in Lemma~\ref{lem:moser}, we check that the second order derivative $\partial_{tt} m$ is well-defined as a function of $\boC^0([0, T], H^{k - 2}(\R^N))$. In view of~\eqref{def:E-LL-k}, the energies $E_{\rm LL}^\ell$ are of class $\boC^1$ on $[0, T]$, and we can integrate by parts in order to obtain the formula
\begin{equation}
\label{alabama}
\begin{split}
\big[ E_{\rm LL}^\ell \big]'(t) = \sum_{|\alpha| = \ell - 2} \int_{\R^N} \Big\langle \partial_t \partial_x^\alpha m, \partial_x^\alpha \Big( \partial_{tt} m + \Delta^2 m & - (\lambda_1 + \lambda_3) \big( \Delta m_1 e_1 + \Delta m_3 e_3 \big)\\
& + \lambda_1 \lambda_3 \big( m_1 e_1 + m_3 e_3 \big) \Big) \Big\rangle_{\R^3}(x, t) \, dx,
\end{split}
\end{equation}
for any $t \in [0, T]$. On the other hand, we derive from~\eqref{LL} the identity
\begin{equation}
\label{eq:second-LL}
\partial_{tt} m + \Delta^2 m - (\lambda_1 + \lambda_3) \big( \Delta m_1 e_1 + \Delta m_3 e_3 \big) + \lambda_1 \lambda_3 \big( m_1 e_1 + m_3 e_3 \big) = F(m),
\end{equation}
where we have set
\begin{equation}
\label{usa}
\begin{split}
& F(m) := \sum_{1 \leq i, j \leq N} \Big( \partial_i \big( 2 \langle \partial_i m, \partial_j m \rangle_{\R^3} \partial_j m - |\partial_j m|^2 \partial_i m \big) - 2 \partial_{ij} \big( \langle \partial_i m, \partial_j m \rangle_{\R^3} m \big) \Big)\\
& + \lambda_1 \Big( \div \big( (m_3^2 - 2 m_1^2) \nabla m + (m_1 m - m_3^2 e_1 + m_1 m_3 e_3) \nabla m_1 + (m_1 m_3 e_1 - m_3 m - m_1^2 e_3) \nabla m_3 \big)\\
& \quad \quad \quad + \nabla m_1 \cdot \big( m_1 \nabla m - m \nabla m_1 \big) + \nabla m_3 \cdot \big( m \nabla m_3 - m_3 \nabla m \big) + m_3 |\nabla m|^2 e_3\\
& \quad \quad \quad + \big( m_1 \nabla m_3 - m_3 \nabla m_1 \big) \cdot \big( \nabla m_1 e_3 - \nabla m_3 e_1 \big) + \lambda_1 m_1^2 \big( m_1 e_1 - m \big) \Big)\\
& + \lambda_3 \Big( \div \big( (m_1^2 - 2 m_3^2) \nabla m + (m_1 m_3 e_3 - m_1 m - m_3^2 e_1) \nabla m_1 + (m_3 m - m_1^2 e_3 + m_1 m_3 e_1) \nabla m_3 \big)\\
& \quad \quad \quad + \nabla m_3 \cdot \big( m_3 \nabla m - m \nabla m_3 \big) + \nabla m_1 \cdot \big( m \nabla m_1 - m_1 \nabla m \big) + m_1 |\nabla m|^2 e_1\\
& \quad \quad \quad + \big( m_1 \nabla m_3 - m_3 \nabla m_1 \big) \cdot \big( \nabla m_1 e_3 - \nabla m_3 e_1 \big) + \lambda_3 m_3^2 \big( m_3 e_3 - m \big) \Big)\\
& + \lambda_1 \lambda_3 \Big( (m_1^2 + m_3^2) m + m_1^2 m_3 e_3 + m_3^2 m_1 e_1 \Big).
\end{split}
\end{equation}
In order to derive this expression, we have used the pointwise identities
$$\langle m, \partial_i m \rangle_{\R^3} = \langle m, \partial_{ii} m \rangle_{\R^3} + |\partial_i m|^2 = \langle m, \partial_{iij} m \rangle_{\R^3} + 2 \langle \partial_i m, \partial_{ij} m \rangle_{\R^3} + \langle \partial_j m, \partial_{ii} m \rangle_{\R^3} = 0,$$
which hold for any $1 \leq i, j \leq N$, due to the property that $m$ is valued into the sphere $\S^2$.

Combining~\eqref{alabama} with~\eqref{eq:second-LL} and~\eqref{usa}, we obtain
\begin{equation}
\label{arkansas}
\big[ E_{\rm LL}^\ell \big]'(t) = \sum_{|\alpha| = \ell - 2} \int_{\R^N} \Big\langle \partial_t \partial_x^\alpha m, \partial_x^\alpha F(m) \Big\rangle_{\R^3}(x, t) \, dx,
\end{equation}
for any $t \in [0, T]$. In order to establish the bound in~\eqref{eq:energy-estimate-LL}, we need to control the derivatives $\partial_x^\alpha F(m)$ with respect to the various terms in the quantity $\Sigma_{\rm LL}^k$ by applying the Leibniz formula and the Moser estimates in Lemma~\ref{lem:moser}. We face the difficulty that the derivative $\partial_x^\alpha F(m)$ contains partial derivatives of order $\ell + 1$ of the function $m$, which cannot be a priori controlled by the quantity $\Sigma_{\rm LL}^\ell$.

In order to by-pass this difficulty, we decompose the derivative $\partial_x^\alpha F(m)$ as
\begin{equation}
\label{arizona}
\partial_x^\alpha F(m) = G^\alpha(m) - 2 \sum_{1 \leq i, j \leq N} \partial_x^\alpha \partial_{ij} \big( \langle \partial_i m, \partial_j m \rangle_{\R^3} \big) m,
\end{equation}
where the function $G^\alpha(m)$ satisfies
\begin{equation}
\label{california}
\| G^\alpha(m)(\cdot, t) \|_{L^2} \leq C_k \big( 1 + \| m_1(\cdot, t) \|_{L^\infty}^2 + \| m_3(\cdot, t) \|_{L^\infty}^2 + \| \nabla m(\cdot, t) \|_{L^\infty}^2 \big) \, \sqrt{\Sigma_{LL}^\ell(t)},
\end{equation}
for any $t \in [0, T]$. Inequality~\eqref{california} is a consequence of the Leibniz formula and the Moser estimates in Lemma~\ref{lem:moser}. The use of theses estimates is allowed by the uniform boundedness of the gradient $\nabla m$, which results from the Sobolev embedding theorem and the assumption $k > N/2 + 1$.

We then introduce the remaining term of the decomposition of $\partial_x^\alpha F(m)$ into~\eqref{arkansas}, and integrate by parts in order to write
\begin{equation}
\label{carolina}
\begin{split}
\int_{\R^N} \Big\langle \partial_t \partial_x^\alpha m, & \partial_x^\alpha \partial_{ij} \big( \langle \partial_i m, \partial_j m \rangle_{\R^3} \big) m \Big\rangle_{\R^3}\\
& = - \int_{\R^N} \partial_x^\alpha \partial_j \big( \langle \partial_i m, \partial_j m \rangle_{\R^3} \big) \Big( \big\langle \partial_t \partial_x^\alpha \partial_i m, m \big\rangle_{\R^3} + \big\langle \partial_t \partial_x^\alpha m, \partial_i m \big\rangle_{\R^3} \Big),
\end{split}
\end{equation}
for any $1 \leq i, j \leq N$. Invoking once again the Leibniz formula and Lemma~\ref{lem:moser}, we directly check that
\begin{equation}
\label{connecticut}
\bigg| \int_{\R^N} \Big( \partial_x^\alpha \partial_j \big( \langle \partial_i m, \partial_j m \rangle_{\R^3} \big) \, \big\langle \partial_t \partial_x^\alpha m, \partial_i m \big\rangle_{\R^3} \Big)(x, t) \, dx \bigg| \leq C_k \, \| \nabla m(\cdot, t) \|_{L^\infty}^2 \, \Sigma_{LL}^\ell(t).
\end{equation}
On the other hand, we can invoke the Landau-Lifshitz equation and the Leibniz formula in order to write
\begin{align*}
\partial_x^\alpha & \partial_j \big( \langle \partial_i m, \partial_j m \rangle_{\R^3} \big) \, \big\langle \partial_t \partial_x^{\alpha^*} m, m \big\rangle_{\R^3}\\
= & - \sum_{\beta \leq \alpha^*} \binom{\alpha^*}{\beta} \partial_x^\alpha \partial_j \big( \langle \partial_i m, \partial_j m \rangle_{\R^3} \big) \, \big\langle \partial_x^\beta m \times \partial_x^{\alpha^* - \beta} \big( \Delta m - \lambda_1 m_1 e_1 - \lambda_3 m_3 e_3 \big), m \big\rangle_{\R^3},
\end{align*}
where $\partial_x^{\alpha^*} := \partial_x^\alpha \partial_i$. For $\beta = 0$, the quantity in the right-hand side of this formula vanishes. This cancellation is the key point in order to infer again from Lemma~\ref{lem:moser} that
\begin{align*}
\bigg| \int_{\R^N} \Big( \partial_x^\alpha \partial_j & \big( \langle \partial_i m, \partial_j m \rangle_{\R^3} \big) \, \big\langle \partial_t \partial_x^{\alpha^*} m, m \big\rangle_{\R^3} \Big)(x, t) \, dx \bigg|\\
\leq & C_k \| \nabla m(\cdot, t) \|_{L^\infty} \, \big( \| \nabla m(\cdot, t) \|_{L^\infty} + \| m_1(\cdot, t) \|_{L^\infty} + \| m_3(\cdot, t) \|_{L^\infty} \big) \, \Sigma_{LL}^\ell(t).
\end{align*}
We finally gather this estimate with~\eqref{arkansas},~\eqref{california},~\eqref{carolina} and~\eqref{connecticut} in order to derive~\eqref{eq:energy-estimate-LL}. This completes the proof of Proposition~\ref{prop:LL-energy-estimate}. \qed

%%%%%%%%%%%%%%%%%%%%%%%%%%%%%%%%%%%%%%%%%%%%%%%%%%%%%%%%%%%%
%%%%%%%%%%%%%%%%%%%%%%%%%%%%%%%%%%%%%%%%%%%%%%%%%%%%%%%%%%%%
\subsection{Proof of Proposition~\ref{prop:LL-diff-control}}
\label{sub:LL-diff-control}
%%%%%%%%%%%%%%%%%%%%%%%%%%%%%%%%%%%%%%%%%%%%%%%%%%%%%%%%%%%%
%%%%%%%%%%%%%%%%%%%%%%%%%%%%%%%%%%%%%%%%%%%%%%%%%%%%%%%%%%%%

We first compute the equation
\begin{equation}
\label{eq:diff-LL}
\partial_t u = - v \times \big( \Delta u - J(u) \big) - u \times \big( \Delta v - J(v) \big),
\end{equation}
where we have set $J(u) := \lambda_1 u_1 e_1 + \lambda_3 u_3 e_3$ and $J(v) := \lambda_1 v_1 e_1 + \lambda_3 v_3 e_3$. Under the assumptions of Proposition~\ref{prop:LL-diff-control}, the time derivative $\partial_t u$ lies in $\boC^0([0, T], L^2(\R^N))$, so that the function $u - u^0$ belongs to the space $\boC^1([0, T], L^2(\R^N))$. Since $u_1^0$ and $u_3^0$ are in $L^2(\R^N)$, the quantity $\gE_{\rm LL}^0$ is well-defined and of class $\boC^1$ on $[0, T]$. When $1 \leq \ell \leq k - 1$, it similarly follows from the smoothness assumptions in Proposition~\ref{prop:LL-diff-control} that the quantities $\gE_{\rm LL}^\ell$ are well-defined and of class $\boC^1$ on $[0, T]$. We now split the proof of their control into three cases according to the value of $\ell$.

\begin{cas}
\label{D1}
$\ell = 0$.
\end{cas}

In view of~\eqref{def:gE0} and~\eqref{eq:diff-LL}, we obtain after integrating by parts,
\begin{align*}
\big[ \gE_{\rm LL}^0 \big]'(t) = \int_{\R^N} \Big( \sum_{j = 1}^N \big\langle \partial_j u - \partial_j u_2^0 \, e_2, v \times \partial_j u & + u \times \partial_j v \big\rangle_{\R^3}\\
& + \big\langle u - u_2^0 \, e_2, v \times J(u) + u \times J(v) \big\rangle_{\R^3} \Big)(x, t) \, dx,
\end{align*}
for any $t \in [0, T]$. The estimate in~\eqref{eq:control-diff-0-LL} is then a consequence of the H\"older inequality and the fact that $|v| \leq 1$.

\begin{cas}
\label{D2}
$\ell = 1$.
\end{cas}

We similarly derive from the definition of the function $\gE_{\rm LL}^1$ that
$$\big[ \gE_{\rm LL}^1 \big]'(t) = \int_{\R^N} \Big( - \big\langle \partial_t u, \Delta u \big\rangle_{\R^3} + \sum_{i = 1}^N \big\langle u \times \partial_i v + v \times \partial_i u, \partial_t (u \times \partial_i v + v \times \partial_i u) \big\rangle_{\R^3} \Big)(x, t) \, dx.$$
In view of~\eqref{eq:diff-LL} and after some integration by parts, the first term in the right-hand side of this formula writes as
$$- \int_{\R^N} \big\langle \partial_t u, \Delta u \big\rangle_{\R^3}(x, t) \, dx = - \sum_{i = 1}^N \int_{\R^N} \big\langle u \times \partial_i v + v \times \partial_i u, \partial_i \Delta u \big\rangle_{\R^3}(x, t) \, dx + \gI_1^1(t),$$
where we have set
$$\gI_1^1(t) := \sum_{i = 1}^N \int_{\R^N} \big\langle \partial_i u, \partial_i v \times J(u) + v \times \partial_i J(u) + u \times \partial_i J(v) \big\rangle_{\R^3}(x, t) \, dx.$$
Therefore, we have
\begin{equation}
\label{delaware}
\big[ \gE_{\rm LL}^1 \big]'(t) = \gI_1^1(t) + \sum_{i = 1}^N \int_{\R^N} \big\langle u \times \partial_i v + v \times \partial_i u, \partial_t (u \times \partial_i v + v \times \partial_i u) - \partial_i \Delta u \big\rangle_{\R^3} \Big)(x, t) \, dx.
\end{equation}

At this stage, we have to compute the time derivative in the right-hand side of~\eqref{delaware}. We first check that the function $v$ is solution to the equation
$$\partial_t v = - v \times \big( \Delta v - J(v) \big) - \frac{1}{4} u \times \big( \Delta u - J(u) \big),$$
In view of~\eqref{eq:diff-LL} and using the identities
\begin{equation}
\label{florida}
\langle u, v \rangle_{\R^3} = 0, \quad {\rm and} \quad |v|^2 + \frac{1}{4} |u|^2 = 1,
\end{equation}
we are led to the formula
\begin{equation}
\label{dakota}
\begin{split}
\partial_t (u \times \partial_i v + v \times \partial_i u) & = \partial_i \Delta u - \partial_i u \, \Big( \frac{1}{4} \langle u, \Delta u \rangle_{\R^3} + \langle v, \Delta v \rangle_{\R^3} \Big) - \partial_i v \, \big( \langle u, \Delta v \rangle_{\R^3} + \langle v, \Delta u \rangle_{\R^3} \big)\\
& + u \, \Big( \langle \partial_i v, \Delta v \rangle_{\R^3} + \frac{1}{4} \langle \partial_i u, \Delta u \rangle_{\R^3} - \langle v, \partial_i \Delta v \rangle_{\R^3} - \frac{1}{4} \langle u, \partial_i \Delta u \rangle_{\R^3} \Big)\\
& + v \, \Big( \langle \partial_i v, \Delta u \rangle_{\R^3} + \langle \partial_i u, \Delta v \rangle_{\R^3} - \langle u, \partial_i \Delta v \rangle_{\R^3} - \langle v, \partial_i \Delta u \rangle_{\R^3} \Big) + \gj_2^1(u, v).
\end{split}
\end{equation}
In this expression, the term depending on the anisotropic vectors $J(u)$ and $J(v)$ is given by
\begin{align*}
\gj_2^1(u, v) := & - \partial_i J(u) + \partial_i u \, \Big( \frac{1}{4} \langle u, J(u) \rangle_{\R^3} + \langle v, J(v) \rangle_{\R^3} \Big) + \partial_i v \, \big( \langle u, J(v) \rangle_{\R^3} + \langle v, J(u) \rangle_{\R^3} \big)\\
& + u \, \Big( - \langle \partial_i v, J(v) \rangle_{\R^3} - \frac{1}{4} \langle \partial_i u, J(u) \rangle_{\R^3} + \langle v, \partial_i J(v) \rangle_{\R^3} + \frac{1}{4} \langle u, \partial_i J(u) \rangle_{\R^3} \Big)\\
& + v \, \Big( - \langle \partial_i v, J(u) \rangle_{\R^3} - \langle \partial_i u, J(v) \rangle_{\R^3} + \langle u, \partial_i J(v) \rangle_{\R^3} + \langle v, \partial_i J(u) \rangle_{\R^3} \Big).
\end{align*}
This anisotropic term is not difficult to estimate, but we have to find a cancellation in the other terms in order to prove the bound in~\eqref{eq:control-diff-1-LL}.

In this direction, we first differentiate the identities in~\eqref{florida} in order to get
$$\langle u, \Delta v \rangle_{\R^3} + \langle \Delta u, v \rangle_{\R^3} = - 2 \sum_{j = 1}^N \langle \partial_j u, \partial_j v \rangle_{\R^3},$$
and
$$\langle v, \Delta v \rangle_{\R^3} + \frac{1}{4} \langle u, \Delta u \rangle_{\R^3} = - |\nabla v|^2 - \frac{1}{4} |\nabla u|^2.$$
Similarly, we have
$$\langle u, \partial_i \Delta v \rangle_{\R^3} + \langle \partial_i \Delta u, v \rangle_{\R^3} = - 2 \sum_{j = 1}^N \big( \langle \partial_{ij} u, \partial_j v \rangle_{\R^3} + \langle \partial_j u, \partial_{ij} v \rangle_{\R^3} \big) - \langle \partial_i u, \Delta v \rangle_{\R^3} - \langle \Delta u, \partial_i v \rangle_{\R^3},$$
and
\begin{align*}
\langle v, \partial_i \Delta v \rangle_{\R^3} + & \frac{1}{4} \langle u, \partial_i \Delta u \rangle_{\R^3}\\
& = - 2 \sum_{j = 1}^N \Big( \langle \partial_{ij} v, \partial_j v \rangle_{\R^3} + \frac{1}{4} \langle \partial_{ij} u, \partial_j u \rangle_{\R^3} \Big) - \langle \partial_i v, \Delta v \rangle_{\R^3} - \frac{1}{4} \langle \partial_i u, \Delta u \rangle_{\R^3}.
\end{align*}
Introducing these identities into~\eqref{dakota}, we get
\begin{equation}
\label{louisiana}
\begin{split}
\partial_t (u \times \partial_i v + v \times \partial_i u) & = \partial_i \Delta u + 2 v \, \sum_{j = 1}^N \partial_j \Big( \langle \partial_i v, \partial_j u \rangle_{\R^3} + \langle \partial_i u, \partial_j v \rangle_{\R^3} \Big) + \gj_2^1(u, v) + \gj_3^1(u, v),
\end{split}
\end{equation}
where we have set
\begin{align*}
\gj_3^1(u, v) := & \Big( |\nabla v|^2 + \frac{1}{4} |\nabla u|^2 \Big) \, \partial_i u + 2 \sum_{j = 1}^N \langle \partial_j u, \partial_j v \rangle_{\R^3} \, \partial_i v\\
& + 2 u \, \bigg( \langle \partial_i v, \Delta v \rangle_{\R^3} + \frac{1}{4} \langle \partial_i u, \Delta u \rangle_{\R^3} + \sum_{j = 1}^N \Big( \langle \partial_{ij} v, \partial_j v \rangle_{\R^3} + \frac{1}{4} \langle \partial_{ij} u, \partial_j u \rangle_{\R^3} \Big) \bigg).
\end{align*}

At this point, we come back to~\eqref{delaware}. We use the cancellation of the terms $\partial_i \Delta u$ in~\eqref{delaware} and~\eqref{louisiana}, and integrate by parts in order to obtain
\begin{equation}
\label{georgia}
\begin{split}
\big[ \gE_{\rm LL}^1 \big]'(t) = & \gI_1^1(t) + \sum_{i = 1}^N \int_{\R^N} \big\langle u \times \partial_i v + v \times \partial_i u, \gj_2^1(u, v) + \gj_3^1(u, v) \big\rangle(x, t) \, dx\\
& - 2 \sum_{j = 1}^N \sum_{i = 1}^N \int_{\R^N} \Big( \big( \langle \partial_j u \times \partial_i v, v \rangle_{\R^3} + \langle u \times \partial_{ij} v, v \rangle_{\R^3} + \langle u \times \partial_i v, \partial_j v \rangle_{\R^3} \big) \times\\
& \quad \quad \quad \quad \quad \quad \quad \quad \times \big( \langle \partial_i v, \partial_j u \rangle_{\R^3} + \langle \partial_i u, \partial_j v \rangle_{\R^3} \big) \Big)(x, t) \, dx.
\end{split}
\end{equation}
With this formula at hand, we can prove the bound in~\eqref{eq:control-diff-1-LL}. Indeed, we first check that the last terms in~\eqref{georgia} satisfy
\begin{align*}
\int_{\R^N} \bigg| \big( \langle \partial_j u \times \partial_i v, v \rangle_{\R^3} + & \langle u \times \partial_{ij} v, v \rangle_{\R^3} + \langle u \times \partial_i v, \partial_j v \rangle_{\R^3} \big) \, \big( \langle \partial_i v, \partial_j u \rangle_{\R^3} + \langle \partial_i u, \partial_j v \rangle_{\R^3} \big) \bigg|\\
\leq & K \| \nabla v \|_{L^\infty} \, \| \nabla u \|_{L^2} \, \Big( \| \nabla v \|_{L^\infty} \, \| \nabla u \|_{L^2} + \big( \| \nabla v \|_{L^\infty}^2 + \| v \|_{\dot{H}^2} \big) \, \| u \|_{L^\infty} \Big), 
\end{align*}
where $K$ refers, here as in the sequel, to a universal constant. Coming back to the definitions of the functions $\gj_2^1(u, v)$ and $\gj_3^1(u, v)$ and using the inequalities $|u| \leq 2$ and $|v| \leq 1$, we check that
$$\| \gj_2^1(u, v) \|_{L^2} \leq K \big( \| \nabla u \|_{L^2} + \| \nabla v \|_{L^2} \, \| u \|_{L^\infty} \big),$$
and
$$\| \gj_3^1(u, v) \|_{L^2} \leq K \Big( \big( \| \nabla v \|_{L^\infty}^2 + \| \nabla u \|_{L^\infty}^2 \big) \, \| \nabla u \|_{L^2} + \big( \| \nabla v \|_{L^\infty} \, \| v \|_{\dot{H}^2} + \| \nabla u \|_{L^\infty} \, \| u \|_{\dot{H}^2} \big) \, \| u \|_{L^\infty} \Big).$$
This provides the estimate
\begin{align*}
& \int_{\R^N} \big\langle u \times \partial_i v + v \times \partial_i u, \gj_2^1(u, v) + \gj_3^1(u, v) \big\rangle \leq K \big( \| \nabla u \|_{L^2} + \| \nabla v \|_{L^2} \, \| u \|_{L^\infty} \big) \times\\
& \times \Big( \big( 1 + \| \nabla v \|_{L^\infty}^2 + \| \nabla u \|_{L^\infty}^2 \big) \, \| \nabla u \|_{L^2} + \big( \| \nabla v \|_{L^2} + \| \nabla v \|_{L^\infty} \, \| v \|_{\dot{H}^2} + \| \nabla u \|_{L^\infty} \, \| u \|_{\dot{H}^2} \big) \, \| u \|_{L^\infty} \Big).
\end{align*}
Finally, we bound the quantity $\gI_1^1(t)$ by
$$\big| \gI_1^1(t) \big| \leq \big( \| \nabla u(\cdot, t) \|_{L^2} + \| \nabla v(\cdot, t) \|_{L^2} \, \| u(\cdot, t) \|_{L^\infty} \big) \, \| \nabla u(\cdot, t) \|_{L^2}.$$

Gathering all these estimates of~\eqref{georgia}, and recalling that
$$\| \nabla u \|_{L^\infty} + \| \nabla v \|_{L^\infty} \leq K \big( \| \nabla m \|_{L^\infty} + \| \nabla \tilde{m} \|_{L^\infty} \big),$$
and
$$\| u \|_{\dot{H}^j} + \| v \|_{\dot{H}^j} \leq K \big( \| m \|_{\dot{H}^j} + \| \tilde{m} \|_{\dot{H}^j} \big),$$
for $1 \leq j \leq 2$, we obtain~\eqref{eq:control-diff-1-LL}.

\begin{cas}
\label{D3}
$2 \leq \ell \leq k - 1$.
\end{cas}

The proof is similar to the case $\ell = 1$. We now integrate by parts in order to obtain the formula
\begin{align*}
\big[ \gE_{\rm LL}^\ell \big]'(t) = \sum_{|\alpha| = \ell - 2} \int_{\R^N} \Big\langle \partial_t \partial_x^\alpha u, \partial_x^\alpha \Big( \partial_{tt} u + \Delta^2 u & - (\lambda_1 + \lambda_3) \big( \Delta u_1 e_1 + \Delta u_3 e_3 \big)\\
& + \lambda_1 \lambda_3 \big( u_1 e_1 + u_3 e_3 \big) \Big) \Big\rangle_{\R^3}(x, t) \, dx,
\end{align*}
for any $t \in [0, T]$. Coming back to~\eqref{eq:second-LL},~\eqref{usa} and~\eqref{arizona}, we obtain
\begin{equation}
\label{colorado}
\big[ \gE_{\rm LL}^\ell \big]'(t) = \gI_1^\ell(t) + \gI_2^\ell(t),
\end{equation}
where we have set
$$\gI_1^\ell(t) := \sum_{|\alpha| = \ell - 2} \int_{\R^N} \big\langle \partial_t \partial_x^\alpha u, G^\alpha(\tilde{m}) - G^\alpha(m) \big\rangle_{\R^3}(x, t) \, dx,$$
and
$$\gI_2^\ell(t) := 2 \sum_{|\alpha| = \ell - 2} \sum_{1 \leq i, j \leq N} \int_{\R^N} \big\langle \partial_t \partial_x^\alpha u, \partial_x^\alpha \partial_{ij} \big( \langle \partial_i m, \partial_j m \rangle_{\R^3} \big) m - \partial_x^\alpha \partial_{ij} \big( \langle \partial_i \tilde{m}, \partial_j \tilde{m} \rangle_{\R^3} \big) \tilde{m} \big\rangle_{\R^3}(x, t) dx.$$

We now deal with the quantity $\gI_2^\ell$, which is the more difficult term to control in order to derive the bound in~\eqref{eq:control-diff-l-LL}. Expressing the integrand in the formula for $\gI_2^\ell(t)$ in terms of the functions $u$ and $v$, we get
\begin{align*}
\partial_x^\alpha \partial_{ij} \big( \langle \partial_i \tilde{m}, \partial_j \tilde{m} \rangle_{\R^3} \big) \tilde{m} - \partial_x^\alpha \partial_{ij} \big( \langle \partial_i m, & \partial_j m \rangle_{\R^3} \big) m\\
= & \partial_x^\alpha \partial_{ij} \big( \langle \partial_i u, \partial_j v \rangle_{\R^3} \big) v + \partial_x^\alpha \partial_{ij} \big( \langle \partial_i v, \partial_j u \rangle_{\R^3} \big) v\\
& + \partial_x^\alpha \partial_{ij} \big( \langle \partial_i v, \partial_j v \rangle_{\R^3} \big) u + \frac{1}{4} \partial_x^\alpha \partial_{ij} \big( \langle \partial_i u, \partial_j u \rangle_{\R^3} \big) u.
\end{align*}

Concerning the last two terms in this identity, we can rely on the Moser estimates in Lemma~\ref{lem:moser} in order to obtain
$$\bigg| \int_{\R^N} \big\langle \partial_t \partial_x^\alpha u, \partial_x^\alpha \partial_{ij} \big( \langle \partial_i v, \partial_j v \rangle_{\R^3} \big) u \big\rangle_{\R^3} \bigg| \leq C \| \nabla v \|_{L^\infty} \, \| \nabla v \|_{\dot{H}^\ell} \, \| u \|_{L^\infty} \, \| \partial_t \partial_x^\alpha u \|_{L^2},$$
and
$$\bigg| \int_{\R^N} \big\langle \partial_t \partial_x^\alpha u, \partial_x^\alpha \partial_{ij} \big( \langle \partial_i u, \partial_j u \rangle_{\R^3} \big) u \big\rangle_{\R^3} \bigg| \leq C \| \nabla u \|_{L^\infty} \, \| \nabla u \|_{\dot{H}^\ell} \, \| u \|_{L^\infty} \, \| \partial_t \partial_x^\alpha u \|_{L^2},$$
where $C$ refers, here as in the sequel, to a positive number depending only on $k$.

The estimates of the two other terms follow from integrating by parts and using~\eqref{eq:diff-LL}. Indeed, this provides the identity
\begin{align*}
\int_{\R^N} \big\langle \partial_t \partial_x^\alpha u, & \partial_x^\alpha \partial_{ij} \big( \langle \partial_i u, \partial_j v \rangle_{\R^3} \big) v \big\rangle_{\R^3}\\
& = \int_{\R^N} \partial_x^\alpha \partial_j \big( \langle \partial_i u, \partial_j v \rangle_{\R^3} \big) \, \big\langle \partial_x^\alpha \partial_i \big( u \times (\Delta v - J(v)) + v \times (\Delta u - J(u)) \big), v \big\rangle_{\R^3}.
\end{align*}
We then directly derive the bound
\begin{align*}
\bigg| \int_{\R^N} \partial_x^\alpha \partial_j \big( \langle \partial_i u, \partial_j v \rangle_{\R^3} \big) \, \big\langle \partial_x^\alpha \partial_i \big( & u \times (\Delta v - J(v)) \big), v \big\rangle_{\R^3} \bigg| \leq C \big( \| u \|_{L^\infty} \| \nabla v \|_{\dot{H}^\ell} + \| \nabla v \|_{L^\infty} \| u \|_{\dot{H}^\ell} \big) \times\\
& \times \Big( \big( \| \nabla v \|_{\dot{H}^{\ell - 2}} + \| \nabla v \|_{\dot{H}^\ell} \big) \, \| u \|_{L^\infty} + \| \nabla v \|_{L^\infty} \, \| u \|_{\dot{H}^\ell} + \| u \|_{\dot{H}^{\ell - 1}} \Big).
\end{align*}
On the other hand, we can invoke the Leibniz formula in order to write
\begin{align*}
\big\langle \partial_x^\alpha \partial_i \big( v \times (\Delta u - J(u)) \big), v \big\rangle_{\R^3} = \sum_{\beta \leq \alpha^*} \binom{\alpha^*}{\beta} \big\langle \partial_x^\beta v \times \partial_x^{\alpha^* - \beta} (\Delta u - J(u)), v \big\rangle_{\R^3},
\end{align*}
with $\partial_x^{\alpha^*} = \partial_x^\alpha \partial_i$, as before. As in the proof of Proposition~\ref{prop:LL-energy-estimate}, we again observe a cancellation for $\beta = 0$. This is enough to guarantee that
\begin{align*}
\bigg| & \int_{\R^N} \partial_x^\alpha \partial_j \big( \langle \partial_i u, \partial_j v \rangle_{\R^3} \big) \, \big\langle \partial_x^\alpha \partial_i \big( v \times (\Delta u - J(u)) \big), v \big\rangle_{\R^3} \bigg| \leq C \big( \| u \|_{L^\infty} \| \nabla v \|_{\dot{H}^\ell} + \| \nabla v \|_{L^\infty} \| u \|_{\dot{H}^\ell} \big) \times\\
& \times \Big( \big( \| \nabla v \|_{\dot{H}^{\ell - 2}} + \| \nabla v \|_{\dot{H}^\ell} \big) \, \| u \|_{L^\infty} + \| \nabla v \|_{L^\infty} \, \| u \|_{\dot{H}^\ell} + \| u \|_{\dot{H}^{\ell - 1}} \Big).
\end{align*}
Collecting all these estimates leads to the inequality
\begin{equation}
\label{hawai}
\begin{split}
|\gI_2^\ell & (t)| \leq C \Big( \big( \| \nabla u(\cdot, t) \|_{L^\infty} \, \| \nabla u(\cdot, t) \|_{\dot{H}^\ell} + \| \nabla v(\cdot, t) \|_{L^\infty} \, \| \nabla v(\cdot, t) \|_{\dot{H}^\ell} \big) \, \| u(\cdot, t) \|_{L^\infty} \, \| \partial_t u(\cdot, t) \|_{\dot{H}^{\ell - 2}}\\
& + \big( \| u(\cdot, t) \|_{L^\infty} \, \big( \| \nabla v(\cdot, t) \|_{\dot{H}^{\ell - 2}} + \| \nabla v(\cdot, t) \|_{\dot{H}^\ell} \big) + \| \nabla v(\cdot, t) \|_{L^\infty} \, \| u(\cdot, t) \|_{\dot{H}^\ell} + \| u(\cdot, t) \|_{\dot{H}^{\ell - 1}} \big) \times\\
& \times \big( \| u(\cdot, t) \|_{L^\infty} \, \| \nabla v(\cdot, t) \|_{\dot{H}^\ell} + \| \nabla v(\cdot, t) \|_{L^\infty} \, \| u(\cdot, t) \|_{\dot{H}^\ell} \big) \Big).
\end{split}
\end{equation}

We next turn to the quantity $\gI_1^\ell$, which we simply bound by
\begin{equation}
\label{idaho}
\gI_1^\ell(t) \leq \sum_{|\alpha| \leq \ell - 2} \big\| \partial_t \partial_x^\alpha u(\cdot, t) \big\|_{L^2} \, \big\| G^\alpha(\tilde{m})(\cdot, t) - G^\alpha(m)(\cdot, t) \big\|_{L^2}.
\end{equation}
Coming back to the definition of the nonlinearity $G^\alpha$ and using as before the Moser estimates in Lemma~\ref{lem:moser}, we can compute
\begin{align*}
\big\| G^\alpha(\tilde{m}) - G^\alpha(m) \big\|_{L^2} \leq & C \Big( \big( \| v \|_{\dot{H}^\ell} + \| \nabla u \|_{L^\infty} \, \| \nabla u \|_{\dot{H}^\ell} + \| \nabla v \|_{L^\infty} \, \| \nabla v \|_{\dot{H}^\ell} \big) \, \| u \|_{L^\infty}\\
& + \big( 1 + \| \nabla u \|_{L^\infty}^2 + \| \nabla v \|_{L^\infty}^2 \big) \, \| u \|_{\dot{H}^\ell} + \delta_{\ell \neq 2} \big( \| v \|_{\dot{H}^{\ell - 2}} \, \| u \|_{L^\infty} + \| u \|_{\dot{H}^{\ell - 2}} \big)\\
& + \delta_{\ell = 2} \big( \big( \| v_1 \|_{L^2} + \| v_3 \|_{L^2} \big) \, \| u \|_{L^\infty} + \| u_1 \|_{L^2} + \| u_3 \|_{L^2} \Big).
\end{align*}
Combining with~\eqref{colorado},~\eqref{hawai} and~\eqref{idaho}, we obtain~\eqref{eq:control-diff-l-LL}. This ends the proof of Proposition~\ref{prop:LL-diff-control}. \qed

%%%%%%%%%%%%%%%%%%%%%%%%%%%%%%%%%%%%%%%%%%%%%%%%%
%%%%%%%%%%%%%%%%%%%%%%%%%%%%%%%%%%%%%%%%%%%%%%%%%
\subsection{Proof of Theorem~\ref{thm:LL-Cauchy}}
\label{sub:LL-Cauchy}
%%%%%%%%%%%%%%%%%%%%%%%%%%%%%%%%%%%%%%%%%%%%%%%%%
%%%%%%%%%%%%%%%%%%%%%%%%%%%%%%%%%%%%%%%%%%%%%%%%%

The construction of the solutions splits into three parts. We first consider an initial datum $m^0 \in \boE(\R^N)$, with $\nabla m^0 \in H^\infty(\R^N)$, and we construct the unique corresponding maximal solution $m$ to~\eqref{LL}. We next establish that the corresponding flow map is well-defined and locally Lipschitz continuous from $\boE^{k_N}(\R^N)$ to spaces of the form $\boC^0([0, T], \boE^{k_N - 1}(\R^N))$. Here, the notation $k_N$ refers to the smallest integer such that $k_N > N/2 + 1$. In particular, we are allowed to extend uniquely the flow to the whole set $\boE^{k_N}(\R^N)$. We finally check that the corresponding solutions to~\eqref{LL} satisfy all the statements in Theorem~\ref{thm:LL-Cauchy}.

\setcounter{step}{0}
\begin{step}
\label{L1}
Construction of smooth solutions to~\eqref{LL}.
\end{step}

Let $m^0 \in \boE(\R^N)$, with $\nabla m^0 \in H^\infty(\R^N)$. Note that the existence of such initial conditions is a direct consequence of Lemma~\ref{lem:dens-smooth}.

In order to construct a solution $m$ corresponding to this initial datum, we rely on the bounds in Proposition~\ref{prop:LL-energy-estimate}. Due to the Sobolev embedding theorem, the quantity $\Sigma_{\rm LL}^{k_N}$ corresponding to a smooth enough solution $m : \R^N \times [0, T_*] \to \S^2$ satisfies the differential inequality
$$\big[ \Sigma_{\rm LL}^{k_N} \big]' (t) \leq C_{k_N} \, \big( 1 + \Sigma_{\rm LL}^{k_N}(t)^2 \big) \, \Sigma_{\rm LL}^{k_N}(t),$$
for any $t \in [0, T_*]$. Here as in the sequel, the notation $C_k$ refers to a positive number, depending only on $k$. As a consequence, we obtain the estimate
\begin{equation}
\label{indiana}
\Sigma_{\rm LL}^{k_N}(t) \leq \frac{\Sigma_{\rm LL}^{k_N}(0) \, e^{C_{k_N} t}}{\big( 1 + \Sigma_{\rm LL}^{k_N}(0)^2 - \Sigma_{\rm LL}^{k_N}(0)^2 \, e^{2 C_{k_N} t} \big)^\frac{1}{2}},
\end{equation}
when
$$t < T_*^{k_N} := \frac{1}{2 C_{k_N}} \, \ln \Big( \frac{1 + \Sigma_{\rm LL}^{k_N}(0)^2}{\Sigma_{\rm LL}^{k_N}(0)^2} \Big).$$

We next iterate this argument for any integer $k > k_N$. Given a positive number $0 < T < \min \{ T_*, T_*^{k_N} \}$, we similarly derive from Proposition~\ref{prop:LL-energy-estimate} that
$$\big[ \Sigma_{\rm LL}^k \big]' (t) \leq C_k \, \big( 1 + \Sigma_{\rm LL}^{k_N}(t)^2 \big) \, \Sigma_{\rm LL}^k(t),$$
for any $t \in [0, T]$. We then infer from~\eqref{indiana} that
\begin{equation}
\label{missouri}
\Sigma_{\rm LL}^k(t) \leq \frac{\big( 1 + \Sigma_{\rm LL}^{k_N}(0)^2 \big) \, \Sigma_{\rm LL}^k(0) \, e^{C_k t}}{1 + \Sigma_{\rm LL}^{k_N}(0)^2 - \Sigma_{\rm LL}^{k_N}(0)^2 \, e^{2 C_{k_N} T}}.
\end{equation}
In view of the definition of the quantity $\Sigma_{\rm LL}^k$, the functions $m$ and $\partial_t m$ are uniformly bounded in $Z^k(\R^N)$, respectively $H^{k - 2}(\R^N)$, on the time interval $[0, T]$.

Arguing as in~\cite{SulSuBa1}, we check that the a priori bounds in~\eqref{indiana} and~\eqref{missouri} remain available when the equation is discretized according to a finite difference scheme. The existence and uniqueness of discretized solutions follow from the standard theory of ordinary differential equations. Classical weak compactness and local strong compactness results, as well as a standard diagonal argument, provide the existence of a maximal solution $m : \R^N \times[0, T_{\max}) \to \S^2$ to~\eqref{LL} with initial datum $m^0$. This solution $m$ is in $L^\infty([0, T], Z^k(\R^N))$, with $\partial_t m \in L^\infty([0, T], H^{k - 2}(\R^N))$, for any number $0 \leq T < T_{\max}$ and any integer $k \geq k_N$. In particular, it is a smooth solution to~\eqref{LL}.

Note that its maximal time of existence $T_{\max}$ does not depend on the integer $k$. This follows from~\eqref{missouri}, which guarantees that its maximal time of existence $T_{\max}^k$ as a solution in $L^\infty([0, T], Z^k(\R^N))$, with $\partial_t m \in L^\infty([0, T], H^{k - 2}(\R^N))$, for any $0 \leq T < T_{\max}^k$, is characterized by the condition
$$\lim_{t \to T_{\max}^k} \Sigma_{\rm LL}^{k_N}(t) = \infty,$$
if this maximal time is finite. As a consequence of this property, we obtain
$$T_{\max}^k = T_{\max}^{k_N} := T_{\max},$$
for any $k \geq k_N$.

Note also that we are allowed to invoke again Proposition~\ref{prop:LL-energy-estimate} in order to prove the bound
\begin{equation}
\label{iowa}
\Sigma_{\rm LL}^k(t) \leq C_k \, \Sigma_{\rm LL}^k(0) \, e^{\int_0^t \big( 1 + \| \nabla m(\cdot, s) \|_{L^\infty}^2 \big) \, ds},
\end{equation}
for any $0 \leq t \leq T_{\max}$ and any $k \geq k_N$. Therefore, if the maximal time of existence $T_{\max}$ is finite, it satisfies the condition
$$\int_0^{T_{\max}} \| \nabla m(\cdot, t) \|_{L^\infty}^2 \, dt = \infty.$$

We finally turn to the question of the uniqueness of this solution. We fix an integer $k \geq k_N$. Given an initial condition $\tilde{m}^0 \in \boE(\R^N)$, with $\nabla \tilde{m}_0 \in H^\infty(\R^N)$, we denote by $\tilde{m} : \R^N \times [0, \tilde{T}_{\max}) \to \S^2$ a corresponding smooth solution to~\eqref{LL}. Set $T_* = \min \{ T_{\max}, \tilde{T}_{\max} \}$. The solutions $m$ and $\tilde{m}$ belong to $\boC^0([0, T_*), \boE^{k_N + 1}(\R^N))$, with $(\partial_t \tilde{m}, \partial_t m) \in \boC^0([0, T_*), H^{k_N - 1}(\R^N))^2$. Therefore, we are allowed to invoke Proposition~\ref{prop:LL-diff-control} in order to find a positive number $C_k$ for which the difference $u := \tilde{m} - m$ satisfies
$$\big[ \gS_{\rm LL}^{k - 1} \big]'(t) \leq C_k \big( 1 + \tilde{\Sigma}_{\rm LL}^k(s) + \Sigma_{\rm LL}^k(s) \big)^3 \, \big( \gS_{\rm LL}^{k - 1}(t) + \| u(\cdot, t) \|_{L^\infty}^2 + \| \nabla u_2^0 \|_{L^2}^2 \big),$$
for any $0 \leq t < T_*$. On the other hand, we infer from the Sobolev embedding theorem that
$$\| u(\cdot, t) - u_2^0 \, e_2 \|_{L^\infty}^2 \leq C_k \, \gS_{\rm LL}^{k - 1}(t).$$
This is enough to obtain the bound
\begin{align*}
\max_{t \in [0, T]} \gS_{\rm LL}^{k - 1}(t) \leq \gS_{\rm LL}^{k - 1}(0) \, & e^{\int_0^t C_k \, \big( 1 + \tilde{\Sigma}_{\rm LL}^k(s) + \Sigma_{\rm LL}^k(s) \big)^3 \, ds}\\
& + \big( \| u_2^0 \|_{L^\infty}^2 + \| \nabla u_2^0 \|_{L^2}^2 \big) \, \Big( e^{\int_0^t C_k \, \big( 1 + \tilde{\Sigma}_{\rm LL}^k(s) + \Sigma_{\rm LL}^k(s) \big)^3 \, ds} - 1 \Big),
\end{align*}
for any $0 \leq T < T_*$. Here, the quantity $\tilde{\Sigma}_{\rm LL}^k$ is defined with respect to the solution $\tilde{m}$. In view of the definition of the quantity $\gS_{\rm LL}^{k - 1}$, this provides the estimate
\begin{equation}
\label{kansas}
\begin{split}
\max_{t \in [0, T]} \Big( \| \nabla u(\cdot, t) & \|_{H^{k - 2}}^2 + \| u(\cdot, t) - u_2^0 \, e_2 \|_{L^2}^2 \Big)\\
\leq & C_k \big( \| \nabla u^0 \|_{H^{k - 2}}^2 + \| u_1^0 \|_{L^2}^2 + \| u_3^0 \|_{L^2}^2 \big) \, e^{\int_0^t C_k \, \big( 1 + \tilde{\Sigma}_{\rm LL}^k(s) + \Sigma_{\rm LL}^k(s) \big)^3 \, ds}\\
& + \| u_2^0 \|_{L^\infty}^2 \, \Big( e^{\int_0^t C_k \, \big( 1 + \tilde{\Sigma}_{\rm LL}^k(s) + \Sigma_{\rm LL}^k(s) \big)^3 \, ds} - 1 \Big).
\end{split}
\end{equation}
We conclude that the difference $u$ identically vanishes on $[0, T_*)$ when $\tilde{m}^0 = m^0$. This proves the uniqueness of the solution.

\begin{step}
\label{L2}
Unique extension of the flow map.
\end{step}

Given an integer $k \geq k_N$, we now consider an initial datum $m^0 \in \boE^k(\R^N)$. Lemma~\ref{lem:dens-smooth} provides the existence of a sequence of initial conditions $m_n^0 \in \boE(\R^N)$, with $\nabla m_n^0 \in H^\infty(\R^N)$, such that
$$m_n^0 - m^0 \to 0 \quad {\rm in} \ H^k(\R^N),$$
as $n \to \infty$. Let $m_n$ be the corresponding smooth solutions to~\eqref{LL} constructed in Step~\ref{L1} above. Combining this convergence with~\eqref{missouri}, we check that the quantities $\Sigma_{\rm LL}^{k, n}$ defined with respect to the solutions $m_n$ are bounded on the time intervals $[0, T]$ for any $0 \leq T < T_*^{k_N}$, uniformly with respect to $n$ large enough. As a consequence of~\eqref{kansas}, the sequence $(m_n)_{n \in \N}$ is a Cauchy sequence in $\boC^0([0, T], \boE^{k - 1}(\R^N))$.

Let us denote by $m$ its limit. Note first that this limit is independent on the choice of the sequence $(m_n^0)_{n \in \N}$. Note also that it is an (at least) weak solution to~\eqref{LL} with initial datum $m^0$ due to the Sobolev embedding theorem of $H^{k - 1}(\R^N)$ into $\boC^0(\R^N)$. Actually, it is the unique solution in $\boC^0([0, T], \boE^{k - 1}(\R^N))$ to~\eqref{LL} with initial datum $m^0$, which is a limit of smooth solutions to~\eqref{LL}. Finally, since the quantities $\Sigma_{\rm LL}^{k, n}$ are bounded on $[0, T]$, uniformly with respect to $n$, so is the quantity $\Sigma_{\rm LL}^k$. In particular, the function $m$ belongs to $L^\infty([0, T], \boE^k(\R^N))$, with $\partial_t m \in L^\infty([0, T], H^{k - 2}(\R^N))$. Moreover, since $k > N/2 + 1$, it follows from the Sobolev embedding theorem and standard interpolation arguments that
$$\nabla m_n \to \nabla m \quad {\rm in} \ \boC^0([0, T] \times \R^N),$$
as $n \to \infty$.

Concerning the maximal time of existence of this solution, we denote by $T_*^k$ the supremum of the positive times $T$ for which there exists a sequence of initial conditions $m_n^0 \in \boE(\R^N)$, with $\nabla m_n^0 \in H^\infty(\R^N)$, such that the corresponding solutions $m_n$ are well-defined in $L^\infty([0, T], \boE^k(\R^N))$, with $\partial_t m_n \in L^\infty([0, T], H^{k - 2}(\R^N))$, and satisfy
\begin{equation}
\label{nevada}
m_n^0 \to m^0 \ {\rm in} \ \boE^k(\R^N), \quad m_n \to m \ {\rm in} \ \boC^0([0, T], \boE^{k - 1}(\R^N)), \quad {\rm and} \quad \nabla m_n \to \nabla m \ {\rm in} \ \boC^0([0, T] \times \R^N),
\end{equation}
as $n \to \infty$.

We first claim that the solution $m$ in this statement is uniquely defined on the time interval $[0, T_*^k)$. When $T_*^k \leq T_*^{k_N}$, this follows from the previous construction. When $T_*^k > T_*^{k_N}$, we fix a number $0 < T < T_*^k$ and consider two sequences of smooth solutions $m_n$ and $\tilde{m_n}$, which satisfy the properties in~\eqref{nevada} for two possible solutions $m$ and $\tilde{m}$. The solutions $m_n$ and $\tilde{m_n}$ satisfy the bound in~\eqref{iowa} on $[0, T]$, and the left-hand side in this bound is uniformly bounded with respect to $n$ due to the convergences in~\eqref{nevada}. As a consequence of~\eqref{kansas}, the sequences $m_n$ and $\tilde{m}_n$ own a common limit $m = \tilde{m}$ in $\boC^0([0, T], \boE^{k - 1}(\R^N))$. This proves the uniqueness of the solution $m$ satisfying the properties in~\eqref{nevada}.

Our goal is now to establish that either $T_*^k = \infty$, or
$$I_*^k := \int_0^{T_*^k} \| \nabla m(\cdot, t) \|_{L^\infty}^2 \, dt = \infty.$$
Note first that $T_*^k$ is well-defined and positive due to the inequality $T_*^k \geq T_*^{k_N}$. We now argue by contradiction assuming that $T_*^k$ and the integral $I_*^k$ are finite. We again fix a number $0 < T < T_*^k$, and consider a sequence of smooth solutions $m_n$, which satisfy the properties in~\eqref{nevada}. Invoking~\eqref{iowa} and~\eqref{nevada} as before, we check that the quantities $\Sigma_{\rm LL}^{k, n}$ are bounded on $[0, T]$ by a positive number $\Sigma_*$ depending only on $\Sigma_{\rm LL}^k(0)$ and $I_*^k$. As a consequence of~\eqref{indiana} and~\eqref{missouri}, we can extend the solutions $m_n$ on a time interval of the form $[T, T + \tau_*]$, where the positive number $\tau_*$ only depends on $\Sigma_*$. Moreover, due to~\eqref{kansas}, the sequence $(m_n)_{n \in \N}$ is a Cauchy sequence in $\boC^0([0, T + \tau_*], \boE^{k - 1}(\R^N))$, and we can check as before that it satisfies the properties in~\eqref{nevada} for a solution $m_*$ to~\eqref{LL}, which is equal to $m$ due to the previous unique determination of this solution. Applying this argument to $T = T_*^k - \tau/2$ leads to a contradiction with the definition of the maximal time $T_*^k$. Hence, either $T_* = \infty$, or the integral $I_*^k$ is infinite.

Note finally that, due to this characterization, the maximal time of existence $T_*^k$ does not depend on the possible choice of the integer $k$. We denote by $T_{\max}$ this maximal time in the sequel.

\begin{step}
\label{L3}
Conclusion of the proof of Theorem~\ref{thm:LL-Cauchy}.
\end{step}

Let $k \geq k_N$. Step~\ref{L2} above provides the existence of a unique solution $m: \R^N \times [0, T_{\max}) \to \S^2$ to~\eqref{LL} corresponding to an initial datum $m^0$, which is the limit (according to the properties in~\eqref{nevada}) of smooth solutions to~\eqref{LL}. This solution is in $\boC^0([0, T_{\max}), \boE^{k - 1}(\R^N))$. Its maximal time of existence $T_{\max}$ satisfies the statement $(ii)$ in Theorem~\ref{thm:LL-Cauchy}.

Concerning statement $(i)$, we consider a sequence of smooth solutions $m_n$ converging to $m$ on a time interval $[0, T]$, with $0 < T < T_{\max}$. We then combine as before~\eqref{iowa} and~\eqref{nevada} in order to check that the quantities $\Sigma_{\rm LL}^{k, n}$ are bounded on $[0, T]$, uniformly with respect to $n$. Statement $(i)$ then follows from a standard weak compactness argument.

Statement $(iv)$ is a direct consequence of the previous construction of the solution $m$, while the conservation of the energy in $(v)$ results from a standard density argument. This property is indeed satisfied by smooth solutions in view of Proposition~\ref{prop:LL-energy-estimate}.

Concerning the local Lipschitz continuity of the flow map in statement $(iii)$, we fix a solution $m: \R^N \times [0, T_{\max}) \to \S^2$ with initial condition $m^0$ and a number $0 < T < T_{\max}$. We set
$$\Sigma_T := 2 + \Sigma_{\rm LL}^k(0) \, \Big( 1 + C_k \, e^{\int_0^T \big( 1 + \| \nabla m(\cdot, s) \|_{L^\infty}^2 \big) \, ds} \Big),$$
where $C_k$ is the constant in the right-hand side of~\eqref{iowa}. We notice that the inequality in~\eqref{indiana} remains available (for smooth solutions) when $k_N$ is replaced by $k$ for a possibly different positive number $C_k$, and we fix a positive number $\tau_k$ such that
$$\frac{(\Sigma_T - 1) \, e^{C_k \tau_k}}{\big( 1 + (\Sigma_T - 1)^2 - (\Sigma_T - 1)^2 \, e^{2 C_k \tau_k} \big)^\frac{1}{2}} \leq \Sigma_T,$$
for this further number $C_k$. Here, the number $\tau_k$ is tailored such that, if the quantity $\bar{\Sigma}_{\rm LL}^k(0)$ corresponding to a smooth solution $\bar{m}$ is less than $\Sigma_T - 1$, then the quantity $\bar{\Sigma}_{\rm LL}^k(t)$ is bounded by $\Sigma_T$ on $[0, \tau_k]$.

We finally introduce a sequence of smooth solutions $m_n : \R^N \times [0, T] \to \S^2$, which satisfy the properties in~\eqref{nevada}. Due to~\eqref{iowa}, we can also assume that the corresponding quantities $\Sigma_{\rm LL}^{k, n}$ satisfy the bound
$$\max_{t \in [0, T]} \Sigma_{\rm LL}^{k, n}(t) \leq \Sigma_T.$$
We are now in position to establish the local Lipschitz continuity of the flow.

Given a positive number $R$, we take an initial datum $\tilde{m}^0 \in \boE(\R^N)$, with $\nabla \tilde{m}^0 \in H^\infty(\R^N)$, such that
$$\big\| \tilde{m}^0 - m \big\|_{Z^k} \leq R,$$
and consider the corresponding smooth solution $\tilde{m} : \R^N \times [0, \tilde{T}_{\max}) \to \S^2$. For $R$ small enough, we have
$$\tilde{\Sigma}_{\rm LL}^k(0) \leq \Sigma_T - 1.$$
In view of~\eqref{iowa} (with $k_N$ replaced by $k$), we infer that $\tilde{T}_{\max} \geq \tau_k$, and that the quantity $\tilde{\Sigma}_{\rm LL}^k(t)$ is bounded by $\Sigma_ T$ on $[0, \tau_k]$. Invoking~\eqref{kansas}, we next find a positive number $\Lambda_T$, depending only on $k$, $T$ and $\Sigma_T$, such that
\begin{equation}
\label{ohio}
\begin{split}
\max_{t \in [0, \tau_k]} \| \tilde{m}(\cdot, t) - m_n(\cdot, t) \|_{Z^{k - 1}} \leq \Lambda_T \| \tilde{m}^0 - m_n^0 \|_{Z^{k - 1}} \leq \Lambda_T R,
\end{split}
\end{equation}
for any $n \in \N$. Taking the limit $n \to \infty$, this inequality remains true for the difference $\tilde{m} - m$. As a consequence of the Sobolev embedding theorem and standard interpolation theory, we next find two positive numbers $A_k$ and $\alpha_k$, depending only on $k$, such that
\begin{equation}
\label{oregon}
\max_{t \in [0, \tau_k]} \| \nabla \tilde{m}(\cdot, t) - \nabla m(\cdot, t) \|_{L^\infty} \leq A_k \Lambda_T^{\alpha_k} R^{\alpha_k} \Sigma_T^{1 - \alpha_k}.
\end{equation}
We finally come back to~\eqref{iowa} in order to obtain
$$\tilde{\Sigma}_k(t) \leq C_k \tilde{\Sigma}_k(0) e^{\int_0^t \big( 1 + \| \nabla m(\cdot, s) \|_{L^\infty}^2 + 2 A_k \Lambda_T^{\alpha_k} R^{\alpha_k} \Sigma_T^{1 - \alpha_k} \| \nabla m(\cdot, s) \|_{L^\infty} + 4 A_k^2 \Lambda_T^{2 \alpha_k} R^{2 \alpha_k} \Sigma_T^{2 - 2 \alpha_k} \big) \, ds},$$
for any $t \in [0, \tau_k]$. For $R$ small enough, we infer that
\begin{equation}
\label{pennsylvania}
\tilde{\Sigma}_k(t) \leq \Sigma_T - 1,
\end{equation}
for any $t \in [0, \min \{ \tau_k, T \}]$.

When $T > \tau_k$, we iterate this argument on the time interval $[0, 2 \tau_k]$. Since $\tilde{\Sigma}_k(\tau_k)$ is less than $\Sigma_T - 1$ by~\eqref{pennsylvania}, the maximal time of existence $\tilde{T}_{\max}$ is more than $2 \tau_k$, and the quantity $\tilde{\Sigma}_{\rm LL}^k(t)$ is bounded by $\Sigma_ T$ on $[0, 2 \tau_k]$. Estimates~\eqref{ohio} (with $m_n$ replaced by $m$) and~\eqref{oregon} follow for the same constants $\Lambda_T$, $A_k$ and $\alpha_k$. For $R$ small enough, we derive~\eqref{pennsylvania} on the time interval $[0, \min \{ 2 \tau_k, T \}]$.

Arguing inductively, we conclude that there exists a positive number $R$ such that, if
$$\big\| \tilde{m}^0 - m^0 \|_{Z^k} \leq R,$$
then the maximal time of existence $\tilde{T}_{\max}$ of the solution $\tilde{m}$ is larger than, or equal to $T$, and we have
$$\max_{t \in [0, T]} \| \tilde{m}(\cdot, t) - m(\cdot, t) \|_{Z^{k - 1}} \leq \Lambda_T \| \tilde{m}^0 - m^0 \|_{Z^{k - 1}}.$$
It only remains to apply a standard density argument in order to replace the smooth solution $\tilde{m}$ in this inequality by an arbitrary solution. The flow map is then well-defined and Lipschitz continuous from the ball $B(m^0, R)$ of $\boE^k(\R^N)$ towards $\boC^0([0, T], \boE^{k - 1}(\R^N)$. This concludes the proof of Theorem~\ref{thm:LL-Cauchy}. \qed

%%%%%%%%%%%%%%%%%%%%%%%%%%%%%%%%%%%%%%%%%%%%%%%%%%%%
%%%%%%%%%%%%%%%%%%%%%%%%%%%%%%%%%%%%%%%%%%%%%%%%%%%%
\subsection{Proof of Corollary~\ref{cor:HLL-Cauchy}}
\label{sub:Cauchy-smooth-HLL}
%%%%%%%%%%%%%%%%%%%%%%%%%%%%%%%%%%%%%%%%%%%%%%%%%%%%
%%%%%%%%%%%%%%%%%%%%%%%%%%%%%%%%%%%%%%%%%%%%%%%%%%%%

Consider an initial datum $(u^0, \phi^0) \in \boN\boV^k(\R^N)$ and set
$$m^0 := \big( \rho^0 \sin(\phi^0), \rho^0 \cos(\phi^0), u^0 \big),$$
with $\rho^0 := (1 - (u^0)^2)^{1/2}$. Assume first the existence of a solution $(u, \phi) : \R^N \times [0, T_{\max}) \to (- 1, 1) \times \R$ to~\eqref{HLL} with initial datum $(u^0, \phi^0)$, which satisfies the statements in Corollary~\ref{cor:HLL-Cauchy}. Let $0 < T < T_{\max}$ be fixed. Since $k - 1 > N/2$, it follows from statement $(iii)$ in Corollary~\ref{cor:HLL-Cauchy} and the Sobolev embedding theorem that $u$ is continuous from $[0, T]$ to $\boC_b^0(\R^N)$. Moreover, we claim that
\begin{equation}
\label{soie}
\eta^T := \max_{t \in [0, T]} \| u(\cdot, t) \|_{L^\infty} < 1.
\end{equation}
Indeed, due to the non-vanishing condition in~\eqref{def:NVk}, the number $\eta^T$ is less than or equal to $1$. Assume by contradiction that it is equal to $1$. In this case, there exists a number $0 \leq t_* \leq T$ such that
$$\| u(\cdot, t_*) \|_{L^\infty} = 1.$$
On the other hand, the function $u(\cdot, t_*)$ lies in $H^{k - 1}(\R^N)$. By the Sobolev embedding theorem, it converges to $0$ at infinity. As a consequence, there exists a position $x_* \in \R^N$ such that
$$|u(x_*, t_*)| = 1,$$
which contradicts the non-vanishing condition in~\eqref{def:NVk}.

Set $\rho := (1 - u^2)^{1/2}$, and
$$m := \big( \rho \sin(\phi), \rho \cos(\phi), u \big).$$
Since $\eta^T < 1$, there exists a smooth function $F : \R \to \R$ such that $F(x) = (1 - x^2)^{1/2}$ for any $|x| \leq \eta^T$. In particular, we can combine statement $(i)$ in Corollary~\ref{cor:HLL-Cauchy}, and inequality~\eqref{eq:compose-moser} for this function in order to prove that the map $\rho$ is well-defined, bounded and continuous on $\R^N \times [0, T]$, with $\partial_t \rho \in L^\infty([0, T], H^{k - 2}(\R^N))$ and $\nabla \rho \in L^\infty([0, T], H^{k - 1}(\R^N))$. Applying Lemma~\ref{lem:moser} again, and using Corollary~\ref{cor:moser-trigo}, we deduce that the function $m$ lies in $L^\infty([0, T], \boE^k(\R^N))$, with $\partial_t m \in L^\infty([0, T], H^{k - 2}(\R^N))$. A direct computation then shows that this function is a weak solution to~\eqref{LL}. As a consequence of the uniqueness property in Theorem~\ref{thm:LL-Cauchy}, it is the unique solution to~\eqref{LL} with initial datum $m^0$. This provides the uniqueness of the function $u$, which is equal to the third component $m_3$ by 
definition. Concerning the phase function $\phi$, we can combine statement $(iii)$ in Corollary~\ref{cor:HLL-Cauchy}, and Corollary~\ref{cor:carac} in order to prove that it is continuous on $\R^N \times [0, T]$. Moreover, it satisfies the identity
$$e^{i \phi} = \frac{m_2 + i m_1}{(1 - m_3^2)^\frac{1}{2}},$$
on $\R^N \times [0, T]$. Due to the continuity of the function $m$ on $\R^N \times [0, T]$, the continuous solutions of this equation are unique up to a constant number in $\pi \Z$. Since $\phi(\cdot, 0) = \phi^0$, this number is uniquely determined, so that $\phi$ is also uniquely determined. In case of existence, the solution $(u, \phi)$ is therefore the unique hydrodynamical pair corresponding to the solution $m$ to~\eqref{LL} with initial datum $m^0$ (as long as this hydrodynamical pair makes sense).

Concerning existence, we first check that $m^0$ is in $\boE^k(\R^N)$. Indeed, since $u^0 \in H^k(\R^N)$ with $k > N/2$, it follows as before from the non-vanishing condition $|u^0| < 1$ that
\begin{equation}
\label{coton}
\eta^0 := \| u^0 \|_{L^\infty} < 1.
\end{equation}
Arguing as above, this guarantees that the function $m^0$ is in $\boE^k(\R^N)$. Theorem~\ref{thm:LL-Cauchy} then provides the existence of a unique solution $m : \R^N \times [0, T_{\max}) \to \S^2$ to~\eqref{LL} with initial datum $m^0$.

Set
$$\tau_{\max} := \sup \big\{ \tau \in [0, T_{\max}) : |m_3| < 1 \ {\rm on} \ \R^N \times [0, \tau] \big\}.$$
We deduce from statement $(iii)$ in Theorem~\ref{thm:LL-Cauchy} and the Sobolev embedding theorem that the function $m_3$ is continuous from $[0, T_{\max})$ to $\boC_b^0(\R^N)$. Since $m_3^0 = u^0$, it follows from~\eqref{coton} that $\tau_{\max}$ is a positive number. Similarly, we show that
\begin{equation}
\label{laine}
\lim_{t \to \tau_{\max}} \| m_3(\cdot, t) \|_{L^\infty} = \| m_3(\cdot, \tau_{\max}) \|_{L^\infty} = 1,
\end{equation}
when $\tau_{\max} < T_{\max}$.

Let $0 < T < \tau_{\max}$. Arguing as for~\eqref{soie}, we obtain
\begin{equation}
\label{lin}
\eta^T := \max_{t \in [0, T]} \| m_3(\cdot, t) \|_{L^\infty} < 1.
\end{equation}
Set $\rho := (1 - m_3^2)^{1/2}$. As before, the function $1/\rho$ is well-defined, bounded and continuous on $\R^N \times [0, T]$, with $\partial_t (1/\rho) \in L^\infty([0, T], H^{k - 2}(\R^N))$ and $\nabla (1/\rho) \in L^\infty([0, T], H^k(\R^N))$. As a consequence, we can lift the solution $m$ as
$$m = \big( \rho \sin(\phi), \rho \cos(\phi), m_3 \big),$$
where the phase function $\phi$ is uniquely defined by the condition $\phi(\cdot, 0) = \phi^0$, and continuous from $\R^N \times [0, \tau_{\max})$ to $\R$. Since
$$\sin(\phi) = \frac{m_1}{\rho}, \quad \partial_t \phi = \frac{\langle \check{m}, i \partial_t \check{m} \rangle_\C}{\rho^2} \quad {\rm and} \quad \nabla \phi = \frac{\langle \check{m}, i \nabla \check{m} \rangle_\C}{\rho^2},$$
we also observe that $\sin(\phi)$, $\partial_t \phi$ and $\nabla \phi$ lie in $L^\infty([0, T], L^2(\R^N))$, $L^\infty([0, T], H^{k - 2}(\R^N))$, respectively $L^\infty([0, T], H^{k - 1}(\R^N))$ for any $0 < T < \tau_{\max}$. In particular, the function $\phi$ belongs to $L^\infty([0, T], H_{\sin}^k(\R^N))$.

At this stage, we set $u := m_3$ on $\R^N \times [0, \tau_{\max})$. Then, the pair $(u, \phi)$ lies in $L^\infty([0, T], \boN\boV^k(\R^N))$, with $(\partial_t u, \partial_t \phi) \in L^\infty([0, T], H^{k - 2}(\R^N))$, for any $0 < T < \tau_{\max}$. A simple computation shows that it is a strong solution to~\eqref{HLL} with initial datum $(u^0, \phi^0)$. Statements $(iv)$ and $(v)$ in Corollary~\ref{cor:HLL-Cauchy} are then direct consequences of the same statements in Theorem~\ref{thm:LL-Cauchy}.

We also derive from statement $(iii)$ in Theorem~\ref{thm:LL-Cauchy} that the functions $\rho \sin(\phi)$ and $u$ are in $\boC^0([0, T], H^{k - 1}(\R^N))$, while the function $\rho \cos(\phi)$ lies in $\boC^0([0, T], \boC_b^0(\R^N))$, with $\nabla (\rho \cos(\phi)) \in \boC^0([0, T], H^{k - 2}(\R^N))$. Combining~\eqref{lin} with Lemma~\ref{lem:moser}, we check that the function $1/\rho$ is continuous from $[0, T]$ to $\boC_b^0(\R^N)$, with $\nabla (1/\rho) \in \boC^0([0, T], H^{k - 2}(\R^N))$. By Lemma~\ref{lem:moser} again, the function $\sin(\phi)$ is in $\boC^0([0, T], H^{k - 1}(\R^N))$, therefore in $\boC^0([0, T], \boC_b^0(\R^N))$, while $\cos(\phi)$ lies in $\boC^0([0, T], \boC_b^0(\R^N))$, with $\nabla \cos(\phi) \in \boC^0([0, T], H^{k - 2}(\R^N))$. Finally, we rely on the identities
\begin{equation}
\label{elasthane}
\nabla \phi = \cos(\phi) \nabla \big( \sin(\phi) \big) - \sin(\phi) \nabla \big( \cos(\phi) \big),
\end{equation}
and
\begin{equation}
\label{viscose}
\begin{split}
\sin \big( \phi(\cdot, t_2) - \phi(\cdot, t_1) \big) = \Big( \sin \big( \phi(\cdot, t_2) \big) & - \sin \big( \phi(\cdot, t_1) \big) \Big) \cos \big( \phi(\cdot, t_1) \big)\\
& + \sin \big( \phi(\cdot, t_1) \big) \Big( \cos \big( \phi(\cdot, t_1) \big) - \cos \big( \phi(\cdot, t_2) \big) \Big),
\end{split}
\end{equation}
in order to conclude that the phase $\phi$ belongs to $\boC^0([0, T], H_{\sin}^{k - 1}(\R^N))$. Hence, the flow map $(u^0, \phi^0) \mapsto (u, \phi)$ is well-defined from $\boN\boV^k(\R^N)$ to $\boC^0([0, T], \boN\boV^{k - 1}(\R^N))$. In order to complete the proof of statement $(iii)$, it remains to address the continuity of this map. 

Consider initial data $(u_n^0, \phi_n^0) \in \boN\boV^k(\R^N)$ such that
\begin{equation}
\label{polyester}
u_n^0 \to u^0 \quad {\rm in} \ H^k(\R^N), \quad {\rm and} \quad \phi_n^0 \to \phi^0 \quad {\rm in} \ H_{\sin}^k(\R^N),
\end{equation}
as $n \to \infty$. Applying Corollary~\ref{cor:moser-trigo} provides
$$\sin(\phi_n^0 - \phi^0) \to 0 \quad {\rm in} \ H^k(\R^N),$$
so that by the Sobolev embedding theorem,
$$\sin(\phi_n^0 - \phi^0) \to 0 \quad {\rm in} \ \boC_b^0(\R^N).$$
Hence, there exists an integer $\ell$ such that
$$\phi_n^0 - \phi^0 \to \ell \pi \quad {\rm in} \ \boC_b^0(\R^N).$$
In particular, we have
$$\sin(\phi_n^0) \to (- 1)^\ell \sin(\phi^0) \quad {\rm in} \ \boC_b^0(\R^N), \quad {\rm and} \quad \cos(\phi_n^0) \to (- 1)^\ell \cos(\phi^0) \quad {\rm in} \ \boC_b^0(\R^N).$$
Combining~\eqref{polyester} with the identity
$$(- 1)^\ell \sin(\phi_n^0) - \sin(\phi^0) = (- 1)^\ell \sin(\phi_n^0 - \phi^0) \cos(\phi^0) + \sin(\phi^0) \big( (- 1)^\ell \cos(\phi_n^0 - \phi^0) - 1),$$
we are led to
$$\sin(\phi_n^0) \to (- 1)^\ell \sin(\phi^0) \quad {\rm in} \ H^k(\R^N), \quad {\rm and} \quad \nabla \cos(\phi_n^0) \to (- 1)^\ell \nabla \cos(\phi^0) \quad {\rm in} \ H^{k - 1}(\R^N).$$
Similarly, we infer from~\eqref{polyester} and Lemma~\ref{lem:moser} that
$$\rho_n^0 \to \rho^0 \quad {\rm in} \ \boC_b^0(\R^N) \quad {\rm and} \quad \nabla \rho_n^0 \to \nabla \rho^0 \quad {\rm in} \ H^{k - 1}(\R^N).$$
Setting
$$m_n^0 := \big( (- 1)^\ell \rho_n^0 \sin(\phi_n^0), (- 1)^\ell \rho_n^0 \cos(\phi_n^0), u_n^0 \big),$$
with $\rho_n^0 := (1 - (u_n^0)^2)^{1/2}$, we deduce from Lemma~\ref{lem:moser} again that
$$m_n^0 \to m^0 \quad {\rm in} \ \boE^k(\R^N),$$
as $n \to \infty$.

We now rely on statement $(iii)$ in Theorem~\ref{thm:LL-Cauchy}. For $n$ large enough, this statement guarantees that the solutions $m_n$ to~\eqref{LL} with initial data $m_n^0$ are well-defined on the time interval $[0, T]$, and that they satisfy the convergences
\begin{equation}
\label{acrylique}
\max_{t \in [0, T]} \big\| m_n(\cdot, t) - m(\cdot, t) \big\|_{Z^{k - 1}} \to 0,
\end{equation}
as $n \to \infty$. By uniqueness, the solutions $(u_n, \phi_n)$ to~\eqref{HLL} with initial data $(u_n^0, \phi_n^0)$ are well-defined on $[0, T]$. Moreover, we have
$$m_n := \big( (- 1)^\ell \rho_n \sin(\phi_n), (- 1)^\ell \rho_n \cos(\phi_n), u_n \big),$$
where we have set $\rho_n := (1 - (u_n)^2)^{1/2}$. Arguing as before, we derive from~\eqref{acrylique} that
$$u_n \to u \quad {\rm in} \ \boC^0([0, T], H^{k - 1}(\R^N)), \quad (- 1)^\ell \sin(\phi_n) \to \sin(\phi) \quad {\rm in} \ \boC^0([0, T], H^{k - 1}(\R^N)),$$
as well as
$$(- 1)^\ell \cos(\phi_n) \to \cos(\phi) \quad {\rm in} \ \boC^0([0, T], \boC_b^0(\R^N)),$$
and
$$(- 1)^\ell \nabla \cos(\phi_n) \to \nabla \cos(\phi) \quad {\rm in} \ \boC^0([0, T], H^{k - 2}(\R^N)).$$
Invoking analogues of~\eqref{elasthane} and~\eqref{viscose}, we conclude that
$$\phi_n \to \phi \quad {\rm in} \ \boC^0([0, T], H_{\sin}^{k - 1}(\R^N)).$$
This completes the proof of the continuity of the flow map.

We finally turn to the characterization of the maximal time of existence. Due to the previous correspondence between the solutions to~\eqref{LL} and~\eqref{HLL}, the maximal time of existence of the solution $(u_, \phi)$ to~\eqref{HLL} is equal to $\tau_{\max}$. When $\tau_{\max} < T_{\max}$, formula~\eqref{laine} provides the second condition of statement $(ii)$ in Corollary~\ref{cor:HLL-Cauchy}. Otherwise, the maximal time of existence of $(u_, \phi)$ is equal to $T_{\max}$, and the condition in Theorem~\ref{thm:LL-Cauchy} then translates into the first condition of statement $(ii)$ in Corollary~\ref{cor:HLL-Cauchy}. This concludes the proof of this statement, and of Corollary~\ref{cor:HLL-Cauchy}. \qed

%%%%%%%%%%%%%%%%%%%%%%%%%%%%%%%%%%%%%%%%%%%%%%%%%%%%
%%%%%%%%%%%%%%%%%%%%%%%%%%%%%%%%%%%%%%%%%%%%%%%%%%%%
%%%%%%%%%%%%%%%%%%%%%%%%%%%%%%%%%%%%%%%%%%%%%%%%%%%%
\section{The derivation of the Sine-Gordon equation}
\label{sec:SG-derivation}
%%%%%%%%%%%%%%%%%%%%%%%%%%%%%%%%%%%%%%%%%%%%%%%%%%%%
%%%%%%%%%%%%%%%%%%%%%%%%%%%%%%%%%%%%%%%%%%%%%%%%%%%%
%%%%%%%%%%%%%%%%%%%%%%%%%%%%%%%%%%%%%%%%%%%%%%%%%%%%

%%%%%%%%%%%%%%%%%%%%%%%%%%%%%%%%%%%%%%%%%%%%%%%%%%%%%
%%%%%%%%%%%%%%%%%%%%%%%%%%%%%%%%%%%%%%%%%%%%%%%%%%%%%
\subsection{Proof of Proposition~\ref{prop:estimate}}
\label{sub:energy-eps}
%%%%%%%%%%%%%%%%%%%%%%%%%%%%%%%%%%%%%%%%%%%%%%%%%%%%%
%%%%%%%%%%%%%%%%%%%%%%%%%%%%%%%%%%%%%%%%%%%%%%%%%%%%%

Set $\rho_\varepsilon := 1 - \varepsilon^2 U_\varepsilon^2$. With this notation at hand, we can rewrite~\eqref{HLLeps} as
\begin{equation}
\label{HLLeps-bis}
\begin{cases} \partial_t U_\varepsilon = \rho_\varepsilon \, \Delta \Phi_\varepsilon + \nabla \rho_\varepsilon \cdot \nabla \Phi_\varepsilon - \frac{\sigma}{2} \rho_\varepsilon \, \sin(2 \Phi_\varepsilon),\\[5pt]
\partial_t \Phi_\varepsilon = U_\varepsilon - \frac{\varepsilon^2}{\rho_\varepsilon} \, \Delta U_\varepsilon - \frac{\varepsilon^2}{2} \nabla U_\varepsilon \cdot \nabla \big( \frac{1}{\rho_\varepsilon} \big) - \varepsilon^2 U_\varepsilon \, |\nabla \Phi_\varepsilon|^2 - \sigma \varepsilon^2 U_\varepsilon \sin^2(\Phi_\varepsilon), \end{cases}
\end{equation}
while the energy of order $\ell$ is given by
\begin{equation}
\label{energy-bis}
E_\varepsilon^\ell(U_\varepsilon, \Phi_\varepsilon) = \frac{1}{2} \sum_{|\alpha| = \ell - 1} \int_{\R^N} \Big( \frac{\varepsilon^2}{\rho_\varepsilon} \, |\nabla \partial_x^\alpha U_\varepsilon|^2 + |\partial_x^\alpha U_\varepsilon|^2 + \rho_\varepsilon \, |\nabla \partial_x^\alpha \Phi_\varepsilon|^2 + \sigma \rho_\varepsilon \, |\partial_x^\alpha \sin(\Phi_\varepsilon)|^2 \Big).
\end{equation}
In view of~\eqref{HLLeps-bis}, the time derivatives $\partial_t U_\varepsilon$ and $\partial_t \Phi_\varepsilon$ belong to $\boC^0([0, T], H^{k + 1}(\R^N))$, when the pair $(U_\varepsilon, \Phi_\varepsilon)$ lies in $\boC^0([0, T], \boN\boV^{k + 3}(\R^N))$. For $2 \leq \ell \leq k + 1$, the energy $E_\varepsilon^\ell(U_\varepsilon, \Phi_\varepsilon)$ is therefore of class $\boC^1$ on $[0, T]$. In view of~\eqref{HLLeps-bis}, its time derivative can be decomposed as
$$\big[ E_\varepsilon^\ell(U_\varepsilon, \Phi_\varepsilon) \big]'(t) = \sum_{j = 1}^5 \boI_j(t),$$
where we set
\begin{align*}
\boI_1 := & \sum_{|\alpha| = \ell - 1} \int_{\R^N} \partial_x^\alpha U_\varepsilon \, \partial_x^\alpha \Big( \rho_\varepsilon \, \Delta \Phi_\varepsilon + \nabla \rho_\varepsilon \cdot \nabla \Phi_\varepsilon - \frac{\sigma}{2} \rho_\varepsilon \, \sin(2 \Phi_\varepsilon) \Big),\\
\boI_2 := & \sum_{|\alpha| = \ell} \int_{\R^N} \rho_\varepsilon \, \partial_x^\alpha \Phi_\varepsilon \, \partial_x^\alpha \Big( U_\varepsilon - \frac{\varepsilon^2}{\rho_\varepsilon} \, \Delta U_\varepsilon - \frac{\varepsilon^2}{2} \nabla U_\varepsilon \cdot \nabla \Big( \frac{1}{\rho_\varepsilon} \Big) - \varepsilon^2 U_\varepsilon \, |\nabla \Phi_\varepsilon|^2 - \sigma \varepsilon^2 U_\varepsilon \, \sin^2(\Phi_\varepsilon) \Big),\\
\boI_3 := & \sum_{|\alpha| = \ell} \int_{\R^N} \frac{\varepsilon^2}{\rho_\varepsilon} \, \partial_x^\alpha U_\varepsilon \, \partial_x^\alpha \Big( \rho_\varepsilon \, \Delta \Phi_\varepsilon + \nabla \rho_\varepsilon \cdot \nabla \Phi_\varepsilon - \frac{\sigma}{2} \rho_\varepsilon \, \sin(2 \Phi_\varepsilon) \Big),\\
\boI_4 := & \sigma \sum_{|\alpha| = \ell - 1} \int_{\R^N} \rho_\varepsilon \, \partial_x^\alpha \big( \sin(\Phi_\varepsilon) \big) \times\\
& \quad \quad \times \partial_x^\alpha \Big( \cos(\Phi_\varepsilon) \, \big( U_\varepsilon - \frac{\varepsilon^2}{\rho_\varepsilon} \, \Delta U_\varepsilon - \frac{\varepsilon^2}{2} \nabla U_\varepsilon \cdot \nabla \Big( \frac{1}{\rho_\varepsilon} \Big) - \varepsilon^2 U_\varepsilon \, |\nabla \Phi_\varepsilon|^2 - \sigma \varepsilon^2 U_\varepsilon \, \sin^2(\Phi_\varepsilon) \big) \Big),
\end{align*}
and 
$$\boI_5 := - \varepsilon^2 \sum_{|\alpha| = \ell - 1} \int_{\R^N} U_\varepsilon \, \big( \partial_t U_\varepsilon \big) \, \Big( |\nabla \partial_x^\alpha \Phi_\varepsilon|^2 - \frac{\varepsilon^2}{\rho_\varepsilon^2} |\nabla \partial_x^\alpha U_\varepsilon|^2 + \sigma |\partial_x^\alpha \sin(\Phi_\varepsilon)|^2 \Big).$$
In order to establish~\eqref{der-E-j}, we now bound all these quantities. For the sake of simplicity, we drop, here as in the sequel, the dependence on $t \in [0, T]$.

We first collect the estimates for the functions $U_\varepsilon$, $\Phi_\varepsilon$ and $\rho_\varepsilon$ that we are using for controlling the quantities $\boI_j$. Concerning the function $\rho_\varepsilon$, we infer from~\eqref{borne-W} and direct computations the uniform estimates
\begin{equation}
\label{est-W-1a}
\frac{1}{2} \leq \rho_\varepsilon \leq 1, \quad \| \nabla \rho_\varepsilon \|_{L^\infty} \leq 2 \varepsilon^2 \| U_\varepsilon \|_{L^\infty} \, \| \nabla U_\varepsilon \|_{L^\infty}, \quad {\rm and} \quad \Big\| \nabla \Big( \frac{1}{\rho_\varepsilon} \Big) \Big\|_{L^\infty} \leq 8 \varepsilon^2 \| U_\varepsilon \|_{L^\infty} \, \| \nabla U_\varepsilon \|_{L^\infty},\\
\end{equation}
as well as
\begin{equation}
\label{est-W-1b}
\| d^2 \rho_\varepsilon \|_{L^\infty} + \Big\| d^2 \Big( \frac{1}{\rho_\varepsilon} \Big) \Big\|_{L^\infty} \leq C \varepsilon^2 \big( \| \nabla U_\varepsilon \|_{L^\infty}^2 + \| U_\varepsilon \|_{L^\infty} \| d^2 U_\varepsilon \|_{L^\infty} \big).
\end{equation}
Here as in the sequel, the notation $d^2 f$ stands for the second differential of the function $f$, while $C$ refers to a positive number, possibly different from line to line, and depending on $k$ and $N$, but not on $\varepsilon$ and $\sigma$. Applying Lemma~\ref{lem:moser} with $F(x) = 1 - x$, and using~\eqref{borne-W}, we also get
\begin{equation}
\label{est-W-2}
\| \partial_x^\alpha \rho_\varepsilon \|_{L^2} \leq C \varepsilon^2 \| U_\varepsilon \|_{L^\infty} \, \| U_\varepsilon \|_{\dot{H}^m},
\end{equation}
for any $1 \leq |\alpha| = m \leq k + 1$. Similarly, we can find a smooth function $G \in \boC_0^\infty(\R)$, with $G(x) = 1/(1 - x)$ for $|x| \leq 1/2$, so that, by~\eqref{borne-W}, $1/\rho_\varepsilon = G(\varepsilon U_\varepsilon)$. Applying again Lemma~\ref{lem:moser}, we are led to
\begin{equation}
\label{est-W-3}
\Big\| \partial_x^\alpha \Big( \frac{1}{\rho_\varepsilon} \Big) \Big\|_{L^2} \leq C \varepsilon^2 \| U_\varepsilon \|_{L^\infty} \, \| U_\varepsilon \|_{\dot{H}^m},
\end{equation}
for any $1 \leq |\alpha| = m \leq k + 1$.

Coming back to the definition of $\Sigma_\varepsilon^{k + 1}$, we deduce from~\eqref{energy-bis} that 
\begin{equation}
\label{est-U1}
\| \partial_x^\alpha U_\varepsilon \|_{L^2} \leq \big( 2 \Sigma_\varepsilon^{k + 1} \big)^\frac{1}{2}, \quad \varepsilon \| \partial_x^\alpha \nabla U_\varepsilon \|_{L^2} \leq \big( 2 \Sigma_\varepsilon^{k + 1} \big)^\frac{1}{2},
\end{equation}
as well as
\begin{equation}
\label{est-U2}
\sigma^\frac{1}{2} \| \partial_x^\alpha \sin(\Phi_\varepsilon) \|_{L^2} \leq 2 \big( \Sigma_\varepsilon^{k + 1} \big)^\frac{1}{2}, \quad {\rm and} \quad \| \partial_x^\alpha \nabla \Phi_\varepsilon \|_{L^2} \leq 2 \big( \Sigma_\varepsilon^{k + 1} \big)^\frac{1}{2},
\end{equation}
for any $0 \leq |\alpha| \leq k$. It then follows from Corollary~\ref{cor:cos} that
\begin{equation}
\label{est-cos}
\min \big\{ \sigma^\frac{1}{2}, 1 \big\} \| \partial_x^\alpha \nabla \cos(\Phi_\varepsilon) \|_{L^2} \leq C \big( \| \sin(\Phi_\varepsilon) \|_{L^\infty} + \| \nabla \Phi_\varepsilon \|_{L^\infty} \big) \, \big( \Sigma_\varepsilon^{k + 1} \big)^\frac{1}{2},
\end{equation}
when $0 \leq |\alpha| \leq k$.

We first estimate the quantity $\boI_5$. In view of~\eqref{est-W-1a}, we have 
$$\big| \boI_5 \big| \leq 4 \varepsilon^2 \| U_\varepsilon \|_{L^\infty} \, \| \partial_t U_\varepsilon \|_{L^\infty} \, E_\varepsilon^\ell.$$
Since
$$\| \partial_t U_\varepsilon \| _{L^\infty} \leq \| \Delta \Phi_\varepsilon \|_{L^\infty} + 2 \varepsilon^2 \| U_\varepsilon \|_{L^\infty} \, \| \nabla U_\varepsilon \|_{L^\infty} \| \nabla \Phi_\varepsilon \|_{L^\infty} + \sigma \| \sin(\Phi_\varepsilon) \|_{L^\infty},$$
by~\eqref{HLLeps-bis} and~\eqref{est-W-1a}, this is bounded by
\begin{equation}
\label{eq:est-I5}
\big| \boI_5 \big| \leq 8 \varepsilon^2 \| U_\varepsilon \|_{L^\infty} \, \big( \| \Delta \Phi_\varepsilon \|_{L^\infty} + \varepsilon^2 \| U_\varepsilon \|_{L^\infty} \, \| \nabla U_\varepsilon \|_{L^\infty} \, \| \nabla \Phi_\varepsilon \|_{L^\infty} + \sigma \| \sin(\Phi_\varepsilon) \|_{L^\infty} \big) \, \Sigma_\varepsilon^{k + 1}.
\end{equation}

We next split the quantity $\boI_1$ as
$$\boI_1 = \sum_{m = 1}^3 \boI_{1, m},$$
with
\begin{align*}
\boI_{1, 1} := & \sum_{|\alpha| = \ell - 1} \int_{\R^N} \partial_x^\alpha U_\varepsilon \, \big( \rho_\varepsilon \, \partial_x^\alpha \Delta \Phi_\varepsilon + \nabla \rho_\varepsilon \cdot \partial_x^\alpha \nabla \Phi_\varepsilon + \partial_x^\alpha \nabla \rho_\varepsilon \cdot \nabla \Phi_\varepsilon \big),
\\
\boI_{1, 2} := & \sum_{|\alpha| = \ell - 1} \int_{\R^N} \partial_x^\alpha U_\varepsilon \, \bigg( \sum_{\substack{1 \leq |\beta|\leq \ell - 1\\ \beta \leq \alpha}} \binom{\alpha}{\beta} \partial_x^\beta \rho_\varepsilon \, \partial_x^{\alpha - \beta} \Delta \Phi_\varepsilon + \sum_{\substack{1 \leq |\beta| \leq \ell - 2\\ \beta \leq \alpha}} \binom{\alpha}{\beta} \partial_x^\beta \nabla \rho_\varepsilon \cdot \partial_x^{\alpha - \beta} \nabla \Phi_\varepsilon \bigg),\\
\boI_{1, 3} := & - \frac{\sigma}{2} \sum_{|\alpha| = \ell - 1} \int_{\R^N} \partial_x^\alpha U_\varepsilon \, \partial_x^\alpha \big( \rho_\varepsilon \, \sin(2 \Phi_\varepsilon) \big).
\end{align*}
The quantity $\boI_{1, 1}$ contains the higher order derivatives. It cannot be estimated according to~\eqref{der-E-j} without taking into account cancellations with the similar parts $\boI_{2, 1}$ and $\boI_{3, 1}$ of the quantities $\boI_2$, respectively $\boI_3$. Hence, we postpone the analysis of $\boI_{1, 1}$, and first deal with $\boI_{1, 2}$. Indeed, we can bound directly this term by combining the estimates in Lemma~\ref{lem:moser} and Corollary~\ref{cor:moser} with~\eqref{est-W-1a},~\eqref{est-W-1b} and~\eqref{est-W-2} in order to get
\begin{equation}
\label{eq:est-I12}
\begin{split}
\big| \boI_{1, 2} \big| \leq C & \, \| U_\varepsilon \|_{\dot{H}^{\ell - 1}} \big( \| \nabla \rho_\varepsilon \|_{L^\infty} \, \| \nabla \Phi_\varepsilon \|_{\dot{H}^{\ell - 1}} + \| d^2 \Phi_\varepsilon \|_{L^\infty} \, \| \nabla \rho_\varepsilon \|_{\dot{H}^{\ell - 2}} + \| d^2 \rho_\varepsilon \|_{L^\infty} \, \| \nabla \Phi_\varepsilon \|_{\dot{H}^{\ell - 2}} \big),\\
\leq C & \varepsilon^2 \Big( \| U_\varepsilon \|_{L^\infty} \, \big( \| \nabla U_\varepsilon \|_{L^\infty} + \| d^2 \Phi_\varepsilon \|_{L^\infty} + \| d^2 U_\varepsilon \|_{L^\infty} \big) + \| \nabla U_\varepsilon \|_{L^\infty}^2 \Big) \, \Sigma_\varepsilon^\ell.
\end{split}
\end{equation}

We next control $\boI_{1, 3}$ by taking advantage of a cancellation with the quantity $\boI_4$. As before, we decompose this quantity as
$$\boI_4 = \sum_{m = 1}^3 \boI_{4, m},$$
with
\begin{align*}
\boI_{4, 1} := & \sigma \sum_{|\alpha| = \ell - 1} \int_{\R^N} \rho_\varepsilon \, \partial_x^\alpha \big( \sin(\Phi_\varepsilon) \big) \, \partial_x^\alpha \big( \cos(\Phi_\varepsilon) U_\varepsilon \big),\\
\boI_{4, 2} := & -\sigma \varepsilon^2 \sum_{|\alpha| = \ell - 1} \int_{\R^N} \rho_\varepsilon \, \partial_x^\alpha \big( \sin(\Phi_\varepsilon) \big) \, \partial_x^\alpha \Big( \frac{1}{\rho_\varepsilon} \Delta U_\varepsilon \cos(\Phi_\varepsilon) \Big),\\
\boI_{4, 3} := & - \sigma \varepsilon^2 \sum_{|\alpha| = \ell - 1} \int_{\R^N} \partial_x^\alpha \big( \sin(\Phi_\varepsilon) \big) \partial_x^\alpha \Big( \cos(\Phi_\varepsilon) \Big( \nabla U_\varepsilon \cdot \nabla \Big( \frac{1}{2 \rho_\varepsilon} \Big) + U_\varepsilon |\nabla \Phi_\varepsilon|^2 + \sigma U_\varepsilon \sin^2(\Phi_\varepsilon) \Big) \Big).
\end{align*}
Applying the Leibniz rule, we observe that
\begin{align*}
\boI_{1, 3} + \boI_{4, 1} = \sigma \sum_{|\alpha| = \ell - 1} \sum_{\substack{1 \leq |\beta| \leq \ell - 1\\ \beta \leq \alpha}} \binom{\alpha}{\beta} \int_{\R^N} \Big( & \rho_\varepsilon \, \partial_x^\alpha \big( \sin(\Phi_\varepsilon) \big) \, \partial_x^\beta \big( \cos(\Phi_\varepsilon) \big) \, \partial_x^{\alpha - \beta} U_\varepsilon\\
& - \partial_x^\alpha U_\varepsilon \, \partial_x^{\alpha - \beta} \big( \sin(\Phi_\varepsilon) \big) \, \partial_x^\beta \big( \rho_\varepsilon \, \cos(\Phi_\varepsilon) \big) \Big).
\end{align*}
Hence, we deduce from~\eqref{est-W-1a},~\eqref{est-W-2},~\eqref{est-U1},~\eqref{est-U2} and~\eqref{est-cos}, as well as from Corollary~\ref{cor:moser}, that
\begin{equation}
\label{eq:est-I13}
\begin{split}
\big| \boI_{1, 3} + \boI_{4, 1} \big| & \leq C \sigma \bigg( \| \sin(\Phi_\varepsilon) \|_{\dot{H}^{\ell - 1}} \, \big( \| \nabla \cos(\Phi_\varepsilon) \|_{L^\infty} \, \| U_\varepsilon \|_{\dot{H}^{\ell - 2}} + \| \cos(\Phi_\varepsilon) \|_{\dot{H}^{\ell - 1}} \, \| U_\varepsilon \|_{L^\infty} \big)\\
& + \| U_\varepsilon \|_{\dot{H}^{\ell - 2}} \, \Big( \| \sin(\Phi_\varepsilon) \|_{\dot{H}^{\ell - 2}} \, \big( \| \nabla \rho_\varepsilon \|_{L^\infty} + \| \nabla \cos(\Phi_\varepsilon) \|_{L^\infty} \big) + \| \sin(\Phi_\varepsilon) \|_{L^\infty} \, \big( \| \rho_\varepsilon \|_{\dot{H}^{\ell - 1}}\\
& + \| \cos(\Phi_\varepsilon) \|_{\dot{H}^{\ell - 1}} + \delta_{\ell \neq 2} \| \nabla \rho_\varepsilon \|_{L^\infty} \, \| \cos(\Phi_\varepsilon) \|_{\dot{H}^{\ell - 2}} + \delta_{\ell \neq 2} \| \nabla \cos(\Phi_\varepsilon) \|_{L^\infty} \, \| \rho_\varepsilon \|_{\dot{H}^{\ell - 2}} \big) \Big) \bigg)
\\
& \leq C \max \big\{ \sigma, 1 \big\} \Big( \big( \| \sin(\Phi_\varepsilon) \|_{L^\infty} + \| \nabla \Phi_\varepsilon \|_{L^\infty} \big) \, \big( \| \sin(\Phi_\varepsilon) \|_{L^\infty} + \| U_\varepsilon \|_{L^\infty} \big)\\
& + \varepsilon^2 \| U_\varepsilon \|_{L^\infty} \, \big( \| \nabla U_\varepsilon \|_{L^\infty} + \| \sin \Phi_\varepsilon \|_{L^\infty} \big) \, \big( 1 + \| \sin(\Phi_\varepsilon) \|_{L^\infty} \, \| \nabla \Phi_\varepsilon \|_{L^\infty} \big) \Big) \, \Sigma_\varepsilon^{k + 1}.
\end{split}
\end{equation}
Note that this estimate is the only one, which contains terms without multiplicative factor $\varepsilon$.

Estimating $\boI_{4, 3}$ is also direct. In view of~\eqref{est-W-1a},~\eqref{est-W-3},~\eqref{est-U1} and~\eqref{est-U2}, we notice that
$$\Big\| \nabla U_\varepsilon \cdot \nabla \Big( \frac{1}{\rho_\varepsilon} \Big) \Big\|_{\dot{H}^{\ell - 1}} \leq C \varepsilon \| U_\varepsilon \|_{L^\infty} \, \| \nabla U_\varepsilon \|_{L^\infty} \, \big( \Sigma_\varepsilon^{k + 1} \big)^\frac{1}{2},$$
$$\big\| U_\varepsilon |\nabla \Phi_\varepsilon|^2 \big\|_{\dot{H}^{\ell - 1}} \leq C \big( \| U_\varepsilon \|_{L^\infty} \, \| \nabla \Phi_\varepsilon \|_{L^\infty} + \| \nabla \Phi_\varepsilon \|^2_{L^\infty} \big) \, \big( \Sigma_\varepsilon^{k + 1} \big)^\frac{1}{2},$$
and
$$\sigma^\frac{1}{2} \big\| U_\varepsilon \sin(\Phi_\varepsilon)^2 \big\|_{\dot{H}^{\ell - 1}} \leq C \big( \| U_\varepsilon \|_{L^\infty} \, \| \sin(\Phi_\varepsilon) \|_{L^\infty} + \sigma^\frac{1}{2} \| \sin(\Phi_\varepsilon) \|^2_{L^\infty} \big) \, \big( \Sigma_\varepsilon^{k + 1} \big)^\frac{1}{2}.$$

Using~\eqref{est-cos}, we are led to
\begin{equation}
\label{eq:est-I43}
\begin{split}
\big| & \boI_{4, 3} \big| \leq C \max \big\{ 1, \sigma^\frac{3}{2} \big\} \, \varepsilon^2 \Big( \varepsilon \|U_\varepsilon \|_{L^\infty} \, \| \nabla U_\varepsilon \|_{L^\infty} + \big( \| \sin(\Phi_\varepsilon) \|_{L^\infty} + \| \nabla \Phi_\varepsilon \|_{L^\infty} \big) \times\\
& \times \big( \| U_\varepsilon \|_{L^\infty} + \| \sin(\Phi_\varepsilon) \|_{L^\infty} + \| \nabla \Phi_\varepsilon \|_{L^\infty} + \| U_\varepsilon \|_{L^\infty} \, \| \nabla \Phi_\varepsilon \|_{L^\infty}^2 + \varepsilon^2 \| U_\varepsilon \|_{L^\infty} \, \| \nabla U_\varepsilon \|_{L^\infty}^2 \big) \Big) \, \Sigma_\varepsilon^{k + 1}.
\end{split}
\end{equation}

Concerning $\boI_{4, 2}$, we use a further cancellation with the quantity $\boI_3$. Set
$$\boI_3 = \sum_{m = 1}^3 \boI_{3, m},$$
with
\begin{align*}
\boI_{3, 1} := & \varepsilon^2 \sum_{|\alpha| = \ell} \int_{\R^N} \frac{1}{\rho_\varepsilon} \, \partial_x^\alpha U_\varepsilon \, \bigg( \nabla \rho_\varepsilon \cdot \partial_x^\alpha \nabla \Phi_\varepsilon + \partial_x^\alpha \nabla \rho_\varepsilon \cdot \nabla \Phi_\varepsilon + \sum_{\substack{0 \leq |\beta| \leq 1\\ \beta \leq \alpha }} \binom{\alpha}{\beta} \partial_x^\beta \rho_\varepsilon \, \partial_x^{\alpha - \beta} \Delta \Phi_\varepsilon \bigg),\\
\boI_{3, 2} := & \varepsilon^2 \sum_{|\alpha| = \ell} \int_{\R^N} \frac{1}{\rho_\varepsilon} \, \partial_x^\alpha U_\varepsilon\, \bigg( \sum_{\substack{2 \leq |\beta| \leq \ell\\ \beta \leq \alpha}} \binom{\alpha}{\beta} \partial_x^\beta \rho_\varepsilon \, \partial_x^{\alpha - \beta} \Delta \Phi_\varepsilon + \sum_{\substack{1 \leq |\beta| \leq \ell - 1\\ \beta \leq \alpha}} \binom{\alpha}{\beta} \partial_x^\beta \nabla \rho_\varepsilon \cdot \partial_x^{\alpha - \beta} \nabla \Phi_\varepsilon \bigg),\\
\boI_{3, 3} := & - \sigma \varepsilon^2 \sum_{|\alpha| = \ell - 1} \int_{\R^N} \frac{1}{\rho_\varepsilon} \, \partial_x^\alpha \nabla U_\varepsilon \cdot \partial_x^\alpha \nabla \big( \rho_\varepsilon \sin(\Phi_\varepsilon) \cos(\Phi_\varepsilon)\big),
\end{align*}
The identity
$$\frac{1}{\rho_\varepsilon} \, \Delta U_\varepsilon \, \cos(\Phi_\varepsilon) = \div \Big( \frac{1}{\rho_\varepsilon} \, \nabla U_\varepsilon \, \cos(\Phi_\varepsilon) \Big) - \nabla \Big( \frac{1}{\rho_\varepsilon} \, \cos(\Phi_\varepsilon) \Big) \cdot \nabla U_\varepsilon,$$
provides
\begin{align*}
\boI_{4, 2} = \sigma \varepsilon^2 \sum_{|\alpha| = \ell - 1} \int_{\R^N} \bigg( & \nabla \Big( \rho_\varepsilon \, \partial_x^\alpha \big( \sin(\Phi_\varepsilon) \big) \Big) \cdot \Big( \partial_x^\alpha \Big( \frac{1}{\rho_\varepsilon} \nabla U_\varepsilon \cos(\Phi_\varepsilon) \Big) \Big)\\
& - \rho_\varepsilon \, \partial_x^\alpha \big( \sin(\Phi_\varepsilon) \big) \,\Big( \partial_x^\alpha \Big( \nabla U_\varepsilon \cdot \nabla \Big( \frac{1}{\rho_\varepsilon} \cos(\Phi_\varepsilon) \Big) \Big) \bigg),
\end{align*}
after integrating by parts. Hence, we have
\begin{align*}
\boI_{3, 3} + \boI_{4, 2} = \sigma \varepsilon^2 \sum_{|\alpha| = \ell - 1} \int_{\R^N} & \Bigg( \partial_x^\alpha \big( \sin(\Phi_\varepsilon) \big) \times\\
& \times \bigg( \nabla \rho_\varepsilon \cdot \partial_x^\alpha \Big( \frac{1}{\rho_\varepsilon} \, \nabla U_\varepsilon \, \cos(\Phi_\varepsilon) \Big) - \rho_\varepsilon \, \partial_x^\alpha \Big( \nabla U_\varepsilon \cdot \nabla \Big( \frac{1}{\rho_\varepsilon} \,\cos(\Phi_\varepsilon) \Big) \Big) \bigg)\\
+ \sum_{\substack{1 \leq |\beta| \leq \ell - 1\\\beta \leq \alpha}} & \binom{\alpha}{\beta} \bigg( \rho_\varepsilon \, \partial_x^\beta \Big( \frac{1}{\rho_\varepsilon} \, \cos(\Phi_\varepsilon) \Big) \, \partial_x^\alpha \nabla \big( \sin(\Phi_\varepsilon) \big) \cdot \partial_x^{\alpha - \beta} \nabla U_\varepsilon\\
& - \frac{1}{\rho_\varepsilon} \, \partial_x^\beta \big( \rho_\varepsilon \, \cos(\Phi_\varepsilon) \big) \, \partial_x^\alpha \nabla U_\varepsilon \cdot \partial_x^{\alpha - \beta} \nabla \big( \sin(\Phi_\varepsilon) \big) \bigg) \Bigg).
\end{align*}
A further integration by parts then transforms the third line of this identity into
\begin{align*}
\int_{\R^N} \rho_\varepsilon \, \partial_x^\beta \Big( \frac{1}{\rho_\varepsilon} \, \cos(\Phi_\varepsilon) & \Big) \, \partial_x^\alpha \nabla \big( \sin(\Phi_\varepsilon) \big) \cdot \partial_x^{\alpha - \beta} \nabla U_\varepsilon \Big)\\
= - \int_{\R^N} \partial_x^\alpha \big( \sin(\Phi_\varepsilon) \big) \, \bigg( & \partial_x^\beta \Big( \frac{1}{\rho_\varepsilon} \, \cos(\Phi_\varepsilon) \Big) \, \nabla \rho_\varepsilon \cdot \partial_x^{\alpha - \beta} \nabla U_\varepsilon + \rho_\varepsilon \, \partial_x^\beta \nabla \Big( \frac{1}{\rho_\varepsilon} \, \cos(\Phi_\varepsilon) \Big) \cdot \partial_x^{\alpha - \beta} \nabla U_\varepsilon\\
& + \rho_\varepsilon \, \partial_x^\beta \Big( \frac{1}{\rho_\varepsilon} \, \cos(\Phi_\varepsilon) \Big) \, \partial_x^{\alpha - \beta} \Delta U_\varepsilon \bigg).
\end{align*}

In order to bound these various terms, we first infer from~\eqref{est-W-1a} and~\eqref{est-W-1b} that
$$\Big\| \nabla \big( \rho_\varepsilon \, \cos(\Phi_\varepsilon) \big) \Big\|_{L^\infty} + \Big\| \nabla \Big( \frac{1}{\rho_\varepsilon} \, \cos(\Phi_\varepsilon) \Big) \Big\|_{L^\infty} \leq C \big( \varepsilon^2 \| U_\varepsilon \|_{L^\infty} \, \| \nabla U_\varepsilon \|_{L^\infty} + \| \sin(\Phi_\varepsilon) \|_{L^\infty} \, \| \nabla \Phi_\varepsilon \|_{L^\infty} \big),$$
and
\begin{align*}
\Big\| d^2 \Big( \frac{1}{\rho_\varepsilon} \, \cos(\Phi_\varepsilon) \Big) \Big\|_{L^\infty} \leq C \Big( & \| \sin(\Phi_\varepsilon) \|_{L^\infty} \, \| d^2 \Phi_\varepsilon \|_{L^\infty} + \| \nabla \Phi_\varepsilon \|_{L^\infty}^2\\
& + \varepsilon^2 \big( \| U_\varepsilon \|_{L^\infty} \, \| d^2 U_\varepsilon \|_{L^\infty} + \big( 1 + \varepsilon^2 \| U_\varepsilon \|_{L^\infty}^2 \big) \, \| \nabla U_\varepsilon \|_{L^\infty}^2 \big) \Big).
\end{align*}
Similarly, if follows from~\eqref{est-W-2},~\eqref{est-W-3},~\eqref{est-U1} and~\eqref{est-cos} that 
\begin{align*}
\big\| \rho_\varepsilon \, \cos(\Phi_\varepsilon) \big\|_{\dot{H}^{\ell - 1}} & + \Big\| \frac{1}{\rho_\varepsilon} \, \cos(\Phi_\varepsilon) \Big\|_{\dot{H}^{\ell - 1}} + \Big\| \frac{1}{\rho_\varepsilon} \, \cos(\Phi_\varepsilon) \Big\|_{\dot{H}^\ell}\\
& \leq C \max \Big\{ 1, \frac{1}{\sigma^\frac{1}{2}} \Big\} \, \big( \| \sin(\Phi_\varepsilon) \|_{L^\infty} + \| \nabla \Phi_\varepsilon \|_{L^\infty} + \varepsilon \| U_\varepsilon \|_{L^\infty} \big) \, \big( \Sigma_\varepsilon^{k + 1} \big)^\frac{1}{2}.
\end{align*}
Hence, we derive as before that
\begin{equation}
\label{eq:est-I33}
\begin{split}
\big| \boI_{3, 3} + \boI_{4, 2} \big| \leq & C \max \big\{ 1, \sigma \big\} \varepsilon \Big( \varepsilon^3 \| \nabla U_\varepsilon \|_{L^\infty}^2 (1 + \varepsilon^2 \| U_\varepsilon \|_{L^\infty}^2) + \big( \| \sin(\Phi_\varepsilon) \|_{L^\infty} + \| \nabla \Phi_\varepsilon \|_{L^\infty} \big) \times\\
& \times \big( \| \nabla \Phi_\varepsilon \|_{L^\infty} + \varepsilon \| \nabla U_\varepsilon \|_{L^\infty} + \varepsilon \| \nabla \Phi_\varepsilon \|_{L^\infty} + \varepsilon \| d^2 U_\varepsilon \|_{L^\infty} + \varepsilon \| d^2 \Phi_\varepsilon \|_{L^\infty} \big)\\
& + \varepsilon^2 \| U_\varepsilon \|_{L^\infty} \, \big( \| \nabla U_\varepsilon \|_{L^\infty} + \| \nabla \Phi_\varepsilon \|_{L^\infty} + \varepsilon \| d^2 U_\varepsilon \|_{L^\infty} \big) + \varepsilon^3 \| U_\varepsilon \|_{L^\infty} \, \| \nabla U_\varepsilon \|_{L^\infty} \times\\
& \times \big( \| \sin(\Phi_\varepsilon) \|_{L^\infty} + \| \nabla \Phi_\varepsilon \|_{L^\infty} \big) \, \big( \| \nabla U_\varepsilon \|_{L^\infty} + \| \nabla \Phi_\varepsilon \|_{L^\infty} \big) \Big) \, \Sigma_\varepsilon^{k + 1}.
\end{split}
\end{equation}

Coming back to $\boI_{3, 2}$, a simple computation shows that
\begin{equation}
\label{eq:est-I32}
\big| \boI_{3, 2} \big| \leq C \varepsilon^2 \big( \varepsilon\| \nabla U_\varepsilon \|_{L^\infty}^2 + \varepsilon \| U_\varepsilon \|_{L^\infty} \, \| D^2 U_\varepsilon \|_{L^\infty} + \| U_\varepsilon \|_{L^\infty} \, \| D^2 \Phi_\varepsilon \|_{L^\infty} \big) \, \Sigma_\varepsilon^{k + 1}.
\end{equation}

It then remains to split the quantity $\boI_2$ as
$$\boI_2 = \sum_{m = 1}^3 \boI_{2, m},$$
\begin{align*}
\boI_{2, 1} := & \sum_{|\alpha| = \ell} \int_{\R^N} \rho_\varepsilon \, \partial_x^\alpha \Phi_\varepsilon \, \bigg( \partial_x^\alpha U_\varepsilon - \varepsilon^2 \sum_{\substack{0 \leq |\beta| \leq 1\\ \beta \leq \alpha}} \binom{\alpha}{\beta} \partial_x^\beta \Big( \frac{1}{\rho_\varepsilon} \Big) \, \partial_x^{\alpha - \beta} \Delta U_\varepsilon\\
& \quad \quad \quad \quad \quad \quad \quad \quad \quad - \frac{\varepsilon^2}{2} \Big( \nabla \Big( \frac{1}{\rho_\varepsilon} \Big) \cdot \nabla \partial_x^\alpha U_\varepsilon + \nabla \partial_x^\alpha \Big( \frac{1}{\rho_\varepsilon} \Big) \cdot \nabla U_\varepsilon + 2 U_\varepsilon \cdot \partial_x^\alpha \big( |\nabla \Phi_\varepsilon|^2 \big) \Big) \bigg),\\
\boI_{2, 2} := & - \varepsilon^2 \sum_{|\alpha| = \ell} \int_{\R^N} \rho_\varepsilon \, \partial_x^\alpha \Phi_\varepsilon \, \bigg( \sum_{\substack{2 \leq |\beta| \leq \ell\\ \beta \leq \alpha}} \binom{\alpha}{\beta} \partial_x^\beta \Big( \frac{1}{\rho_\varepsilon} \Big) \, \partial_x^{\alpha - \beta} \Delta U_\varepsilon\\
& \quad \quad \quad \quad \quad + \frac{1}{2} \sum_{\substack{1 \leq |\beta| \leq \ell - 1\\ \beta \leq \alpha }} \binom{\alpha}{\beta} \nabla \partial_x^\beta \Big( \frac{1}{\rho_\varepsilon} \Big) \cdot \nabla \partial_x^{\alpha - \beta} U_\varepsilon + \sum_{\substack{1 \leq |\beta| \leq \ell\\ \beta \leq \alpha}} \binom{\alpha}{\beta} \partial_x^\beta U_\varepsilon \, \partial_x^{\alpha - \beta} (|\nabla \Phi_\varepsilon|^2) \bigg),\\
\boI_{2, 3} := & - \sigma \varepsilon^2 \sum_{|\alpha| = \ell} \int_{\R^N} \rho_\varepsilon \, \partial_x^\alpha \Phi_\varepsilon \, \partial_x^\alpha \big( U_\varepsilon \sin^2(\Phi_\varepsilon) \big).
\end{align*}
We can derive as before the inequality
\begin{equation}
\label{eq:est-I22}
\big| \boI_{2, 2} \big| \leq C \varepsilon \big( \| \nabla \Phi_\varepsilon \|_{L^\infty}^2 + \varepsilon^2 \| U_\varepsilon \|_{L^\infty} \, \| d^2 U_\varepsilon \|_{L^\infty} + \varepsilon^2 \| \nabla U_\varepsilon \|_{L^\infty}^2 \big) \, \Sigma_\varepsilon^{k + 1}.
\end{equation}
Using the inequality
\begin{align*}
\big\| \sin(\Phi_\varepsilon)^2 \big\|_{\dot{H}^\ell} = 2 \big\| \sin(\Phi_\varepsilon) & \, \cos(\Phi_\varepsilon) \, \nabla \Phi_\varepsilon \big\|_{\dot{H}^{\ell - 1}} \leq C \Big( \| \sin(\Phi_\varepsilon) \|_{L^\infty} \, \| \Phi_\varepsilon \|_{\dot{H}^\ell}\\
& + \| \nabla \Phi_\varepsilon \|_{L^\infty} \, \big( \| \sin(\Phi_\varepsilon) \|_{\dot{H}^{\ell - 1}} + \| \sin(\Phi_\varepsilon) \|_{L^\infty} \, \| \cos(\Phi_\varepsilon) \|_{\dot{H}^{\ell - 1}} \big) \Big),
\end{align*}
we also compute
\begin{equation}
\label{eq:est-I23}
\begin{split}
\big| \boI_{2, 3} \big| \leq C \, \max \big\{ \sigma, \sigma^\frac{1}{2} \big\} \, \varepsilon \, \Big( & \| \sin(\Phi_\varepsilon) \|_{L^\infty}^2 + \varepsilon \big( \| U_\varepsilon \|_{L^\infty} \, \| \sin(\Phi_\varepsilon) \|_{L^\infty} + \| U_\varepsilon \|_{L^\infty} \, \| \nabla \Phi_\varepsilon \|_{L^\infty}\\
& + \| U_\varepsilon \|_{L^\infty} \, \| \nabla \Phi_\varepsilon \|_{L^\infty}^2 \big) \Big) \, \Sigma_\varepsilon^{k + 1}.
\end{split}
\end{equation}

At this stage, we are left with the quantity $\boH := \boI_{1, 1} + \boI_{2, 1} + \boI_{3, 1}$, which contains the higher order derivatives. We again control this quantity by integrating by parts in order to make clear the cancellations between the different terms. More precisely, we collect the terms of $\boH$ as
\begin{align*}
\boH = & \sum_{m = 1}^7 \boH_m := \int_{\R^N} \sum_{|\alpha| = \ell - 1} \Big( \partial_x^\alpha U_\varepsilon \, \big( \rho_\varepsilon \, \partial_x^\alpha \Delta \Phi_\varepsilon + \nabla \rho_\varepsilon \cdot \nabla \partial_x^\alpha \Phi_\varepsilon \big) + \rho_\varepsilon \, \partial_x^\alpha \nabla \Phi_\varepsilon \cdot \partial_x^\alpha \nabla U_\varepsilon \Big)\\
& + \varepsilon^2 \int_{\R^N} \sum_{|\alpha| = \ell} \Big( \partial_x^\alpha \Delta \Phi_\varepsilon \, \partial_x^\alpha U_\varepsilon \, - \partial_x^\alpha \Phi_\varepsilon \, \partial_x^\alpha \Delta U_\varepsilon \Big) - \varepsilon^2 \int_{\R^N} \sum_{|\alpha| = \ell} \rho_\varepsilon \, U_\varepsilon \, \partial_x^\alpha \Phi_\varepsilon \, \partial_x^\alpha (|\nabla \Phi_\varepsilon|^2)\\
& + \varepsilon^2 \int_{\R^N} \sum_{|\alpha| = \ell} \frac{1}{\rho_\varepsilon} \, \partial_x^\alpha U_\varepsilon \, \partial_x^\alpha \nabla \rho_\varepsilon \cdot \nabla \Phi_\varepsilon + \int_{\R^N} \sum_{|\alpha| = \ell - 1} \partial_x^\alpha U_\varepsilon \, \partial_x^\alpha \nabla \rho_\varepsilon \cdot \nabla \Phi_\varepsilon\\
& + \varepsilon^2 \int_{\R^N} \sum_{|\alpha| = \ell} \bigg( \frac{\nabla \rho_\varepsilon}{\rho_\varepsilon} \cdot \Big( \frac{1}{2} \partial_x^\alpha \Phi_\varepsilon \, \partial_x^\alpha \nabla U_\varepsilon + \partial_x^\alpha U_\varepsilon \, \partial_x^\alpha \nabla \Phi_\varepsilon \Big) - \frac{\rho_\varepsilon}{2} \, \partial_x^\alpha \Phi_\varepsilon \, \partial_x^\alpha \nabla \Big( \frac{1}{\rho_\varepsilon} \Big) \cdot \nabla U_\varepsilon \bigg)\\
& + \varepsilon^2 \int_{\R^N} \sum_{|\alpha| = \ell} \sum_{\substack{|\beta| = 1\\ \beta \leq \alpha}} \frac{\partial_x^\beta \rho_\varepsilon}{\rho_\varepsilon} \Big( \partial_x^{\alpha - \beta} \Delta U_\varepsilon \, \partial_x^\alpha \Phi_\varepsilon + \partial_x^\alpha U_\varepsilon \, \partial_x^{\alpha - \beta} \Delta \Phi_\varepsilon \Big).
\end{align*}
Integrating by parts directly provides
$$\boH_1 = \boH_2 = 0.$$

We next write $\boH_3$ as
$$\boH_3 = \varepsilon^2 \int_{\R^N} \sum_{|\alpha| = \ell} \rho_\varepsilon \, U_\varepsilon \, \partial_x^\alpha \Phi_\varepsilon \Big( 2 \partial_x^\alpha \nabla \Phi_\varepsilon \cdot \nabla \Phi_\varepsilon - \partial_x^\alpha (|\nabla \Phi_\varepsilon|^2) \Big) - 2 \varepsilon^2 \int_{\R^N} \sum_{|\alpha| = \ell} \rho_\varepsilon \, U_\varepsilon \, \partial_x^\alpha \Phi_\varepsilon \, \partial_x^\alpha \nabla \Phi_\varepsilon \cdot \nabla \Phi_\varepsilon.$$
Applying the Leibniz formula and using Corollary~\ref{cor:moser}, we first check that
$$\big\| 2 \partial_x^\alpha \nabla \Phi_\varepsilon \cdot \nabla \Phi_\varepsilon - \partial_x^\alpha (|\nabla \Phi_\varepsilon|^2) \big\|_{L^2} \leq C \| d^2 \Phi_\varepsilon \|_{L^\infty} \, \| \Phi_\varepsilon \|_{\dot{H}^\ell},$$
when $|\alpha| = \ell$. Hence, we derive from~\eqref{est-W-1a} and~\eqref{est-U2} that
$$\bigg| \varepsilon^2 \int_{\R^N} \sum_{|\alpha| = \ell} \rho_\varepsilon \, U_\varepsilon \, \partial_x^\alpha \Phi_\varepsilon \Big( 2 \partial_x^\alpha \nabla \Phi_\varepsilon \cdot \nabla \Phi_\varepsilon - \partial_x^\alpha (|\nabla \Phi_\varepsilon|^2) \Big) \bigg| \leq C \varepsilon^2 \| U_\varepsilon \|_{L^\infty} \, \| d^2 \Phi_\varepsilon \|_{L^\infty} \Sigma_\varepsilon^{k + 1}.$$
On the other hand, an integration by parts yields
$$- 2 \varepsilon^2 \int_{\R^N} \sum_{|\alpha| = \ell} \rho_\varepsilon \, U_\varepsilon \, \partial_x^\alpha \Phi_\varepsilon \, \partial_x^\alpha \nabla \Phi_\varepsilon \cdot \nabla \Phi_\varepsilon = \varepsilon^2 \int_{\R^N} \sum_{|\alpha| = \ell} \div \big( \rho_\varepsilon \, U_\varepsilon \, \nabla \Phi_\varepsilon \big) \, |\partial_x^\alpha \Phi_\varepsilon|^2,$$
so that we finally obtain 
\begin{equation}
\label{eq:est-H3}
|\boH_3| \leq C \varepsilon^2 \Big( \| U_\varepsilon \|_{L^\infty} \, \| d^2 \Phi_\varepsilon \|_{L^\infty} + \big( 1 + \varepsilon^2 \| U_\varepsilon \|_{L^\infty}^2 \big) \, \| \nabla U_\varepsilon \|_{L^\infty} \, \| \nabla \Phi_\varepsilon \|_{L^\infty} \Big) \, \Sigma_\varepsilon^{k + 1}.
\end{equation}

We argue similarly for the term $\boH_4$, which we decompose as
$$\boH_4 = \varepsilon^2 \int_{\R^N} \sum_{|\alpha| = \ell} \frac{\partial_x^\alpha U_\varepsilon}{\rho_\varepsilon} \, \nabla \Phi_\varepsilon \cdot \big( \partial_x^\alpha \nabla \rho_\varepsilon + 2 \varepsilon^2 U_\varepsilon \, \partial_x^\alpha \nabla U_\varepsilon \big) - 2 \varepsilon^4 \int_{\R^N} \sum_{|\alpha| = \ell} \frac{U_\varepsilon}{\rho_\varepsilon} \, \partial_x^\alpha U_\varepsilon \, \partial_x^\alpha \nabla U_\varepsilon \cdot \nabla \Phi_\varepsilon.$$
Since
$$\big\| \partial_x^\alpha \nabla \rho_\varepsilon + 2 \varepsilon^2 U_\varepsilon \partial_x^\alpha \nabla U_\varepsilon \big\|_{L^2} \leq C \varepsilon^2 \| \nabla U_\varepsilon \|_{L^\infty} \, \| U_\varepsilon \|_{\dot{H}^\ell},$$
when $|\alpha| = \ell$, and
$$2 \varepsilon^4 \int_{\R^N} \sum_{|\alpha| = \ell} \frac{U_\varepsilon}{\rho_\varepsilon} \, \partial_x^\alpha U_\varepsilon \, \partial_x^\alpha \nabla U_\varepsilon \cdot \nabla \Phi_\varepsilon = - \varepsilon^4 \int_{\R^N} \sum_{|\alpha| = \ell} \div \Big( \frac{U_\varepsilon \, \nabla \Phi_\varepsilon}{\rho_\varepsilon} \Big) \, \big( \partial_x^\alpha U_\varepsilon \big)^2,$$
we have
\begin{equation}
\label{eq:est-H4}
\big| \boH_4 \big| \leq C \varepsilon^2 \Big( \| U_\varepsilon \|_{L^\infty} \, \| d^2 \Phi_\varepsilon \|_{L^\infty} + \big( 1 + \varepsilon^2 \| U_\varepsilon \|_{L^\infty}^2 \big) \, \| \nabla U_\varepsilon \|_{L^\infty} \, \| \nabla \Phi_\varepsilon \|_{L^\infty} \Big) \, \Sigma_\varepsilon^{k + 1}.
\end{equation}

We identically estimate the term $\boH_5$ so as to get
\begin{equation}
\label{eq:est-H5}
\big| \boH_5 \big| \leq C \varepsilon^2 \big( \| U_\varepsilon \|_{L^\infty} \, \| d^2 \Phi_\varepsilon \|_{L^\infty} + \| \nabla U_\varepsilon \|_{L^\infty} \, \| \nabla \Phi_\varepsilon \|_{L^\infty} \big) \, \Sigma_\varepsilon^{k + 1}.
\end{equation}

We now turn to the term $\boH_6$. We recall that
$$\nabla \Big( \frac{1}{\rho_\varepsilon} \Big) = \frac{2 \varepsilon^2 \, U_\varepsilon \, \nabla U_\varepsilon}{\rho_\varepsilon^2}.$$
Hence, we can write $\boH_6$ as
\begin{align*}
\boH_6 = \varepsilon^2 \int_{\R^N} \sum_{|\alpha| = \ell} \bigg( \frac{\nabla \rho_\varepsilon}{\rho_\varepsilon} \cdot & \Big( \partial_x^\alpha \Phi_\varepsilon \, \partial_x^\alpha \nabla U_\varepsilon + \partial_x^\alpha U_\varepsilon \, \partial_x^\alpha \nabla \Phi_\varepsilon \Big)\\
& + \varepsilon^2 \rho_\varepsilon \, \partial_x^\alpha \Phi_\varepsilon \nabla U_\varepsilon \cdot \Big( \frac{U_\varepsilon \, \partial_x^\alpha \nabla U_\varepsilon}{\rho_\varepsilon^2} - \partial_x^\alpha \Big( \frac{U_\varepsilon \, \nabla U_\varepsilon}{\rho_\varepsilon^2} \Big) \bigg).
\end{align*}
We can integrate by parts the first line of this identity in order to obtain
$$\varepsilon^2 \int_{\R^N} \sum_{|\alpha| = \ell} \frac{\nabla \rho_\varepsilon}{\rho_\varepsilon} \cdot \Big( \partial_x^\alpha \Phi_\varepsilon \, \partial_x^\alpha \nabla U_\varepsilon + \partial_x^\alpha U_\varepsilon \, \partial_x^\alpha \nabla \Phi_\varepsilon \Big) = - \varepsilon^2 \int_{\R^N} \sum_{|\alpha| = \ell} \div \Big( \frac{\nabla \rho_\varepsilon}{\rho_\varepsilon} \Big) \, \partial_x^\alpha \Phi_\varepsilon \, \partial_x^\alpha \, U_\varepsilon.$$
Combining this formula with the estimate
$$\Big\| \frac{U_\varepsilon \, \partial_x^\alpha \nabla U_\varepsilon}{\rho_\varepsilon^2} - \partial_x^\alpha \Big( \frac{U_\varepsilon \, \nabla U_\varepsilon}{\rho_\varepsilon^2} \Big) \Big\|_{L^2} \leq C \big( 1 + \varepsilon^2 \| U_\varepsilon \|_{L^\infty}^2 \big) \, \| \nabla U_\varepsilon \|_{L^\infty} \, \| U_\varepsilon \|_{\dot{H}^\ell},$$
we are led to the bound
\begin{equation}
\label{eq:est-H6}
\big| \boH_6 \big| \leq C \varepsilon^3 \Big( \big( 1 + \varepsilon^2 \| U_\varepsilon \|_{L^\infty}^2 \big) \, \| \nabla U_\varepsilon \|_{L^\infty}^2 + \| U_\varepsilon \|_{L^\infty} \, \| d^2 U_\varepsilon \|_{L^\infty} \big) \, \Sigma_\varepsilon^{k + 1}.
\end{equation}

We finally address the term $\boH_7$. A first integration by parts gives
\begin{align*}
\boH_7 = - \varepsilon^2 \int_{\R^N} \sum_{|\alpha| = \ell} \sum_{\substack{|\beta| = 1\\ \beta \leq \alpha}} \bigg( \frac{\partial_x^\beta \rho_\varepsilon}{\rho_\varepsilon} \, \big( \partial_x^{\alpha - \beta} \nabla & U_\varepsilon \cdot \partial_x^\alpha \nabla \Phi_\varepsilon + \partial_x^\alpha \nabla U_\varepsilon \cdot \partial_x^{\alpha - \beta} \nabla \Phi_\varepsilon \big)\\
& + \nabla \Big( \frac{\partial_x^\beta \rho_\varepsilon}{\rho_\varepsilon} \Big) \cdot \big( \partial_x^{\alpha - \beta} \nabla U_\varepsilon \, \partial_x^\alpha \Phi_\varepsilon + \partial_x^\alpha U_\varepsilon \,\partial_x^{\alpha - \beta} \nabla \Phi_\varepsilon \big) \bigg).
\end{align*}
Integrating by parts once again the first line provides
\begin{align*}
\boH_7 = \varepsilon^2 \int_{\R^N} \sum_{|\alpha| = \ell} \sum_{\substack{|\beta| = 1\\ \beta \leq \alpha}} \bigg( \partial_x^\beta \Big( \frac{\partial_x^\beta \rho_\varepsilon}{\rho_\varepsilon} \Big) \, \partial_x^{\alpha - \beta} \nabla & U_\varepsilon \cdot \partial_x^{\alpha - \beta} \nabla \Phi_\varepsilon\\
& - \nabla \Big( \frac{\partial_x^\beta \rho_\varepsilon}{\rho_\varepsilon} \Big) \cdot \big( \partial_x^{\alpha - \beta} \nabla U_\varepsilon \, \partial_x^\alpha \Phi_\varepsilon + \partial_x^\alpha U_\varepsilon \,\partial_x^{\alpha - \beta} \nabla \Phi_\varepsilon \big) \bigg).
\end{align*}
In view of~\eqref{est-W-1b},~\eqref{est-W-2},~\eqref{est-U1} and~\eqref{est-U2}, this can be controlled by
\begin{equation}
\label{eq:est-H7}
\big| \boH_7 \big| \leq C \varepsilon^3 \Big( \big( 1 + \varepsilon^2 \| U_\varepsilon \|_{L^\infty}^2 \big) \, \| \nabla U_\varepsilon \|_{L^\infty}^2 + \| U_\varepsilon \|_{L^\infty} \, \| d^2 U_\varepsilon \|_{L^\infty} \big) \, \Sigma_\varepsilon^{k + 1}.
\end{equation}
We finally put together all the estimates from~\eqref{eq:est-I5} to~\eqref{eq:est-H7}. The estimate in~\eqref{der-E-j} then follows from the bounds
$$\varepsilon \| U_\varepsilon \|_{L^\infty} \leq \frac{1}{\sqrt{2}} \quad {\rm and} \quad \| \sin(\Phi_\varepsilon) \|_{L^\infty} \leq 1,$$
and the inequality $2 a b \leq a^2 + b^2$. This concludes the proof of Proposition~\ref{prop:estimate}. \qed

%%%%%%%%%%%%%%%%%%%%%%%%%%%%%%%%%%%%%%%%%%%
%%%%%%%%%%%%%%%%%%%%%%%%%%%%%%%%%%%%%%%%%%%
\subsection{Proof of Corollary~\ref{cor:T}}
\label{sub:estim-eps}
%%%%%%%%%%%%%%%%%%%%%%%%%%%%%%%%%%%%%%%%%%%
%%%%%%%%%%%%%%%%%%%%%%%%%%%%%%%%%%%%%%%%%%%

Consider first an initial condition $(U_\varepsilon^0, \Phi_\varepsilon^0) \in \boN\boV^{k + 4}(\R^N)$ such that~\eqref{CI-petit} holds for some positive number $C$ to be fixed later. In view of Corollary~\ref{cor:HLL-Cauchy}, there exists a maximal time of existence $T_{\max}$, and a unique solution $(U_\varepsilon, \Phi_\varepsilon) \in \boC^0([0, T_{\max}), \boN\boV^{k + 3}(\R^N))$ to~\eqref{HLLeps} with initial datum $(U_\varepsilon^0, \Phi_\varepsilon^0)$, which satisfies all the statements in Corollary~\ref{cor:HLL-Cauchy}.

Invoking the Sobolev embedding theorem, we can find a positive number $K_1$, depending only on $k$ and $N$, such that 
$$\varepsilon \big\| U_\varepsilon^0 \big\|_{L^\infty} \leq \varepsilon K_1 \big\| U_\varepsilon ^0 \big\|_{H^k} \leq \frac{K_1}{C_*} < \frac{1}{\sqrt{2}},$$
when $C_* > \sqrt{2} K_1$. Setting $\rho_\varepsilon := 1 - \varepsilon^2 U_\varepsilon^2$ as before, we derive from the continuity of the solution in $\boN\boV^{k - 1}(\R^N)$ that the condition~\eqref{borne-W} in Proposition~\ref{prop:estimate} is fulfilled for any time small enough. Hence, the quantity $\Sigma_\varepsilon^{k + 1}$ is well-defined, and the stopping time 
\begin{equation}
\label{def:tau*}
T_* := \sup \Big\{ t \in [0, T_{\max}) : \frac{1}{2} \leq \inf_{x \in \R^N} \rho_\varepsilon(x, \tau) \ {\rm and} \ \Sigma_\varepsilon^{k + 1}(\tau) \leq 2 \Sigma_\varepsilon^{k + 1}(0) \ {\rm for} \ {\rm any} \ \tau \in [0, t] \Big\},
\end{equation}
is positive.

Moreover, the quantity $\Sigma_\varepsilon^{k + 1}$ is of class $\boC^1$ on $[0, T_*)$, and there exists a positive number $K_2$, depending only on $\sigma$, $k$ and $N$, such that
\begin{align*}
& \big[ \Sigma_\varepsilon^{k + 1} \big]'(t) \leq K_2 \, \Sigma_\varepsilon^{k + 1}(t) \, \Big( \| \sin(\Phi_\varepsilon(\cdot, t)) \|_{L^\infty}^2 + \| U_\varepsilon(\cdot, t) \|_{L^\infty}^2 + \| \nabla \Phi_\varepsilon(\cdot, t) \|_{L^\infty}^2 + \| \nabla U_\varepsilon(\cdot, t) \|_{L^\infty}^2\\
& + \| d^2 \Phi_\varepsilon(\cdot, t) \|_{L^\infty}^2 + \varepsilon^2 \| d^2 U_\varepsilon(\cdot, t) \|_{L^\infty}^2 + \varepsilon \, \| \nabla \Phi_\varepsilon(\cdot, t) \|_{L^\infty} \, \big( \| \nabla \Phi_\varepsilon(\cdot, t) \|_{L^\infty}^2 + \| \nabla U_\varepsilon(\cdot, t) \|_{L^\infty}^2 \big) \Big),
\end{align*}
for any $0 \leq t < T_*$. Since $k + 1 > N/2 + 2$, we can again invoke the Sobolev embedding theorem and use the bounds in~\eqref{borne-W},~\eqref{est-U1} and~\eqref{est-U2} in order to find a further number $K_3$, depending only on $\sigma$, $k$ and $N$, such that
\begin{equation}
\label{eq:Linfini-bound}
\begin{split}
\| \sin(\Phi_\varepsilon(\cdot, t)) \|_{L^\infty}^2 & + \| U_\varepsilon(\cdot, t) \|_{L^\infty}^2 + \| \nabla \Phi_\varepsilon(\cdot, t) \|_{L^\infty}^2\\
& + \| \nabla U_\varepsilon(\cdot, t) \|_{L^\infty}^2 + \| d^2 \Phi_\varepsilon(\cdot, t) \|_{L^\infty}^2 + \varepsilon^2 \| d^2 U_\varepsilon(\cdot, t) \|_{L^\infty}^2 \leq K_3 \Sigma_\varepsilon^{k + 1}(t).
\end{split}
\end{equation}
This gives
$$\big[ \Sigma_\varepsilon^{k + 1} \big]'(t) \leq K_2 K_3 \, \Big( \Sigma_\varepsilon^{k + 1}(t)^2 + \varepsilon K_3^\frac{1}{2} \Sigma_\varepsilon^{k + 1}(t)^\frac{5}{2} \Big),$$
for any $0 \leq t < T_*$. Coming back to~\eqref{def:tau*}, we next simplify this inequality as
$$\big[ \Sigma_\varepsilon^{k + 1} \big]'(t) \leq K_2 K_3 \, \Big( 1 + \big( 2 \varepsilon^2 K_3 \Sigma_\varepsilon^{k + 1}(0) \big)^\frac{1}{2} \Big)\, \Sigma_\varepsilon^{k + 1}(t)^2.$$
In view of~\eqref{borne-W} and~\eqref{CI-petit}, there also exists a number $K_4$, depending only on $\sigma$, $k$ and $N$, such that
\begin{equation}
\label{eq:petitpetit}
2 \varepsilon^2 K_3 \Sigma_\varepsilon^{k + 1}(0) \leq 2 K_3 K_4 \varepsilon^2 \Big( \big\| U_\varepsilon^0 \big\|_{H^k} + \varepsilon \big\| \nabla U_\varepsilon^0 \big\|_{H^k} + \big\| \nabla \Phi_\varepsilon^0 \big\|_{H^k} + \big\| \sin(\Phi_\varepsilon^0) \big\|_{H^k} \Big)^2 < 1,
\end{equation}
when $C_* > 2 K_3 K_4$. We conclude that
$$\big[ \Sigma_\varepsilon^{k + 1} \big]'(t) \leq 2 K_2 K_3 \, \Sigma_\varepsilon^{k + 1}(t)^2.$$

At this stage, we set
\begin{equation}
\label{eq:T*}
T_\varepsilon := \frac{1}{4 K_2 K_3 \Sigma_\varepsilon^{k + 1}(0)},
\end{equation}
and we deduce from the previous inequality that
$$\Sigma_\varepsilon^{k + 1}(t) \leq \frac{\Sigma_\varepsilon^{k + 1}(0)}{1 - 2 K_2 K_3 \Sigma_\varepsilon^{k + 1}(0) t} \leq 2 \Sigma_\varepsilon^{k + 1}(0),$$
when we additionally assume that $t < T_\varepsilon$. In view of~\eqref{eq:Linfini-bound} and~\eqref{eq:petitpetit}, we also have
$$\varepsilon \| U_\varepsilon(\cdot, t) \|_{L^\infty} \leq \varepsilon K_3^\frac{1}{2} \Sigma_\varepsilon^{k + 1}(0)^\frac{1}{2} < \frac{1}{\sqrt{2}},$$
so that
$$\inf_{x \in \R^N} \rho_\varepsilon(x, t) \geq \frac{1}{2}.$$
Finally, we derive as before from the Sobolev embedding theorem that
$$\int_0^t \bigg( \varepsilon^2 \bigg\| \frac{\nabla U_\varepsilon(\cdot, s)}{\rho_\varepsilon(\cdot, s)^\frac{1}{2}} \bigg\|_{L^\infty}^2 + \big\| \rho_\varepsilon(\cdot, s)^\frac{1}{2} \nabla \Phi_\varepsilon(\cdot, s) \big\|_{L^\infty}^2 \bigg) \, ds \leq K_6 \int_0^t \Sigma_\varepsilon^{k + 1}(s) \, ds \leq \frac{K_6}{2 K_2 K_3} < + \infty.$$
In view of the characterization for the maximal time $T_{\max}$ in Corollary~\ref{cor:HLL-Cauchy}, this guarantees that the stopping time $T_*$ is at least equal to $T_\varepsilon$. We finally derive from~\eqref{borne-W},~\eqref{est-U1} and~\eqref{est-U2} the existence of a number $K_5$, depending only on $\sigma$, $k$ and $N$, such that
\begin{align*}
\big\| U_\varepsilon(\cdot, t) \big\|_{H^k} + \varepsilon \big\| \nabla U_\varepsilon(\cdot, t) \big\|_{H^k} & + \big\| \nabla \Phi_\varepsilon(\cdot, t) \big\|_{H^k} + \big\| \sin(\Phi_\varepsilon(\cdot, t)) \big\|_{H^k}\\
\leq K_5 \Sigma_\varepsilon^{k + 1}(t)^\frac{1}{2} \leq & K_5 \big( 2 \Sigma_\varepsilon^{k + 1}(0) \big)^\frac{1}{2}\\
\leq & (2 K_4)^\frac{1}{2} K_5 \Big( \big\| U_\varepsilon^0 \big\|_{H^k} + \varepsilon \big\| \nabla U_\varepsilon^0 \big\|_{H^k} + \big\| \nabla \Phi_\varepsilon^0 \big\|_{H^k} + \big\| \sin(\Phi_\varepsilon^0) \big\|_{H^k} \Big).
\end{align*}
In view of~\eqref{eq:T*}, we similarly obtain
$$T_\varepsilon \geq \frac{1}{4 K_2 K_3 K_4 \big( \| U_\varepsilon^0 \|_{H^k} + \varepsilon \| \nabla U_\varepsilon^0 \|_{H^k} + \| \nabla \Phi_\varepsilon^0 \|_{H^k} + \| \sin(\Phi_\varepsilon^0) \|_{H^k} \big)^2}.$$
It then remains to suppose additionally that the number $C_*$ satisfies the conditions $C_* > (2 K_4)^\frac{1}{2} K_5$ and $C_* > 4 K_2 K_3 K_4$ in order to complete the proof of Corollary~\ref{cor:T} when $(U_\varepsilon^0, \Phi_\varepsilon^0) \in \boN\boV^{k + 4}(\R^N)$.

We finally rely on the continuity of the flow with respect to the initial datum in Corollary~\ref{cor:HLL-Cauchy} in order to extend Corollary~\ref{cor:T} to any arbitrary initial conditions $(U_\varepsilon^0, \Phi_\varepsilon^0) \in \boN\boV^{k + 2}(\R^N)$ by a density argument. \qed

%%%%%%%%%%%%%%%%%%%%%%%%%%%%%%%%%%%%%%%%%%%%
%%%%%%%%%%%%%%%%%%%%%%%%%%%%%%%%%%%%%%%%%%%%
\subsection{Proof of Lemma~\ref{lem:est-SG}}
\label{sub:est-SG}
%%%%%%%%%%%%%%%%%%%%%%%%%%%%%%%%%%%%%%%%%%%%
%%%%%%%%%%%%%%%%%%%%%%%%%%%%%%%%%%%%%%%%%%%%

When $(U^0, \Phi^0) \in H^k(\R^N) \times H_{\sin}^{k + 1}(\R^N)$, it follows from Theorem~\ref{thm:SG-Cauchy-smooth} that there exists a maximal time of existence $T_{\max}$, and a unique solution $(U, \Phi) \in \boC^0([0, T_{\max}), H^k(\R^N) \times H_{\sin}^{k + 1}(\R^N))$ to~\eqref{sys:SG} for this initial datum, with $(\partial_t U, \partial_t \Phi) \in \boC^0([0, T_{\max}), H^{k - 1}(\R^N) \times H^k(\R^N))$. In this case, the functions $t \mapsto E_{\rm SG}^\ell(U(\cdot, t), \Phi(\cdot, t))$ are well-defined and of class $\boC^1$ on $[0, T_{\max})$ for any $1 \leq \ell \leq k$. Moreover, we deduce from~\eqref{sys:SG}, an integration by parts and the Leibniz formula that
\begin{align*}
[E_{\rm SG}^\ell(U, \Phi)]'(t) = \sigma \sum_{|\alpha| = \ell - 1} \sum_{\substack{\beta \neq 0\\ \beta \leq \alpha}} \binom{\alpha}{\beta} \int_{\R^N} \partial_x^\beta \cos(\Phi(x, t)) \, \Big( & \partial_x^{\alpha - \beta} U(x, t) \, \partial_x^\alpha \sin(\Phi(x, t))\\
& - \partial_x^\alpha U(x, t) \, \partial_x^{\alpha - \beta} \sin(\Phi(x, t)) \Big) \, dx,
\end{align*}
for any $t \in [0, T_{\max})$. In particular, the Sine-Gordon energy $E_{\rm SG}^1 = E_{\rm SG}$ is conserved along the flow. When $\ell \geq 2$, we can invoke Corollaries~\ref{cor:moser} and~\ref{cor:cos} in order to check that
\begin{align*}
[E_{\rm SG}^\ell(U, \Phi)]'(t) \leq A \Big( \| U(\cdot, t) \|_{L^\infty}^2 + \| \nabla \Phi(\cdot, t) \|_{L^\infty}^2 & + \| \sin(\Phi(\cdot, t)) \|_{L^\infty}^2 \Big) \times\\
& \times \Big( [E_{\rm SG}^\ell(U, \Phi)](t) + [E_{\rm SG}^{\ell - 1}(U, \Phi)](t) \Big).
\end{align*}
Here as in the sequel, the positive number $A$ depends only on $\sigma$, $k$ and $N$. Using the Sobolev embedding theorem, we conclude that the quantity $\Sigma_{\rm SG}^k := \sum_{\ell = 1}^k E_{\rm SG}^\ell$ satisfies the inequality
$$[\Sigma_{\rm SG}^k(U, \Phi)]'(t) \leq A [\Sigma_{\rm SG}^k(U, \Phi)](t)^2,$$
for any $t \in [0, T_{\max})$. In particular, we have
$$\Sigma_{\rm SG}^k(U, \Phi)(t) \leq 2 \Sigma_{\rm SG}^k(U, \Phi)(0),$$
when $t \leq T_* := 1/(2 A \Sigma_{\rm SG}^k(U, \Phi)(0))$. In view of statements $(i)$ and $(iii)$ in Theorem~\ref{thm:SG-Cauchy-smooth}, we infer that $T_* \leq T_{\max}$, and we also obtain the existence of a positive number $A_*$, depending only on $\sigma$, $k$ and $N$, such that the solution $(U, \Phi)$ satisfies the statements in Lemma~\ref{lem:est-SG} for the number $T_*$. 

In order to conclude the proof of Lemma~\ref{lem:est-SG}, we finally invoke the continuity of the flow with respect to the initial datum in Theorem~\ref{thm:SG-Cauchy-smooth} in order to extend Lemma~\ref{lem:est-SG} to the solutions $(U, \Phi)$ to~\eqref{sys:SG} corresponding to any arbitrary initial conditions $(U^0, \Phi^0) \in H^{k - 1}(\R^N) \times H_{\sin}^k(\R^N)$ by a density argument. \qed

%%%%%%%%%%%%%%%%%%%%%%%%%%%%%%%%%%%%%%%%%%%%%%%%%%
%%%%%%%%%%%%%%%%%%%%%%%%%%%%%%%%%%%%%%%%%%%%%%%%%%
\subsection{Proof of Proposition~\ref{prop:error}}
\label{sub:error}
%%%%%%%%%%%%%%%%%%%%%%%%%%%%%%%%%%%%%%%%%%%%%%%%%%
%%%%%%%%%%%%%%%%%%%%%%%%%%%%%%%%%%%%%%%%%%%%%%%%%%

When $(U_\varepsilon^0, \Phi_\varepsilon^0) \in \boN\boV^{k + 2}(\R^N)$, the pair $(U_\varepsilon, \Phi_\varepsilon)$ lies in $\boC^0([0, T], \boN\boV^{k + 1}(\R^N))$, so that the quantity $\boK_\varepsilon(T)$ is well-defined. Moreover, it follows from the Sobolev embedding theorem that
\begin{equation}
\label{eq:unif-eps}
\begin{split}
\max_{t \in [0, T]} \Big( \| U_\varepsilon(\cdot, t) \|_{L^\infty} + & \| \nabla U_\varepsilon(\cdot, t) \|_{L^\infty} + \varepsilon \| d^2 U_\varepsilon(\cdot, t) \|_{L^\infty}\\
& + \| \sin(\Phi_\varepsilon(\cdot, t)) \|_{L^\infty} + \| \nabla \Phi_\varepsilon(\cdot, t) \|_{L^\infty} + \| d^2 \Phi_\varepsilon(\cdot, t) \|_{L^\infty} \Big) \leq C \boK_\varepsilon(T).
\end{split}
\end{equation}
Here as in the sequel, the positive number $C$ depends only on $\sigma$, $k$ and $N$. Combining this inequality with the Moser estimates in Lemma~\ref{lem:moser} and Corollary~\ref{cor:cos}, we are led to
\begin{equation}
\label{eq:est-RepsU}
\max_{t \in [0, T]} \| R_\varepsilon^U(\cdot, t) \|_{H^{k - 1}} \leq C \boK_\varepsilon(T)^3 \, \big( 1 + \boK_\varepsilon(T) \big).
\end{equation}
Similarly, we derive
\begin{equation}
\label{eq:est-RepsPhi}
\max_{t \in [0, T]} \| R_\varepsilon^\Phi(\cdot, t) \|_{H^{k - 2}} \leq C \boK_\varepsilon(T) \, \big( 1 + (1 + \varepsilon^2) \boK_\varepsilon(T)^2 \big).
\end{equation}
Here, we have also used~\eqref{est-W-2} and~\eqref{est-W-3}, which remain available due to condition~\eqref{eq:unif-bound}. With these estimates at hand, we are in position to establish the three statements of Proposition~\ref{prop:error}.

\setcounter{step}{0}
\begin{step}
\label{E1}
Proof of~\eqref{error1}.
\end{step}

In view of statement $(iv)$ in Corollary~\ref{cor:HLL-Cauchy}, and Theorem~\ref{thm:SG-Cauchy-smooth}, the functions $\Phi_\varepsilon$ and $\Phi$ are in $\boC^0([0, T], \Phi_\varepsilon^0 + L^2(\R^N))$, respectively $\boC^0([0, T], \Phi^0 + L^2(\R^N))$. Since $\Phi_\varepsilon^0 - \Phi^0 \in L^2(\R^N)$, the function $\varphi_\varepsilon$ belongs to $\boC^0([0, T], L^2(\R^N))$. Moreover, we can write the Duhamel formula corresponding to~\eqref{eq:sys-diff} in order to obtain the identity
\begin{align*}
\varphi_\varepsilon(\cdot, t) = \cos(t D) \, \varphi_\varepsilon^0 + \frac{\sin(t D)}{D} \, v_\varepsilon^0 + & \int_0^t \bigg( \varepsilon^2 \cos((t - s) D) R_\varepsilon^\Phi(\cdot, s) + \frac{\sin((t - s) D)}{D} \times\\
& \times \Big( \varepsilon^2 R_\varepsilon^U(\cdot, s) - \sigma \sin \big( \varphi_\varepsilon(\cdot, s) \big) \, \cos \big( \Phi_\varepsilon(\cdot, s) + \Phi(\cdot, s) \big) \Big) \bigg) \, ds,
\end{align*}
for any $t \in [0, T]$. Here, we have set $D = \sqrt{- \Delta}$ as before. In view of~\eqref{eq:est-RepsU} and~\eqref{eq:est-RepsPhi}, this provides
\begin{align*}
\| \varphi_\varepsilon(\cdot, t) \|_{L^2} \leq & \| \varphi_\varepsilon^0 \|_{L^2} + C t \| v_\varepsilon^0 \|_{L^2} + C \int_0^t (t - s) \, \| \varphi_\varepsilon(\cdot, s) \|_{L^2} \, ds\\
& + C \varepsilon^2 \boK_\varepsilon(T) t \big( 1 + (1 + \varepsilon^2 + t) \boK_\varepsilon(T)^2 + t \boK_\varepsilon(T)^3 \big).
\end{align*}
Set $X(t) := \| \varphi_\varepsilon^0 \|_{L^2} + C \int_0^t (t - s) \, \| \varphi_\varepsilon(\cdot, s) \|_{L^2} \, ds$, and $F(t) := C t \| v_\varepsilon^0 \|_{L^2} + C \varepsilon^2 \boK_\varepsilon(T) t \big( 1 + (1 + \varepsilon^2 + t) \boK_\varepsilon(T)^2 + t \boK_\varepsilon(T)^3 \big)$. The function $X$ is of class $\boC^2$ on $[0, T]$, and it satisfies
$$X''(t) = C \| \varphi_\varepsilon(\cdot, t) \|_{L^2} \leq C \big( X(t) + F(t) \big),$$
for any $t \in [0, T]$. Since $X(0) = \| \varphi_\varepsilon^0 \|_{L^2}$ and $X'(0) = 0$, integrating this differential inequality yields
$$X(t) \leq \| \varphi_\varepsilon^0 \|_{L^2} \, \cosh \big( \sqrt{C} t \big) + \sqrt{C} \int_0^t F(s) \, \sinh \big( \sqrt{C}(t - s) \big) \, ds,$$
so that
$$\| \varphi_\varepsilon(\cdot, t) \|_{L^2} \leq \| \varphi_\varepsilon^0 \|_{L^2} \, \cosh \big( \sqrt{C} t \big) + F(t) + \sqrt{C} \int_0^t F(s) \, \sinh \big( \sqrt{C}(t - s) \big) \, ds.$$
Estimate~\eqref{error1} then follows from the identities
$$C \int_0^t s \, \sinh \big( \sqrt{C}(t -s) \big) \, ds = \sinh(\sqrt{C} t) - \sqrt{C} t,$$
and
$$C \sqrt{C }\int_0^t s^2 \, \sinh \big( \sqrt{C}(t -s) \big) \, ds = 2 \cosh(\sqrt{C} t) - 2 - C t^2.$$

\begin{step}
\label{E2}
Proof of~\eqref{error2}.
\end{step}

Assume first that $(U^0, \Phi^0) \in H^1(\R^N) \times H_{\sin}^2(\R^N)$. In this case, the pair $(\partial_t U, \partial_t \Phi)$ lie in $\boC^0(\R, L^2(\R^N) \times H^1(\R^N))$ by Theorem~\ref{thm:SG-Cauchy}. Since $(\partial_t U_\varepsilon, \partial_t \Phi_\varepsilon)$ belongs to $\boC^0([0, T], H^k(\R^N)^2)$ when $(U_\varepsilon^0, \Phi_\varepsilon^0) \in \boN\boV^{k + 2}(\R^N)$, it follows that the energy $\gE_{\rm SG}^1$ in~\eqref{def:gE-SG-k} is of class $\boC^1$ on $[0, T]$. Moreover, we have
\begin{align*}
& \big[ \gE_{\rm SG}^1 \big]'(t) = \sigma \int_{\R^N} v_\varepsilon(x, t) \, \sin(\varphi_\varepsilon(x, t)) \, \Big( \cos(\varphi_\varepsilon(x, t)) - \cos \big( \Phi_\varepsilon(x, t) + \Phi(x, t) \big) \Big) \, dx\\
+ & \varepsilon^2 \int_{\R^N} \big( R_\varepsilon^U(x, t) \, v_\varepsilon(x, t) + \nabla R_\varepsilon^\Phi(x, t) \cdot \nabla \varphi_\varepsilon(x, t) + \sigma \sin(\varphi_\varepsilon(x, t)) \, \cos(\varphi_\varepsilon(x, t)) \, R_\varepsilon^\Phi(x, t) \big) \, dx,
\end{align*}
for any $t \in [0, T]$. We bound the first integral in this identity by
$$\sigma \int_{\R^N} v_\varepsilon(x, t) \, \sin(\varphi_\varepsilon(x, t)) \, \big( \cos(\varphi_\varepsilon(x, t)) - \cos(\Phi_\varepsilon(x, t) + \Phi(x, t)) \big) \, dx \leq 2 \sigma^\frac{1}{2} \gE_{\rm SG}^1(t),$$
whereas the second integral is controlled by
\begin{align*}
\varepsilon^2 \int_{\R^N} \big( R_\varepsilon^U(x, t) \, v_\varepsilon(x, t) + & \nabla R_\varepsilon^\Phi(x, t) \cdot \nabla \varphi_\varepsilon(x, t) + \sigma \sin(\varphi_\varepsilon(x, t)) \, \cos(\varphi_\varepsilon(x, t)) \, R_\varepsilon^\Phi(x, t) \big) \, dx\\
\leq & \gE_{\rm SG}^1(t) + \frac{\varepsilon^4}{2} \big( \| R_\varepsilon^U(\cdot, t) \|_{L^2}^2 + \| \nabla R_\varepsilon^\Phi(\cdot, t) \|_{L^2}^2 + \sigma \| R_\varepsilon^\Phi(\cdot, t) \|_{L^2}^2 \big).
\end{align*}
We then deduce from the Gronwall lemma that
$$\gE_{\rm SG}^1(t) \leq C \Big( \gE_{\rm SG}^1(0) + \varepsilon^4 \max_{t \in [0, T]} \big( \| R_\varepsilon^U(\cdot, t) \|_{L^2}^2 + \sigma \| R_\varepsilon^\Phi(\cdot, t) \|_{L^2}^2 + \| \nabla R_\varepsilon^\Phi(\cdot, t) \|_{L^2}^2 \big) \Big) \, e^{C t}.$$
When $N \geq 2$, or $N = 1$ and $k > N/2 + 2$, we can control uniformly with respect to $\varepsilon$ the right-hand side of this inequality by~\eqref{eq:est-RepsU} and~\eqref{eq:est-RepsPhi}. This leads to the bound in~\eqref{error2} when $(U^0, \Phi^0) \in H^1(\R^N) \times H_{\sin}^2(\R^N)$. We then complete the proof of~\eqref{error2} by a standard density argument.

\begin{step}
\label{E3}
Proof of~\eqref{errork}.
\end{step}

Let $2 \leq \ell \leq k - 2$. Since the pair $(\partial_t v_\varepsilon, \partial_t \varphi_\varepsilon)$ belongs to $\boC^0([0, T], H^{k - 1}(\R^N)^2)$, we can differentiate the quantities $\gE_{\rm SG}^\ell$ in~\eqref{def:gE-SG-k} and invoke~\eqref{eq:sys-diff} in order to obtain
\begin{equation}
\label{eq:der-gE}
\begin{split}
\big[ \gE_{\rm SG}^\ell \big]'(t) & := \sigma \sum_{|\alpha| = \ell - 1} \int_{\R^N} \Big( \partial_x^\alpha \sin(\varphi_\varepsilon) \, \partial_x^\alpha \big( v_\varepsilon \, \cos(\varphi_\varepsilon) \big) - \partial_x^\alpha v_\varepsilon \, \partial_x^\alpha \big( \sin(\varphi_\varepsilon) \, \cos(\Phi_\varepsilon + \Phi) \big) \Big)(x, t) \, dx\\
+ & \varepsilon^2 \sum_{|\alpha| = \ell - 1} \int_{\R^N} \Big( \partial_x^\alpha v_\varepsilon \, \partial_x^\alpha R_\varepsilon^U + \partial_x^\alpha \nabla \varphi_\varepsilon \cdot \partial_x^\alpha \nabla R_\varepsilon^\Phi + \sigma \partial_x^\alpha \sin(\varphi_\varepsilon) \, \partial_x^\alpha \big( \cos(\varphi_\varepsilon) \, R_\varepsilon^\Phi \big) \Big)(x, t) \, dx,
\end{split}
\end{equation}
after an integration by parts. In order to bound the various terms in the right-hand side of this identity, we first apply the Sobolev embedding theorem in order to get the bound
$$\max_{t \in [0, T]} \Big( \| U(\cdot, t) \|_{L^\infty} + \| \sin(\Phi(\cdot, t)) \|_{L^\infty} + \| \nabla \Phi(\cdot, t) \|_{L^\infty} \Big) \leq C \kappa_\varepsilon(T).$$
Combining~\eqref{eq:unif-eps} with Corollary~\ref{cor:cos}, this gives
\begin{align*}
\big\| \cos(\Phi_\varepsilon(\cdot, t) \pm \Phi(\cdot, t)) \big\|_{\dot{H}^{\ell - 1}} \leq & C \Big( \big\| \sin(\Phi_\varepsilon(\cdot, t) \pm \Phi(\cdot, t)) \big\|_{L^\infty} \, \big\| \nabla \Phi_\varepsilon(\cdot, t) \pm \nabla \Phi(\cdot, t) \big\|_{\dot{H}^{\ell - 1}}\\
& + \big\| \nabla \Phi_\varepsilon(\cdot, t) \pm \nabla \Phi(\cdot, t) \big\|_{L^\infty} \, \big\| \sin(\Phi_\varepsilon(\cdot, t) \pm \Phi(\cdot, t)) \big\|_{\dot{H}^{\ell - 2}} \Big)\\
\leq & C \kappa_\varepsilon(T)^2.
\end{align*}
Setting $\gS^{k - 2} := \sum_{\ell = 1}^{k - 2} \gE_{\rm SG}^\ell$ and assuming that $k > N/2 + 3$, we also derive from the Sobolev embedding theorem that
$$\big\| v_\varepsilon(\cdot, t) \big\|_{L^\infty} + \big\| \sin(\varphi_\varepsilon(\cdot, t)) \big\|_{L^\infty} + \big\| \nabla \varphi_\varepsilon(\cdot, t) \big\|_{L^\infty} \leq C \gS^{k - 2}(t)^\frac{1}{2}.$$
As a consequence of Lemma~\ref{lem:moser}, we are led to the following estimate of the integrals in the first line of~\eqref{eq:der-gE}
\begin{align*}
\bigg| \int_{\R^N} \Big( \partial_x^\alpha \sin(\varphi_\varepsilon) \, \partial_x^\alpha \big( v_\varepsilon \, \cos(\varphi_\varepsilon) \big) - \partial_x^\alpha v_\varepsilon \, \partial_x^\alpha \big( \sin(\varphi_\varepsilon) \, \cos(\Phi_\varepsilon + & \Phi) \big) \Big)(x, t) \, dx \bigg|\\
\leq & C \big( 1 + \kappa_\varepsilon(T)^2 \big) \, \gS^{k - 2}(t).
\end{align*}
Concerning the second line, we similarly check that
\begin{align*}
& \varepsilon^2 \bigg| \int_{\R^N} \Big( \partial_x^\alpha v_\varepsilon \, \partial_x^\alpha R_\varepsilon^U + \partial_x^\alpha \nabla \varphi_\varepsilon \cdot \partial_x^\alpha \nabla R_\varepsilon^\Phi + \sigma \partial_x^\alpha \sin(\varphi_\varepsilon) \, \partial_x^\alpha \big( \cos(\varphi_\varepsilon) \, R_\varepsilon^\Phi \big) \Big)(x, t) \, dx \bigg|\\
\leq & C \Big( \gE_{\rm SG}^\ell(t) + \varepsilon^4 \big( \| R_\varepsilon^U(\cdot, t) \|_{\dot{H}^{\ell - 1}}^2 + \| \nabla R_\varepsilon^\Phi(\cdot, t) \|_{\dot{H}^{\ell - 1}}^2 + \kappa_\varepsilon(T)^4 \, \| R_\varepsilon^\Phi(\cdot, t) \|_{L^\infty}^2 + \| R_\varepsilon^\Phi(\cdot, t) \|_{\dot{H}^{\ell - 1}}^2 \Big).
\end{align*}
Using~\eqref{eq:est-RepsU},~\eqref{eq:est-RepsPhi} and the Sobolev embedding theorem, we finally obtain
$$\big[ \gE_{\rm SG}^\ell \big]'(t) \leq C \Big( \big( 1 + \kappa_\varepsilon(T)^2 \big) \gS^{k - 2}(t) + \varepsilon^4 \kappa_\varepsilon(T)^2 \big( 1 + \kappa_\varepsilon(T)^8 + \varepsilon^4 \kappa_\varepsilon(T)^4 \big( 1 + \kappa_\varepsilon(T)^4 \big) \big) \Big).$$
In view of Step~\ref{E2}, this inequality also holds for $\ell = 1$. Hence, we have
\begin{align*}
\big[ \gS^{k - 2} \big]'(t) \leq & C \big( 1 + \kappa_\varepsilon(T)^2 \big) \, \gS^{k - 2}(t)\\
& + C \varepsilon^4 \, \kappa_\varepsilon(T)^2 \, \Big( 1 + \kappa_\varepsilon(T)^8 + \varepsilon^4 \, \kappa_\varepsilon(T)^4 \, \big( 1 + \kappa_\varepsilon(T)^4 \big) \Big).
\end{align*}
Estimate~\eqref{errork} is then a direct consequence of the Gronwall lemma. This concludes the proof of Proposition~\ref{prop:error}. \qed

%%%%%%%%%%%%%%%%%%%%%%%%%%%%%%%%%%%%%%%%%%%%%
%%%%%%%%%%%%%%%%%%%%%%%%%%%%%%%%%%%%%%%%%%%%%
\section{The derivation of the wave equation}
\label{sec:wave}
%%%%%%%%%%%%%%%%%%%%%%%%%%%%%%%%%%%%%%%%%%%%%
%%%%%%%%%%%%%%%%%%%%%%%%%%%%%%%%%%%%%%%%%%%%%

Our aim is now to prove Theorem~\ref{thm:conv-wave}, which shows that the dynamics of the Landau-Lifshitz equation can be approximated by the free wave equation as $\varepsilon, \sigma \to 0$. This relies on arguments and estimates similar to the ones developed in the previous sections, with some modifications so as to take into account the smallness of the parameter $\sigma$.

First, we could use Proposition~\ref{prop:estimate} in order to control higher order derivatives. However, if $\sigma$ is small, we can obtain better estimates by considering the energy of order $k \geq 2$
$$\tilde{E}_\varepsilon^k(U_{\varepsilon, \sigma}, \Phi_{\varepsilon, \sigma}) := \frac{1}{2} \sum_{|\alpha| = k - 1} \int_{\R^N} \Big( \varepsilon^2 \frac{|\nabla \partial_x^\alpha U_{\varepsilon, \sigma}|^2}{1 - \varepsilon^2 U_{\varepsilon, \sigma}^2} + |\partial_x^\alpha U_{\varepsilon, \sigma}|^2 + (1 - \varepsilon^2 U_{\varepsilon, \sigma}^2) |\nabla \partial_x^\alpha \Phi_{\varepsilon, \sigma}|^2 \Big).$$
Setting $\tilde{E}_\varepsilon^1 := E_\varepsilon(U_{\varepsilon, \sigma}, \Phi_{\varepsilon, \sigma})$, and
\begin{equation}
\label{def:tilde-sigma}
\tilde{\Sigma}_\varepsilon^k := \sum_{j = 1}^k \tilde{E}_\varepsilon^j, 
\end{equation}
for $k \geq 1$, we are led to the following estimates.

\begin{prop}
\label{prop:est-FW}
Let $\varepsilon < 1$ and $\sigma$ be fixed positive numbers, and $k \in \N$, with $k > N/2 + 1$.
Consider a solution $(U_{\varepsilon, \sigma}, \Phi_{\varepsilon, \sigma})$ to~\eqref{HLLeps}, with $(U_{\varepsilon, \sigma}, \Phi_{\varepsilon, \sigma}) \in \boC^0([0, T], \boN\boV^{k + 3}(\R^N))$ for a fixed positive number $T$, and assume that 
\begin{equation}
\label{eq:borne-W-ter}
\inf_{\R^N \times [0, T]} 1 - \varepsilon^2 U_{\varepsilon, \sigma}^2 \geq \frac{1}{2}.
\end{equation}
There exists a positive number $C$, depending only on $k$ and $N$, such that
\begin{equation}
\label{der-E-j-bis}
\begin{split}
\big[ \tilde{E}_\varepsilon^\ell \big]'(t) \leq C \tilde{\Sigma}_\varepsilon^{k + 1}(t) \Big( & \varepsilon \| \nabla \Phi_{\varepsilon, \sigma}(\cdot, t) \|_{W^{1, \infty}}^2 + \varepsilon \| U_{\varepsilon, \sigma}(\cdot, t) \|_{W^{1, \infty}}^2\\
& + \varepsilon^3 \| d^2 U_{\varepsilon, \sigma}(\cdot, t) \|_{L^{\infty}}^2 +\sigma \big( 1 + \| \nabla \Phi_{\varepsilon, \sigma}(\cdot, t) \|_{L^\infty}^k \big) \Big),
\end{split}
\end{equation}
for any $t \in [0, T]$, and any $2 \leq \ell \leq k + 1$. 
\end{prop}

\begin{proof}
We proceed as in the the proof of Proposition~\ref{prop:estimate} and keep the same notation, for which we have 
$$\big[ \tilde{E}_\varepsilon^\ell \big]'(t) = \boI_1 + \boI_2 + \boI_3 + \tilde{\boI}_5,$$
where $\tilde{\boI}_5$ is equal to $\boI_5$ without the last term $\sigma |\partial_x^\alpha \sin(\Phi_\varepsilon)|^2$. Using~\eqref{eq:borne-W-ter}, the inequality $2 a b \leq a^2 + b^2$, and the fact that $0 < \varepsilon < 1$, we deduce from the proof of Proposition~\ref{prop:estimate} that 
$$\big| \boI_{1, 1} + \boI_{1, 2} + \boI_{2, 1} + \boI_{2, 2} + \boI_{3, 1} + \boI_{3, 2} + \tilde{\boI}_5 \big| \leq C \Big( \varepsilon \| \nabla \Phi_{\varepsilon, \sigma} \|_{W^{1, \infty}}^2 + \varepsilon \| U_{\varepsilon, \sigma} \|_{W^{1, \infty}}^2 + \varepsilon^3 \| d^2 U_{\varepsilon, \sigma} \|_{L^{\infty}}^2 \Big) \, \tilde{\Sigma}_\varepsilon^{k + 1}.$$
In order to estimate the remaining terms, we rely on Corollary~\ref{cor:moser-trigo}. In view of Corollary~\ref{cor:moser}, we obtain 
\begin{align*}
\big| \boI_{1, 3} \big| & \leq C \sigma \Big( \| \rho_{\varepsilon, \sigma} \|_{L^\infty} \, \| \sin(2 \Phi_{\varepsilon, \sigma}) \|_{\dot{H}^{\ell - 1}} + \| \rho_{\varepsilon, \sigma} \|_{\dot{H}^{\ell - 1}} \, \| \sin(2 \Phi_{\varepsilon, \sigma}) \|_{L^\infty} \big) \, \big( \tilde{\Sigma}_\varepsilon^{k + 1} \big)^\frac{1}{2}\\
& \leq C \sigma \Big( 1 + \| \nabla \Phi_{\varepsilon, \sigma} \|_{L^\infty}^{k - 1} + \varepsilon^2 \| U_{\varepsilon, \sigma} \|_{L^\infty} \| \sin(\Phi_{\varepsilon, \sigma}) \|_{L^\infty} \Big) \, \tilde{\Sigma}_\varepsilon^{k+1}.
\end{align*}
Proceeding in a similar way, we also get
$$\big| \boI_{2, 3} \big| + \big| \boI_{3, 3} \big| \leq C \varepsilon \sigma \big( 1 + \| \nabla \Phi_{\varepsilon, \sigma} \|_{L^\infty}^k \big) \, \tilde{\Sigma}_\varepsilon^{k + 1}.$$
This completes the proof of~\eqref{der-E-j-bis}.
\end{proof}

In order to state the consequences of Proposition~\ref{prop:est-FW}, we introduce the quantity
$$\boK_{\varepsilon, \sigma}(t) := \big\| U_{\varepsilon, \sigma}(\cdot,t) \big\|_{H^k} + \varepsilon \big\| \nabla U_{\varepsilon, \sigma}(\cdot,t) \big\|_{H^k} + \big\| \nabla \Phi_{\varepsilon, \sigma}(\cdot,t) \big\|_{H^k} + \sigma^\frac{1}{2} \big\| \sin(\Phi_{\varepsilon,\sigma}(\cdot,t) ) \big\|_{L^2}.$$

\begin{cor}
\label{cor:T-bis}
Let $\varepsilon < 1$ and $\sigma$ be fixed positive numbers, and $k \in \N$, with $k > N/2 + 1$. There exists a number $A_* \geq 1$ such that, if an initial condition $(U_{\varepsilon, \sigma}^0, \Phi_{\varepsilon, \sigma}^0) \in \boN\boV^{k + 2}(\R^N)$ satisfies 
$$A_* \varepsilon \boK_{\varepsilon, \sigma}(0) \leq 1,$$
then there exists a positive time
$$T_{\varepsilon, \sigma} \geq \frac{1}{A_* \max \{ \sigma, \varepsilon \} \big( 1 + \tilde{\Sigma}_\varepsilon^{k + 1}(0) \big)^{\max \{ 2, \frac{k}{2} \}}},$$
such that the maximal time of existence of the solution $(U_{\varepsilon, \sigma}, \Phi_{\varepsilon,\sigma})$ to~\eqref{HLLeps} with initial condition $(U_{\varepsilon, \sigma}^0, \Phi_{\varepsilon, \sigma}^0)$ is greater than $T_{\varepsilon, \sigma}$. Moreover, we have
$$\sqrt{2} \varepsilon \| U_{\varepsilon, \sigma}(\cdot, t) \|_{L^\infty} \leq 1,$$
and
$$\boK_{\varepsilon, \sigma}(t) \leq A_* \boK_{\varepsilon, \sigma}(0),$$
for any $t \in [0, T_{\varepsilon, \sigma}]$.
\end{cor}

\begin{proof}
The proof follows the same lines as the proof of Corollary~\ref{cor:T}. Indeed, the same arguments show that the stopping time
$$\tilde{T}_* := \sup \Big\{ t \in [0, T_{\max}) : \frac{1}{2} \leq \inf_{x \in \R^N} \rho_{\varepsilon, \sigma}(x, \tau) \ {\rm and} \ \tilde{\Sigma}_\varepsilon^{k + 1}(\tau) \leq 2 \tilde{\Sigma}_\varepsilon^{k + 1}(0) \ {\rm for} \ {\rm any} \ \tau \in [0, t] \Big\},$$
is positive, and that there exist two positive numbers $K_1$ and $K_2$, depending only on $k$ and $N$, such that 
\begin{equation}
\label{equivalence}
K_1 \tilde{\Sigma}_\varepsilon^{k + 1}(0) \leq \boK_{\varepsilon, \sigma}(0)^2 \leq K_2 \tilde{\Sigma}_\varepsilon^{k + 1}(0).
\end{equation}
Since $\tilde{\Sigma}_\varepsilon^1$ is constant in time, we infer from Proposition~\ref{prop:est-FW} and~\eqref{def:tilde-sigma} that the following differential inequality holds
$$\big[ \tilde{\Sigma}_\varepsilon^{k + 1} \big]'(t) \leq K_3 \max \big\{ \sigma, \varepsilon \big\} \tilde{\Sigma}_\varepsilon^{k + 1}(t) \, \Big( 1 + \tilde{\Sigma}_\varepsilon^{k + 1}(t)^2 + \tilde{\Sigma}_\varepsilon^{k + 1}(t)^\frac{k}{2} \Big),$$
for any $t \in [0,\tilde{T}_*)$ and a further positive number $K_3$. In view of the definition of $\tilde{T}_*$ and the fact that $k \geq 2$, we can enlarge $K_3$ such that
$$\big[ \tilde{\Sigma}_\varepsilon^{k + 1} \big]'(t) \leq K_3 \max \big\{ \sigma, \varepsilon \big\} \tilde{\Sigma}_\varepsilon^{k + 1}(t) \, \Big( 1 + \tilde{\Sigma}_\varepsilon^{k + 1}(0) \Big)^{\max \{ 2, \frac{k}{2} \}},$$
for any $t\in [0, \tilde{T}_*)$. Setting
$$T_{\varepsilon, \sigma} := \frac{\ln(2)}{K_3 \max \{ \sigma, \varepsilon \} \big( 1 + \tilde{\Sigma}_\varepsilon^{k + 1}(0) \big)^{\max \{ 2, \frac{k}{2} \}}},$$
we conclude as in Corollary~\ref{cor:T} that $\tilde{T}_* \geq T_{\varepsilon, \sigma}$. Bearing in mind~\eqref{equivalence}, the other statements follow also as in Corollary~\ref{cor:T}.
\end{proof}

We now conclude the

\begin{proof}[Proof of Theorem~\ref{thm:conv-wave}]
The proof is simpler than the one of Theorem~\ref{thm:conv-SG}. Let $T_{\varepsilon, \sigma}$ be the time given by Corollary~\ref{cor:T-bis}, and fix $T \in [0, T_{\varepsilon, \sigma}]$. Setting $v_{\varepsilon, \sigma} := U_{\varepsilon, \sigma} - U$ and $\varphi_{\varepsilon, \sigma} := \Phi_{\varepsilon, \sigma} - \Phi$, we derive from~\eqref{HLLeps} and~\eqref{FW} that
\begin{equation}
\label{eq:diff-sigma}
\begin{cases}
\partial_t v_{\varepsilon, \sigma} = \Delta \varphi_{\varepsilon, \sigma} - \frac{\sigma}{2} \sin(2 \Phi_{\varepsilon, \sigma}) + \varepsilon^2 R_\varepsilon^U,\\
\partial_t \varphi_{\varepsilon, \sigma} = v_{\varepsilon, \sigma} + \varepsilon^2 R_\varepsilon^\Phi,
\end{cases}
\end{equation}
with $R_\varepsilon^U$ and $R_\varepsilon^\Phi$ as in~\eqref{def:RepsU} and~\eqref{def:RepsPhi}. We next deduce from Corollary~\ref{cor:T-bis} and the Sobolev embedding theorem that
\begin{equation}
\label{est:unif}
\begin{split}
\max_{t \in [0, T]} \Big( \big\| U_{\varepsilon, \sigma}(\cdot, t) \big\|_{L^\infty} + & \big\| \nabla U_{\varepsilon, \sigma}(\cdot, t) \big\|_{L^\infty} + \varepsilon \big\| d^2 U_{\varepsilon, \sigma}(\cdot, t) \big\|_{L^\infty} + \big\| \nabla \Phi_{\varepsilon, \sigma}(\cdot, t) \big\|_{L^\infty} \Big) \leq C \boK_{\varepsilon, \sigma}(0),
\end{split}
\end{equation}
where $C$ refers, here and in the sequel, to a positive number depending only on $N$ and $k$.

Let $0 \leq m \leq k - 2$. We have
\begin{equation}
\label{est:sin}
\sigma^\frac{1}{2} \max_{t \in [0, T]} \big\| \sin(2 \Phi_{\varepsilon, \sigma}(\cdot, t)) \big\|_{L^2} \leq \boK_{\varepsilon, \sigma}(0),
\end{equation}
and we infer from~\eqref{est:unif}, Lemma~\ref{lem:moser} and Corollary~\ref{cor:moser-trigo} that
\begin{equation}
\label{est:sin2}
\max_{t \in [0, T]} \big\| \sin(2 \Phi_{\varepsilon, \sigma}(\cdot, t)) \big\|_{\dot{H}^\ell} \leq C \big( \boK_{\varepsilon, \sigma}(0) + \boK_{\varepsilon, \sigma}(0)^\ell \big),
\end{equation}
for any $1 \leq \ell \leq m$, and
\begin{equation}
\label{est:R-phi}
\max_{t \in [0, T]} \big\| R_\varepsilon^\Phi(\cdot, t) \big\|_{H^m} \leq C \big( \boK_{\varepsilon, \sigma}(0) + \boK_{\varepsilon, \sigma}(0)^3 + \sigma \boK_{\varepsilon, \sigma}(0)^{m + 1} \big).
\end{equation}
Similarly, we have
$$\max_{t \in [0, T]} \big\| R_\varepsilon^U(\cdot, t) \big\|_{L^2} \leq C \big( \sigma \boK_{\varepsilon, \sigma}(0)^{2} + \boK_{\varepsilon, \sigma}(0)^3 \big),$$
and
$$\max_{t \in [0, T]} \big\| R_\varepsilon^U(\cdot, t) \big\|_{\dot{H}^{\ell - 1}} \leq C \big( \sigma \boK_{\varepsilon, \sigma}(0)^2 + \boK_{\varepsilon, \sigma}(0)^3 + \sigma \boK_{\varepsilon, \sigma}(0)^{\ell + 1} \big).$$
for $2 \leq \ell \leq m$. Therefore, using the embedding $L^{2}(\R^N)$ into $H^{- 1}(\R^N)$ if $m = 0$, we conclude that 
\begin{equation}
\label{est-R}
\max_{t \in [0, T]} \big\| R_\varepsilon^U(\cdot, t) \big\|_{H^{m - 1}} \leq C \big( \sigma \boK_{\varepsilon, \sigma}(0)^2 + \boK_{\varepsilon, \sigma}(0)^3 + \sigma \boK_{\varepsilon, \sigma}(0)^{m + 1} \big).
\end{equation}
In view of~\eqref{eq:diff-sigma}, the Duhamel formulation is given by
\begin{align*}
v_{\varepsilon, \sigma}(\cdot, t) & = \cos(t D) v_{\varepsilon, \sigma}^0 - D \sin(t D) \varphi_{\varepsilon,\sigma}^0\\
+ & \int_0^t \bigg( \cos((t - s) D) \Big( \varepsilon^2 R_\varepsilon^U(\cdot, s) - \frac{\sigma}{2} \sin \big( 2 \Phi_{\varepsilon,\sigma}(\cdot, s) \big) \Big) - \varepsilon^2 \sin((t - s) D) D R_\varepsilon^\Phi(\cdot, s) \bigg) \, ds,\\
\varphi_{\varepsilon, \sigma}(\cdot, t) & = \cos(t D) \varphi_{\varepsilon, \sigma}^0 + \frac{\sin(t D)}{D} v_{\varepsilon, \sigma}^0\\
+ & \int_0^t \bigg( \frac{\sin((t - s) D)}{D} \Big( \varepsilon^2 R_\varepsilon^U(\cdot, s) - \frac{\sigma}{2} \sin \big(2 \Phi_{\varepsilon, \sigma}(\cdot, s) \big) \Big) + \varepsilon^2 \cos((t - s) D) R_\varepsilon^\Phi(\cdot, s) \bigg) \, ds,
\end{align*}
for any $t \in [0, T]$. Therefore, we are led to
\begin{equation}
 \label{est:H-m}
\begin{split}
\| v_{\varepsilon, \sigma}(\cdot, t) \|_{H^{m - 1}} & + \| \varphi_{\varepsilon, \sigma}(\cdot, t) \|_{H^m} \leq C \big( 1 + t^2 \big) \, \Big( \| v_{\varepsilon, \sigma}^0 \|_{H^{m - 1}} + \| \varphi_{\varepsilon, \sigma}^0 \|_{H^m}\\
& + \max_{s \in [0, T]} \big( \sigma \| \sin(2\Phi_{\varepsilon, \sigma}(\cdot, s)) \|_{H^{m - 1}} + \varepsilon^2 \| R_\varepsilon^U(\cdot, s) \|_{H^{m - 1}} + \varepsilon^2 \| R_\varepsilon^\Phi(\cdot, s) \|_{H^m} \big) \Big).
\end{split}
\end{equation}
Thus, the estimate~\eqref{est:FW1} follows from~\eqref{est:sin},~\eqref{est:sin2},~\eqref{est:R-phi} and~\eqref{est-R}.

For $1 \leq \ell \leq m$, the estimate in the homogeneous Sobolev spaces is similar to~\eqref{est:H-m}. We only replace the norms $H^m$ and $H^{m - 1}$ by $\dot{H}^\ell$ and $\dot{H}^{\ell - 1}$, and the term $1 + t^2$ by $1 + t$ in~\eqref{est:H-m}. Then, the estimate~\eqref{est:FW2} follows as before using~\eqref{est:sin2} instead of~\eqref{est:sin}.
\end{proof}

%%%%%%%%%%%%%%%%%%%%%%%%%%%%%%%%%%%%%%%%%%%%%%%%%%%%%%%%%%%%%%%%%%%%%%
%%%%%%%%%%%%%%%%%%%%%%%%%%%%%%%%%%%%%%%%%%%%%%%%%%%%%%%%%%%%%%%%%%%%%%
%%%%%%%%%%%%%%%%%%%%%%%%%%%%%%%%%%%%%%%%%%%%%%%%%%%%%%%%%%%%%%%%%%%%%%
\appendix
\section{Properties of the sets \texorpdfstring{$H_{\sin}^k(\R^N)$}{}}
\label{sec:Hsink}
%%%%%%%%%%%%%%%%%%%%%%%%%%%%%%%%%%%%%%%%%%%%%%%%%%%%%%%%%%%%%%%%%%%%%%
%%%%%%%%%%%%%%%%%%%%%%%%%%%%%%%%%%%%%%%%%%%%%%%%%%%%%%%%%%%%%%%%%%%%%%
%%%%%%%%%%%%%%%%%%%%%%%%%%%%%%%%%%%%%%%%%%%%%%%%%%%%%%%%%%%%%%%%%%%%%%

In this first appendix, we collect some useful properties of the sets $H_{\sin}^k(\R^N)$. In particular, we underline the reasons why the Cauchy problem for the Sine-Gordon equation in the product set $H_{\sin}^k(\R^N) \times H^{k - 1}(\R^N)$ cannot be immediately reduced to the usual Sobolev framework.

Let $k \in \N^*$ be fixed. Recall first that the set $H_{\sin}^k(\R^N)$ is not a vector space. Indeed, the constant function $\pi$ belongs to this set, but not the function $\pi/2$. On the other hand, it is an additive group due to the trigonometric identities
$$\sin(- \phi) = - \sin(\phi), \quad {\rm and} \quad \sin(\phi_1 + \phi_2) = \sin(\phi_1) \cos(\phi_2) + \sin(\phi_2) \cos(\phi_1).$$
Since any function in the space $H^k(\R^N)$ belongs to $H_{\sin}^k(\R^N)$, we infer that
\begin{equation}
\label{djokovic}
H^k(\R^N) + H_{\sin}^k(\R^N) \subset H_{\sin}^k(\R^N).
\end{equation}

Concerning the topological structure of the set $H_{\sin}^k(\R^N)$, we identify this set with the quotient group $H_{\sin}^k(\R^N)/\pi \Z$, and we endow it with the metric structure provided by the distance $d_{\sin}^k$ in~\eqref{def:dsink}. 

In many places, we do not work with the distance $d_{\sin}^k$, but instead with the quantity
$$\| \phi \|_{H_{\sin}^k} := d_{\sin}^k(\phi, 0) = \Big( \| \sin(\phi) \|_{L^2}^2 + \| \nabla \phi \|_{H^{k - 1}}^2 \Big)^\frac{1}{2}.$$
This is an abuse of notation since this quantity is not a norm. However, we have the classical identity
$$d_{\sin}^k(\phi_1, \phi_2) = \| \phi_1 - \phi_2 \|_{H_{\sin}^k}.$$
so that the quantity $ \| \phi \|_{H_{\sin}^k}$ satisfies the triangle inequality. Note also the useful estimate
\begin{equation}
\label{murray}
\| \phi \|_{H_{\sin}^k} \leq \| \phi \|_{H^k},
\end{equation}
when $\phi \in H^k(\R^N)$.

Coming back to~\eqref{djokovic}, we provide a decomposition of any function $\phi \in H_{\sin}^k(\R^N)$ as a sum $\phi = f + \varphi$, with $\varphi \in H^k(\R^N)$ and $f \in H_{\sin}^\infty(\R^N)$.

\begin{lem}
\label{lem:decompose}
Given any function $\phi \in H_{\sin}^k(\R^N)$, there exist two functions $f \in H_{\sin}^\infty(\R^N)$ and $\varphi \in H^k(\R^N)$ such that $\phi = f + \varphi$,
$$\| \varphi \|_{H^k} \leq \sqrt{2} \| \nabla \phi \|_{H^{k - 1}},$$
and
$$\| f \|_{H_{\sin}^\ell} \leq A \| \phi \|_{H_{\sin}^k},$$
for any $\ell \geq 1$. The positive number $A$ in this inequality only depends on $k$ and $\ell$.
\end{lem}

\begin{proof}
Consider a function $\chi \in \boC_c^\infty(\R^N)$, with $\supp{\widehat{\chi}} \in B(0, 2)$, and such that
$$\widehat{\chi} = 1 \ {\rm on} \ B(0,1), \quad {\rm and} \quad 0 \leq \widehat{\chi} \leq 1.$$
Set
$$\widehat{f} = \widehat{\chi} \, \widehat{\phi}, \quad {\rm and} \quad \widehat{\varphi} = \big( 1 - \widehat{\chi} \big) \, \widehat{\phi},$$
so that $\phi = f + \varphi$. By the Plancherel theorem, the function $\varphi$ is in $H^k(\R^N)$, and it satisfies
$$\| \varphi \|_{L^2} \leq \| \nabla \phi \|_{L^2}, \quad {\rm and} \quad \| \nabla \varphi \|_{H^{k - 1}} \leq \| \nabla \phi \|_{H^{k - 1}}.$$
Concerning the function $f$, we check that
$$\| \sin(f) \|_{L^2} = \| \sin(\phi - \varphi) \|_{L^2} \leq \| \sin(\phi) \|_{L^2} + \| \varphi \|_{L^2} \leq \| \sin(\phi) \|_{L^2} + \| \nabla \phi \|_{L^2},$$
and we also compute
$$\| \nabla f \|_{H^{\ell - 1}}^2 \leq A \int_{\R^N} \big( 1 + |\xi|^2 \big)^{\ell - 1} |\widehat{\chi}(\xi)|^2 |\widehat{\nabla \phi}(\xi)|^2 \, d\xi \leq A 5^{\ell - k} \, \| \nabla \phi \|_{H^{k - 1}}^2.$$
\end{proof}

This decomposition is enough to establish the density of smooth functions.

\begin{lem}
\label{lem:density}
Given any function $\phi \in H_{\sin}^k(\R^N)$, there exist functions $\phi_n \in H_{\sin}^\infty(\R^N)$, with $\phi_n - \phi \in L^2(\R^N)$, such that 
\begin{equation}
\label{nadal}
\| \phi_n - \phi \|_{H^k} \to 0,
\end{equation}
as $n \to \infty$. In particular, we have
$$d_{\sin}^k(\phi_n, \phi) = \| \phi_n - \phi \|_{H_{\sin}^k} \to 0.$$ 
\end{lem}

\begin{proof}
Let us decompose the function $\phi$ as $\phi = f + \varphi$, with $f \in H_{\sin}^\infty(\R^N)$ and $\varphi \in H^k(\R^N)$. By standard density theorems in the Sobolev spaces, there exist smooth functions $\varphi_n \in H^\infty(\R^N)$ such that
$$\| \varphi_n - \varphi \|_{H^k} \to 0,$$
as $n \to \infty$. Setting $\phi_n = f + \varphi_n$, we have $\phi_n - \phi = \varphi_n - \varphi \in L^2(\R^N)$, so that
$$\| \phi_n - \phi \|_{H^k} = \| \varphi_n - \varphi \|_{H^k} \to 0.$$
Moreover, we derive from~\eqref{djokovic} that the functions $\phi_n$ are in $H_{\sin}^\infty(\R^N)$. The last convergence in Lemma~\ref{lem:density} follows from~\eqref{murray} and~\eqref{nadal}. 
\end{proof}

In our analysis of the Sine-Gordon regime, it is important to control uniformly the function $\phi \in H_{\sin}^k(\R^N)$, at least when the integer $k$ is large enough. In order to obtain such a control, we study the behaviour at infinity of the functions in the space $H_{\sin}^1(\R^N)$ in the spirit of the work by G\'erard~\cite{Gerard2} about the energy space of the Gross-Pitaevskii equation. 

\begin{lem}
\label{lem:carac}
Let $\phi \in H_{\sin}^1(\R^N)$.

$(i)$ For $N = 1$, the function $\phi$ is uniformly continuous and bounded on $\R$, and there exist two integers $(\ell^+, \ell^-) \in \Z^2$ such that
$$\phi(x) \to \ell^{\pm} \pi,$$
as $x \to \pm \infty$. Moreover, the differences $\phi - \ell^\pm \pi$ are in $L^2(\R_\pm)$.

$(ii)$ For $N \geq 2$, there exists an integer $\ell \in \Z$ such that
$$\phi - \ell \pi \in H^1(\R^N).$$
\end{lem}

\begin{proof}
When $N = 1$, the Sobolev embedding theorem implies the uniform continuity of $\phi$, and then of $\sin^2(\phi)$. Since this function is integrable on $\R$, it converges to $0$ as $x \to \pm\infty$. As a consequence of the continuity of the function $\phi$, there exist two integers $(\ell^+, \ell^-) \in \Z^2$ such that we have the convergences in $(i)$. In particular, the function $\phi$ is bounded. Moreover, we have the pointwise inequality $|\phi(x) - \ell^+ \pi| \leq \pi/2$ for $x$ large enough, so that the Jordan inequality gives
$$\frac{2}{\pi} |\phi(x) - \ell^+ \pi| \leq |\sin(\phi(x))|.$$
This is enough to prove that the functions $\phi - \ell^+ \pi$, and similarly $\phi - \ell^- \pi$, are in $L^2(\R_\pm)$.

The proof of $(ii)$ is similar. Let us consider two functions $f \in H_{\sin}^\infty(\R^N)$ and $\varphi \in H^1(\R^N)$ such that $\phi = f + \varphi$. By the Sobolev embedding theorem, the functions $f$ and $\sin^2(f)$ are uniformly continuous. The existence of an integer $\ell$ such that
$$f(x) \to \ell \pi,$$
as $|x| \to \infty$, follows as before. The property that $f - \ell \pi$ is square integrable results again from the Jordan inequality. In view of the decomposition for $\phi$, this is enough to guarantee that $\phi - \ell \pi$ lies in $H^1(\R^N)$.
\end{proof} 

\begin{rem}
When $N \geq 2$, the quotient group $H_{\sin}^1(\R^N)/\pi \Z$ reduces to the Sobolev space $H^1(\R^N)$. In view of~\eqref{murray}, the $H^1$-norm controls the quantity $\| \cdot \|_{H_{\rm sin}^1}$, but the opposite is false. Given a function $\chi \in \boC^\infty(\R^N)$ such that
$$\chi = 1 \ {\rm on} \ B(0,1), \quad {\rm and} \quad \chi = 0 \ {\rm outside} \ B(0, 2),$$
we can set
$$\phi_n(x) = n \chi \Big( \frac{x}{n} \Big),$$
and check the existence of a positive number $A$, depending only on $N$, such that
$$\| \phi_n \|_{L^2} \geq A n \| \phi_n \|_{H_{\sin}^1},$$
for any integer $n \geq 1$.
\end{rem}

When $N \geq 2$, we recover the uniform continuity and the boundedness of a function $\phi \in H_{\sin}^k(\R^N)$ assuming that the integer $k$ is large enough.

\begin{cor}
\label{cor:carac}
Let $N \geq 2$ and $k > N/2$. The functions $\phi \in H_{\sin}^k(\R^N)$ are uniformly continuous and bounded, and there exists an integer $\ell^\infty \in \Z$ such that
\begin{equation}
\label{phi-lim-unif}
\phi(x) \to \ell^\infty \pi,
\end{equation}
as $|x| \to \infty$. When $N \geq 3$, there exists a positive number $A$, depending only on $k$ and $N$, such that
\begin{equation}
\label{H-sin-bounded}
\| \phi - \ell^\infty \pi \|_{L^\infty} \leq A \| \nabla \phi \|_{H^{k - 1}}.
\end{equation}
\end{cor}

\begin{proof}
Lemma~\ref{lem:carac} provides an integer $\ell^\infty \in \Z$ such that the function $\phi - \ell^\infty \pi$ is in $H^k(\R^N)$. The uniform continuity and boundedness of $\phi$ then results from the Sobolev embedding theorem, as well as the limit in~\eqref{phi-lim-unif}.

Estimate~\eqref{H-sin-bounded} is a consequence of the Sobolev and Morrey inequalities. Set $q = 2 N$ if $k \geq N/2 + 1$, and $q = q_k$ otherwise, where the number $q_k$ is defined by the identity $1/q_k = 1/2 - (k - 1)/N$. There exists a positive number $A$, depending only on $k$ and $N$, such that
$$|\phi(x) - \phi(y)| \leq A |x - y|^{1 - \frac{N}{q}} \| \nabla \phi \|_{L^q} \leq A |x - y|^{1 - \frac{N}{q}} \| \nabla \phi \|_{H^{k - 1}},$$
for any $(x, y) \in \R^{2 N}$. On the other hand, it follows from the Sobolev embedding theorem that
$$\| \phi - \ell^\infty \pi \|_{L^{\frac{2 N}{N - 2}}} \leq A \| \nabla \phi \|_{L^2}.$$
Combining these inequalities, we deduce from the H\"older inequality that
\begin{align*}
|\phi(x) - \ell^\infty \pi| \leq \frac{1}{|B(x, 1)|} \bigg( \int_{B(x,1)} |\phi(x) - & \phi(y)| \, dy + \int_{B(x, 1)} |\phi(y) - \ell^\infty \pi| \, dy \bigg)\\
\leq & A \Big( \| \nabla \phi \|_{H^{k - 1}} + \| \phi - \ell^\infty \pi \|_{L^\frac{2 N}{N - 2}} \Big) \leq A \| \nabla \phi \|_{H^{k - 1}}.
\end{align*}
This concludes the proof of Corollary~\ref{cor:carac}.
\end{proof}

\begin{rem}
\label{rem:carac1}
For $N = 1$, though the functions $\phi \in H^1_{\sin}(\R)$ are bounded and have limits $\ell^\pm \pi$ at $\pm \infty$, there are no positive numbers $A_\pm$ such that
\begin{equation}
\label{H-sin-bounded-1} 
\| \phi - \ell^\pm \pi \|_{L^\infty(\R_\pm)} \leq A_\pm \| \phi \|_{H_{\sin}^1(\R_\pm)},
\end{equation}
for any $\phi \in H_{\sin}^1(\R)$. For instance, the even functions $\phi_n$ given by
$$\phi_n(x) = \begin{cases} n \pi & {\rm if} \ 0 \leq x \leq n \pi,\\
2 n \pi - x & {\rm if} \ n \pi \leq x \leq 2 n \pi,\\
0 & {\rm if} \ x \geq 2 n \pi, \end{cases}$$
satisfy
$$\| \phi_n - \ell_n^\pm \pi\|_{L^\infty(\R_\pm)} = n \pi, \quad {\rm and} \quad \| \phi_n \|_{H_{\sin}^1(\R_\pm)} = \sqrt{\frac{3 n \pi}{2}},$$
with $\ell_n^\pm = 0$. Inequality~\eqref{H-sin-bounded-1} cannot hold for $n$ large enough.
\end{rem}

\begin{rem}
\label{rem:carac2}
When $N = 2$ and $k > N/2$, the Sobolev embedding theorem provides the existence of a positive number $A$, depending only on $k$, such that
$$\| \phi - \ell^\infty \pi \|_{L^\infty} \leq A \| \phi - \ell^\infty \pi \|_{H^k}.$$
On the other hand, there is no positive number $A$ such that
\begin{equation}
\label{H-sin-bounded-2} 
\| \phi - \ell^\infty \pi \|_{L^\infty} \leq A \| \phi \|_{H_{\sin}^k}.
\end{equation}
Indeed, let us consider the functions $v_n$ defined by
$$v_n(r) = \begin{cases} n \pi & {\rm if} \ 0 \leq r \leq r_n := \frac{n \pi}{\ln(n)^2},\\
n \pi \Big( \frac{\ln(n \pi) - \ln(r)}{\ln(\ln(n))} - 1 \Big) & {\rm if} \ r_n \leq r \leq s_n := \frac{n \pi}{\ln(n)},\\
0 & {\rm if} \ r \geq s_n, \end{cases}$$
for any integer $ n \geq 3$. Given a non-negative and non-increasing function $\chi \in \boC^\infty(\R)$ such that
$$\chi = 1 \ {\rm on} \ (- \infty, - 1], \quad {\rm and} \quad \chi = 0 \ {\rm on} \ [1, \infty),$$
we set 
$$\phi_n(x) = n \pi \big( 1 - \chi(r_n + 2 - |x|) \big) + v_n(|x|) \chi(r_n + 2 - |x|) \chi \big( |x| - s_n + 2 \big),$$
for any $x \in \R^2$. The functions $\phi_n$ are smooth and compactly supported, so that they belong to the space $H_{\sin}^\infty(\R^2)$, with limits at infinity $\ell_n^\infty \pi = 0$. On the other hand, we check that
$$\| \phi_n \|_{L^\infty} = n \pi, \quad \| \sin(\phi_n) \|_{L^2} = \boO \bigg( \frac{n}{\sqrt{\ln(n)}} \bigg) \quad {\rm and} \quad \| \nabla \phi_n \|_{L^2} = \boO \bigg( \frac{n}{\sqrt{\ln(\ln(n))}} \bigg),$$
for $n \to \infty$, while
$$\| D^k \phi_n \|_{L^2} = \boO \bigg( \frac{\sqrt{n} \ln(n)}{\ln(\ln(n))} \bigg),$$
for any $k \geq 2$. This contradicts the existence of a positive number $A$ such that~\eqref{H-sin-bounded-2} holds.
\end{rem}

For $N = 1$, the limits at infinity $\ell^{\pm} \pi$ of a function $\phi \in H_{\rm sin}^1(\R)$ can be different. This set does not reduce to a collection of constant translations of the space $H^1(\R)$. When $N = 2$ and $k \geq 2$, the uniform norm of a function $\phi \in H_{\rm sin}^k(\R^2)$ is not controlled by the quantity $\| \phi \|_{H_{\sin}^k}$. At least in these two cases, bringing the Cauchy problem for the Sine-Gordon equation in $H_{\rm sin}^k(\R^N) \times H^{k - 1}(\R^N)$ back to a standard Sobolev setting is not immediate.

%%%%%%%%%%%%%%%%%%%%%%%%%%%%%%%%%%%%%%%%%%%%%%%%%%%%%%%%%%
%%%%%%%%%%%%%%%%%%%%%%%%%%%%%%%%%%%%%%%%%%%%%%%%%%%%%%%%%%
%%%%%%%%%%%%%%%%%%%%%%%%%%%%%%%%%%%%%%%%%%%%%%%%%%%%%%%%%%
\section{Tame estimates and composition in Sobolev spaces}
\label{sec:tame-estimates}
%%%%%%%%%%%%%%%%%%%%%%%%%%%%%%%%%%%%%%%%%%%%%%%%%%%%%%%%%%
%%%%%%%%%%%%%%%%%%%%%%%%%%%%%%%%%%%%%%%%%%%%%%%%%%%%%%%%%%
%%%%%%%%%%%%%%%%%%%%%%%%%%%%%%%%%%%%%%%%%%%%%%%%%%%%%%%%%%

For $m \in \N$, we denote by $\dot{H}^m(\R^N)$ the homogeneous Sobolev space endowed with the semi-norm
$$\| f \|_{\dot{H}^m} := \sum_{|\alpha| = m} \| \partial^\alpha f \|_{L^2}.$$
We recall the following Moser estimates (see e.g.~\cite{Moser1, Hormand0, BetDaSm1}).

\begin{lem}
\label{lem:moser}
$(i)$ Let $(f, g) \in L^\infty(\R^N)^2 \cap \dot{H}^m(\R^N)^2$. The product $f g$ is in $\dot{H}^m(\R^N)$, and there exists a positive number $C_m$, depending only on $m$, such that
$$\|f \, g \|_{\dot{H}^m} \leq C_m \max \big\{ \| f \|_{L^\infty} \, \| g \|_{\dot{H}^m}, \| f \|_{\dot{H}^m} \, \| g \|_{L^\infty} \big\} \leq C_m \big( \| f \|_{L^\infty} \, \| g \|_{\dot{H}^m} + \| f \|_{\dot{H}^m} \, \| g \|_{L^\infty}\big).$$
More generally, given any functions $(f_1, \dots ,f_j) \in L^\infty(\R^N)^j \cap \dot{H}^m(\R^N)^j$, there exists a positive number $C_{j, m}$, depending only on $j$ and $m$, such that
$$\| \partial^{\alpha_1} f_1 \cdots \partial^{\alpha_j} f_j \|_{L^2} \leq C_{j, m} \max_{1 \leq i \leq j} \prod_{\ell \neq i} \| f_\ell \|_{L^\infty} \, \| f_i \|_{\dot{H}^m},$$
for any $\alpha = (\alpha_1, \dots, \alpha_j) \in \N^N$ such that $\sum_{i = 1}^j |\alpha_i| = m$. 

$(ii)$ Let $m \in \N^*$. When $f \in L^\infty(\R^N) \cap \dot{H}^m(\R^N)$ and $F \in \boC^m(\R^N)$, the composition function $F(f)$ is in $\dot{H}^m(\R^N)$, and there exists a positive number $C_m$, depending only on $m$, such that 
\begin{equation}
\label{eq:compose-moser}
\| F(f) \|_{\dot{H}^m} \leq C_m \max_{1 \leq \ell \leq m} \| F^{(\ell)} \|_{L^\infty} \, \| f \|_{L^\infty}^{\ell - 1} \, \| f \|_{\dot{H}^m}.
\end{equation}
\end{lem}

As a direct consequence of Lemma~\ref{lem:moser}, we obtain the following useful estimates.

\begin{cor}
\label{cor:moser}
Let $m \in \N$, with $m \geq 2$, and $(\alpha, \beta) \in \N^{2 N}$, with $|\alpha| = m$ and $\beta \leq \alpha$. There exists a positive number $C_m$, depending only on $m$, such that we have the following estimates.

$(i)$ If $|\beta| \geq 1$, $\nabla f \in L^\infty(\R^N) \cap \dot{H}^{m - 1}(\R^N)$ and $g \in L^\infty(\R^N) \cap \dot{H}^{m - 1}(\R^N)$, then
$$\| D^\beta f \, D^{\alpha - \beta} g \|_{L^2} \leq C_m \max \big\{ \| \nabla f \|_{L^\infty} \, \| g \|_{\dot{H}^{m - 1}}, \| f \|_{\dot{H}^m} \, \| g \|_{L^\infty} \big\}.$$

$(ii)$ If $|\beta| \geq 2$, $D^2 f \in L^\infty(\R^N) \cap \dot{H}^{m - 2}(\R^N)$ and $g \in L^\infty(\R^N) \cap \dot{H}^{m - 2}(\R^N)$, then
$$\| D^\beta f \, D^{\alpha - \beta} g \|_{L^2} \leq C_m \max \big\{ \| D^2 f \|_{L^\infty} \, \| g \|_{\dot{H}^{m - 2}}, \| f \|_{\dot{H}^m} \, \| g \|_{L^\infty} \big\}.$$

$(iii)$ If $1 \leq |\beta| \leq m -1$, $\nabla f \in L^\infty(\R^N) \cap \dot{H}^{m - 1}(\R^N)$ and $\nabla g \in L^\infty(\R^N) \cap \dot{H}^{m - 1}(\R^N)$, then
$$\| D^\beta f \, D^{\alpha - \beta} g \|_{L^2} \leq C_m \max \big\{ \| \nabla f \|_{L^\infty} \, \| g \|_{\dot{H}^{m - 1}}, \| f \|_{\dot{H}^{m - 1}} \, \| \nabla g \|_{L^\infty} \big\}.$$
\end{cor}

Our previous analysis of the Sine-Gordon regime of the Landau-Lifshitz equation requires bounds on the functions $\sin(f)$ and $\cos(f)$, which do not depend on the uniform bound of the function $f$. We indeed do not control this quantity in the Hamiltonian framework under our consideration.

In this direction, we derive from Lemma~\ref{lem:moser} the following bounds, which only depend on the uniform norm of the gradient $\nabla f$.

\begin{cor}
\label{cor:moser-trigo}
Let $m \in \N$, with $m \geq 2$. There exists a positive number $C_m$, depending only on $m$, such that 
\begin{equation}
\label{sin-H-k}
\| \sin(f) \|_{\dot{H}^m}+\| \cos(f) \|_{\dot{H}^m}\leq C_m (1+\| \nabla f \|_{L^\infty}^{m-1}) \| \nabla f \|_{H^{m-1}},
\end{equation}
for any function $f:\R^N \to \R$, with $\nabla f \in L^\infty(\R^N) \cap \dot{H}^{m - 1}(\R^N)$.
\end{cor}

\begin{proof}
The proof is by induction. For $m = 2$, we directly check that
$$\| \sin(f) \|_{\dot{H}^2} + \| \cos(f) \|_{\dot{H}^2} \leq 2 \big( \| \nabla f \|_{L^2} \| \nabla f \|_{L^\infty} + \| D^2 f \|_{L^2} \big).$$
Let us assume that~\eqref{sin-H-k} holds for any $2 \leq \ell \leq m$. It results from Lemma~\ref{lem:moser} and the inductive assumption that
\begin{align*}
\| \sin(f) \|_{\dot{H}^{m + 1}} = \| \cos(f) \, \nabla f \|_{\dot{H}^m} \leq & C_m \big( \| \cos(f) \|_{L^\infty} \, \| \nabla f \|_{\dot{H}^m} + \| \cos(f) \|_{\dot{H}^m} \, \| \nabla f \|_{L^\infty} \big)\\
\leq & C_m \big( \| f \|_{\dot{H}^{m + 1}} + (1 + \| \nabla f \|_{L^\infty}^{m - 1}) \, \| \nabla f \|_{H^{m - 1}} \, \| \nabla f \|_{L^\infty} \big)\\
\leq & C_m \big( 1 + \| \nabla f \|_{L^\infty}^m \big) \, \| \nabla f \|_{H^m}.
 \end{align*}
The proof for $\cos(f)$ follows in the same manner.
\end{proof}

In some parts of our proofs, we need to avoid the polynomial growth on the norm $\| \nabla f \|_{L^\infty}$ in Corollary~\ref{cor:moser-trigo}. For this reason, we establish the next bounds with an at most linear growth in terms of the norm $\| \nabla f \|_{L^\infty}$.

\begin{cor}
\label{cor:cos}
Let $m \geq 2$. There exists a positive number $C_m$, depending only on $m$, such that
\begin{equation}
\label{ineq1}
\| \cos(f) \|_{\dot{H}^m} \leq C_m \big( \| \sin(f) \|_{L^\infty} + \| \nabla f \|_{L^\infty} \big) \max \big\{ \| \sin(f) \|_{\dot{H}^{m - 1}}, \| f \|_{\dot{H}^m} \big\},
\end{equation}
for any function $f: \R^N \to \R$ such that $\sin(f) \in H^{m - 1}(\R^N)$, $\nabla f \in L^\infty(\R^N)$ and $\nabla f \in H^{m - 1}(\R^N)$.
\end{cor}

\begin{proof}
Lemma~\ref{lem:moser} indeed provides
$$\| \cos(f) \|_{\dot{H}^m} = \| \sin(f) \, \nabla f \|_{\dot{H}^{m - 1}} \leq C_m \max \big\{ \| \sin(f) \|_{L^\infty} \, \| f \|_{\dot{H}^m}, \| \nabla f \|_{L^\infty} \, \| \sin(f) \|_{\dot{H}^{m - 1}} \big\},$$
which yields~\eqref{ineq1}.
\end{proof}

%%%%%%%%%%%%%%%%%%%%%%%%%%%%%%%%%%%%%%%%%%%%%%%%%%
%%%%%%%%%%%%%%%%%%%%%%%%%%%%%%%%%%%%%%%%%%%%%%%%%%
%%%%%%%%%%%%%%%%%%%%%%%%%%%%%%%%%%%%%%%%%%%%%%%%%%
\section{Solitons of the Landau-Lifshitz equation}
\label{sec:solitons}
%%%%%%%%%%%%%%%%%%%%%%%%%%%%%%%%%%%%%%%%%%%%%%%%%%
%%%%%%%%%%%%%%%%%%%%%%%%%%%%%%%%%%%%%%%%%%%%%%%%%%
%%%%%%%%%%%%%%%%%%%%%%%%%%%%%%%%%%%%%%%%%%%%%%%%%%

Solitons are special solutions to the one-dimensional Landau-Lifshitz equation, which take the form
$$m(x, t) = m_c(x - c t),$$
for a given speed $c \in \R$. The profile $m_c$ is solution to the ordinary differential equation
$$- c m_c' + m_c \times \big( m_c'' - \lambda_1 [m_c]_1 e_1 - \lambda_3 [m_c]_3 e_3 \big) = 0.$$
In the case of biaxial ferromagnets, we can assume, without loss of generality, that $\lambda_3 > \lambda_1 > 0$. Set $c^* := \lambda_3^{1/2} - \lambda_1^{1/2}$. For $|c| \leq c^*$, non-constant solitons $m_c$ are explicitly given by the formulae
$$m_c^\pm(x) = \bigg( \frac{a_c^\pm}{\cosh(\mu_c^\pm x)}, \, \tanh(\mu_c^\pm x), \, \frac{(1 - (a_c^\pm)^2)^\frac{1}{2}}{\cosh(\mu_c^\pm x)} \bigg),$$
up to the geometric invariances of the equation, which are the translations and the orthogonal symmetries with respect to the lines $\R e_1$, $\R e_2$ and $\R e_3$. In this formula, the values of $a_c^\pm$ and $\mu_c^\pm$ are equal to
$$a_c^\pm := \delta_c \bigg( \frac{c^2 + \lambda_3 - \lambda_1 \mp \big( (\lambda_3 + \lambda_1 - c^2)^2 - 4 \lambda_1 \lambda_3 \big)^\frac{1}{2}}{2 (\lambda_3 - \lambda_1)} \bigg)^\frac{1}{2},$$
and
$$\mu_c^\pm = \bigg( \frac{\lambda_3 + \lambda_1 - c^2 \pm \big( (\lambda_3 + \lambda_1 - c^2)^2 - 4 \lambda_1 \lambda_3 \big)^\frac{1}{2}}{2} \bigg)^\frac{1}{2},$$
with $\delta_c = 1$, if $c \geq 0$, and $\delta_c = - 1$, when $c < 0$. Note that 
$$(\mu_c^\pm)^2 = \lambda_1 (a_c^\pm)^2 + \lambda_3 (1 - (a_c^\pm)^2).$$

The Landau-Lifshitz energy of the solitons $m_c^\pm$ is equal to
$$E_{\rm LL}(m_c^\pm) = 2 \mu_c^\pm.$$ 
The solitons form two branches in the plane $(c, E_{\rm LL})$.

\begin{figure}[ht!]
\begin{center}
\includegraphics{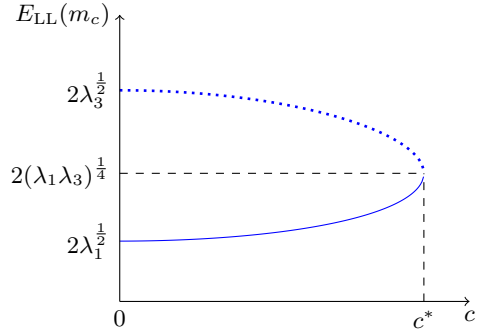}
\end{center}
\caption{The curves $E_{\rm LL}(m_c^+)$ and $E_{\rm LL}(m_c^-)$ in dotted and solid lines, respectively.}
\label{fig-E}
\end{figure}

The lower branch corresponds to the solitons $m_c^-$, and the upper one to the solitons $m_c^+$ as depicted in Figure~\ref{fig-E}. The lower branch is strictly increasing and convex with respect to $c \in [0, c^*]$,
with
$$E(m_0^-) = 2 \lambda_1^\frac{1}{2} \quad {\rm and} \quad E(m_{c^*}^-) = 2 (\lambda_1 \lambda_3)^\frac{1}{4}.$$
The upper branch is a strictly decreasing and concave function of $c \in [0, c^*]$, with
$$E(m_0^+) = 2 \lambda_3^\frac{1}{2} \quad {\rm and} \quad E(m_{c^*}^+) = 2 (\lambda_1 \lambda_3)^\frac{1}{4}.$$
The two branches meet at the common soliton $m_{c^*}^- = m_{c^*}^+$.

In the limit $\lambda_1 \to 0$, the lower branch vanishes, while the upper branch goes to the branch of solitons for the Landau-Lifshitz equation with an easy-plane anisotropy (see e.g.~\cite{deLaire3}).

In the sequel, we focus on the solitons $m_c^-$ corresponding to the lower branch. Since $a_c^-$ is a strictly decreasing and a continuous function of $c$, with
$$a_0^- = 1 \quad {\rm and} \quad a_{c^*}^- = \Big( \frac{\sqrt{\lambda_3}}{\sqrt{\lambda_1} + \sqrt{\lambda_3}} \Big)^\frac{1}{2},$$
the function $\check{m}_c^- := [m_c^-]_1 + i [m_c^-]_2$ may always be lifted as
$$\check{m}_c^- = \big( 1 - [m_c^-]_3^2 \big)^\frac{1}{2} \, \big( \sin(\varphi_c^-) + i \cos(\varphi_c^-) \big),$$
with 
$$\varphi_c^-(x) = 2 \arctan \bigg( \frac{\big( (a_c^-)^2 + \sinh^2(\mu_c^- x) \big)^\frac{1}{2} - \sinh(\mu_c^- x)}{a_c^-} \bigg).$$
When $c = 0$, we observe that
\begin{equation}
\label{phase-c-0}
\varphi_0^-(x) = 2 \arctan \Big( e^{- \lambda_1^\frac{1}{2} x} \Big).
\end{equation}

Let us now fix a number $0 < \varepsilon < 1$, and set $\lambda_1 = \nu \varepsilon$ and $\lambda_3 = 1/\varepsilon$, so that $c^* = 1/\varepsilon^{1/2} - (\nu \varepsilon)^{1/2}$. Given a number $0 \leq \tau < 1$, we are allowed to consider the solitons $m_{c_\varepsilon}^-$ with speeds $c_\varepsilon = \tau/\varepsilon^{1/2}$ when $\varepsilon$ is small enough. In this regime, the parameters $a_{c_\varepsilon}^-$ and $\mu_{c_\varepsilon}^-$ satisfy
$$a_{c_\varepsilon^-} = 1 - \frac{\nu \tau^2 \varepsilon^2}{2 (1 - \tau^2)} - \frac{\nu^2 \tau^2 (8 - 7 \tau^2 + 3 \tau^4) \varepsilon^4}{8 (1 - \tau^2)^3} + \boO \big( \varepsilon^6 \big),$$
and
$$\mu_{c_\varepsilon}^- = \Big( \frac{\nu \varepsilon}{1 - \tau^2} \Big)^\frac{1}{2} \, \Big( 1 - \frac{\nu \tau^2 \varepsilon^2}{(1 - \tau^2)^2} + \boO \big( \varepsilon^4 \big) \Big),$$
when $\varepsilon \to 0$. Here as in the sequel, the notation $\boO \big( \varepsilon^k \big)$ stands for a quantity, which is bounded by $C \varepsilon^k$, where the positive number $C$ only depends on $\nu$ and $\tau$. Coming back to the scaling performed so as to obtain~\eqref{HLLeps}, we compute
$$U_{c_\varepsilon}(x, t) := \frac{1}{\varepsilon} \Big[ m_{c_\varepsilon}^- \Big( \frac{x - \tau t}{\varepsilon^\frac{1}{2}} \Big) \Big]_3 = \frac{\nu^\frac{1}{2} \tau + \boO(\varepsilon^2)}{(1 - \tau^2)^\frac{1}{2} \, \cosh \Big( \frac{(\nu^\frac{1}{2} + \boO(\varepsilon^2)) (x - \tau t)}{(1 - \tau^2)^\frac{1}{2}} \Big)},$$
and
$$\Phi_{c_\varepsilon}(x, t) := \varphi_{c_\varepsilon}^- \Big( \frac{x - \tau t}{\varepsilon^\frac{1}{2}} \Big) = 2 \arctan \bigg( \big( 1 + \boO(\varepsilon^2) \big) \, \Big( \exp \Big( - \frac{(\nu^\frac{1}{2} + \boO(\varepsilon^2)) (x - \tau t)}{(1 - \tau^2)^\frac{1}{2}} \Big) + \boO(\varepsilon^2) \Big) \bigg).$$
In view of~\eqref{soliton-SG}, the pairs $(U_{c_\varepsilon}, \Phi_{c_\varepsilon})$ form a family of solitons for~\eqref{HLLeps} with speed $\tau$, which converge towards the soliton $(U_\tau, \Phi_\tau)$ with speed $\tau$ of the Sine-Gordon system in~\eqref{sys:SG}. Actually, the pairs $(U_{c_\varepsilon}, \Phi_{c_\varepsilon})$ and $(U_\tau, \Phi_\tau)$ are identically equal for $\tau = 0$. When $\tau \neq 0$, they satisfy the estimates
$$\| U_{c_\varepsilon} - U_\tau \|_{H^k} + \| \sin(\Phi_{c_\varepsilon} - \Phi_\tau) \|_{L^2} + \| \Phi_{c_\varepsilon} - \Phi_\tau \|_{H^{k + 1}} \leq C_k \varepsilon^2,$$
for $\varepsilon$ small enough and any integer $k$. Here, the positive numbers $C_k$ only depend on $\nu$, $\tau$ and $k$. Moreover, we can check that
$$\frac{\| U_{c_\varepsilon} - U_\tau \|_{L^2}}{\varepsilon^2} \underset{\varepsilon \to 0}{\sim} \frac{\nu^\frac{5}{4} \tau (2 - 2 \tau^2 + \tau^4)}{2 (1 - \tau^2)^\frac{9}{4}} \bigg( \int_\R \Big( \cosh(x) + \frac{2 \tau^2}{2 - 2 \tau^2 + \tau^4} x \sinh(x) \Big)^2 \, \frac{dx}{\cosh(x)^2} \bigg)^\frac{1}{2}.$$
Since the integral in the right-hand side of this formula is positive, this equivalence proves that the estimates of order $\varepsilon^2$ in Theorem~\ref{thm:conv-SG} are sharp.

\begin{merci}
The authors acknowledge support from the project ``Schr\"odinger equations and applications'' (ANR-12-JS01-0005-01) of the Agence Nationale de la Recherche, and from the grant ``Qualitative study of nonlinear dispersive equations'' (Dispeq) of the European Research Council. A.~de~Laire was partially supported by the Labex CEMPI (ANR-11-LABX-0007-01) and the MathAmSud program.
\end{merci}

\bibliographystyle{plain}
\bibliography{Bibliogr}

\end{document}